\documentclass[11pt,oneside,reqno]{amsart}
\usepackage{mathpazo} 
\normalfont
\usepackage[T1]{fontenc}

\usepackage[utf8]{inputenc}
\usepackage{amsmath}
\usepackage{amsfonts}
\usepackage{amssymb}
\usepackage{amsthm}
\usepackage{tikz}
\usepackage{tikz-cd}
\usepackage[all]{xy}
\usepackage[margin=1.2in]{geometry}
\usepackage{amscd}
\usepackage[title]{appendix}

\DeclareMathOperator{\Hom}{Hom}
\DeclareMathOperator{\Spec}{Spec}

\DeclareMathOperator{\Frac}{Frac}

\DeclareMathOperator{\Pic}{Pic}

\newcommand{\angles}[1]{\left\langle #1 \right\rangle}

\input xy
\xyoption{all}
\theoremstyle{definition}
\newtheorem{mydef}{\textbf{Definition}}[section]
\newtheorem{myeg}[mydef]{\textbf{Example}}

\newtheorem{rmk}[mydef]{\textbf{Remark}}

\theoremstyle{plain}
\newtheorem{mythm}[mydef]{\textbf{Theorem}}

\newtheorem*{nothma}{\textbf{Theorem A}}
\newtheorem*{nothmb}{\textbf{Theorem B}}
\newtheorem*{nothmc}{\textbf{Theorem C}}
\newtheorem*{nothmd}{\textbf{Theorem D}}
\newtheorem*{nothme}{\textbf{Theorem E}}
\newtheorem*{nothmf}{\textbf{Theorem F}}

\newtheorem{mytheorem}[mydef]{\textbf{Theorem}}
\newtheorem{lem}[mydef]{\textbf{Lemma}}
\newtheorem{pro}[mydef]{\textbf{Proposition}}

\newtheorem{cor}[mydef]{\textbf{Corollary}}

\newtheorem{condition}[mydef]{\textbf{Condition}}

\newcommand{\R}{\mathbb{R}}
\newcommand{\T}{\mathbb{T}}
\newcommand{\TT}{\T}
\newcommand{\Z}{\mathbb{Z}}

\begin{document}

\title{Equivariant vector bundles on toric schemes over semirings}
\author{Jaiung Jun}
\address{Department of Mathematics, State University of New York at New Paltz, NY 12561, USA}
\email{junj@newpaltz.edu}

\author{Kalina Mincheva}
\address{Department of Mathematics, Tulane University, New Orleans, LA 70118, USA}
\email{kmincheva@tulane.edu}

\author{Jeffrey Tolliver}
\address{}
\email{jeff.tolli@gmail.com}
\makeatletter
\@namedef{subjclassname@2020}{%
	\textup{2020} Mathematics Subject Classification}
\makeatother

\thanks{}

\begin{abstract}
We introduce a notion of equivariant vector bundles on schemes over semirings. We do this by considering the functor of points of a locally free sheaf. We prove that every toric vector bundle on a toric scheme $X$ over an idempotent semifield equivariantly splits as a sum of toric line bundles. We then study the equivariant Picard group $\Pic_G(X)$. Finally, we prove a version of Klyachko's classification theorem for toric vector bundles over an idempotent semifield.
\end{abstract}

\maketitle

\tableofcontents

\section{Introduction}


In this paper we take the approach of functor of points to study equivariant vector bundles on schemes over semirings and monoids. The functor of points approach of schemes says that one may view schemes as locally representable sheaves on the category of commutative rings. In fact, T\"oen and Vaqui\'e \cite{toen2009dessous} developed scheme theory over a closed symmetric monoidal category $\mathcal{C}$ by defining schemes  as (suitably defined) locally representable sheaves on $\mathcal{C}$. The construction is entirely formal and category-theoretic, and hence one may immediately apply the construction of  T\"oen and Vaqui\'e to various closed symmetric monoidal categories $\mathcal{C}$ to have scheme theory over $\mathcal{C}$ without much efforts. For instance,  when $\mathcal{C}$ is the category of $\mathbb{N}$-algebras (i.e., semirings), the theory of T\"oen and Vaqui\'e immediately gives us schemes over $\mathbb{N}$ as it does for schemes over $\mathbb{Z}$. Likewise, T\"oen and Vaqui\'e's theory applied to the category of (pointed) monoids gives one monoid schemes, which is equivalent to monoid schemes constructed via prime ideals first introduced by Deitmar \cite{Deitmar}\footnote{See \cite{marty2012relative} for the equivalence between \cite{toen2009dessous} and \cite{Deitmar} for monoid schemes.} 

While the theory proposed by T\"oen and Vaqui\'e \cite{toen2009dessous} is intriguing, it offers an opportunity for further enrichment through the incorporation of concrete examples. In this regard, the exploration of monoids and semirings presents an ideal place for not only testing but also enhancing the applicability and clarity of this theory. 

We will take this approach in the special case of the category of monoids and semirings in order to study vector bundles.
Scheme theory over monoids is very closely related to toric varieties. For instance, in \cite{cortinas2015toric}, Corti\~nas, Haesemeyer, Walker, and Weibel utilized the theory of monoid schemes to study the algebraic K-theory of blow-up squares of toric varieties and schemes. Semirings naturally appear in various ares of mathematics, but in this paper our focus will be on algebraic geometry over semirings. 

To develop scheme theory over semirings, one may consider two types of semirings based on their relations to rings. 

Firstly, there are semirings which are subsets of rings, such as $\mathbb{N}$ or $\mathbb{R}_{>0}$. One potential application of developing scheme theory over these semirings would be to obtain positive models for schemes over rings. This can be seen as an analogy to the fact that schemes over $\mathbb{Z}$ provide integral models of schemes over $\mathbb{Q}$. For instance, in \cite{BJ24}, Borger and the first author show that the narrow class group of a number field can be recovered as a reflexive Picard group of its subsemiring of totally nonnegative algebraic integers.\footnote{See the introduction of \cite{BJ24} for more detailed explanation concerning positive models of schemes.} 
Similarly, one can develop a theory of algebraic groups over semifields to capture positivity, such as in \cite{bao2021flag} and \cite{lusztig2018positive}.

The second type is (additively) idempotent semirings, which can never be embedded in rings. A class of such semirings arises naturally from totally ordered abelian groups $G$ as follows: let $S=G\cup \{-\infty\}$. The multiplication of $S$ is the addition of $G$ and the addition of $S$ is defined by $a+b = \max\{a,b\}$ with $-\infty$ the smallest element. When $G=\mathbb{R}$, the associated semiring is called the \emph{tropical semifield} denoted by $\mathbb{T}$.

Tropical geometry is a version of algebraic geometry over the tropical semifield (or its variants, described for instance in \cite{lorscheid2023unifying}). Tropical geometry brings a new set of combinatorial tools to approach classical problems in algebraic geometry.

Although the two types of semirings and semiring geometries described above may seem intrinsically different, they complement each other. In algebraic geometry, we often study a scheme $X$ over $\mathbb{Z}$ through its base changes to $X_{\mathbb{F}_q}$ or $X_\mathbb{C}$. Similarly here, one may consider a scheme $X$ over $\mathbb{N}$ and its base change to $\mathbb{T}$, which we denote by $X_\mathbb{T}$. In some sense, for a scheme $X$ over $\mathbb{N}$, the ``complement'' of $X_\mathbb{T}$ is $X_\mathbb{Z}$ as a semiring $R$ is a ring if and only if $R\otimes_{\mathbb{N}}\mathbb{T}=\{0\}$. See \cite{borger2016witt} and \cite{borger2016boolean}.

In this paper we focus our attention to schemes over additively idempotent semifields. We explore several properties of equivariant vector bundles on schemes over idempotent semifields, especially torus-equivariant vector bundles on the toric scheme $X_R$ associated to a fan, where $R$ is a semiring. Loosely speaking, $X_R$ is a semiring scheme over $R$ obtained from the monoid scheme associated to a fan via base change to $R$. For a rigorous definition and examples we refer the reader to Section~\ref{section: preliminaries}.

Recently there has been a lot of interest in extending tropical methods to the study of vector bundles and thus developing the notion of a tropical vector bundles. In \cite{JMT20} the authors of this paper develop a theory of tropical vector bundles on tropical schemes, defined as locally free sheaves of finite rank. There we observe that vector bundles globally split as sum of line bundles. In \cite{GUZ22} the authors take a different approach and construct tropical vector bundles for metric graphs (tropical curves). The definition of vector bundles they use is the one proposed by Allermann in \cite{allermann2012chern} in terms of cocycle conditions.

In \cite{khan2024tropical} Khan and Maclagan and independently in \cite{kaveh2024toric} Kaveh and Manon propose a different definition of a vector bundle on a tropical variety. This combinatorial approach is inspired by the observation that tropical linear spaces correspond to (valuated) matroids.\footnote{A matroid is a combinatorial abstraction of linear independence. See \cite{oxley2006matroid}. We refer the reader whose background is in algebraic geometry to \cite{katz2016matroid}.} Roughly speaking, a tropical toric vector bundle in the sense of Khan and Maclagan on a tropical toric variety associated to a fan $\Delta$ is a triple $(\mathcal{M},\mathcal{G},\{E^\rho(j)\})$, where $\mathcal{M}$ is a (simple) valuated matroid on a ground set $\mathcal{G}$, and $\{E^\rho(j)\}$ is a collection of flats of the underlying matroid $\underline{\mathcal{M}}$ with some compatibility condition analogous to compatibility condition of the filtrations in Klyachko's classification of toric vector bundles \cite[Theorem 2.2.1]{alexander1989equivariant}. A more general case is defined by using embedded tropicalization and restriction of tropical toric vector bundle. The tropical toric vector bundles of Khan and Maclagan (in the case of trivial  valuated matroid structure) are the same as the Kaveh and Manon's toric matroid bundles.

Yet another approach to a notion of vector bundles over semirings focuses on the impact of the choice of (Grothendieck) topology on the base space. In \cite{BJ24}, Borger and the first author study the module theory over semirings and show that while not all of the classical definitions of vector bundles agree over semirings, all definitions of line bundle agree. In particular, for a module $M$ over a semiring $R$, the following three are equivalent: (1) $M$ is Zariski-locally free of rank $1$, (2) $M$ is fpqc-locally free of rank $1$, and (3) $M$ is invertible. 


Our previous results in \cite{JMT20} do not cover the case of equivariant vector bundles. In particular, while we show that a tropical vector bundle globally splits as sum of line bundles, we do not know if this happens equivariantly in the case of tropical toric vector bundles. 

In this paper we define toric vector bundles over a semiring in a purely algebraic way, and we show they equivariantly split as a sum of line bundles and show that we have a  classification analogous to Klyachko's classification of toric vector bundles in \cite{alexander1989equivariant}. At the end, we briefly discuss how our work is related to the notion of tropical vector bundles by Khan and Maclagan \cite{khan2024tropical} . 

\subsection{Summary of results}
In this paper, we begin by confirming that the natural correspondence between locally free sheaves and geometric vector bundles still holds in the semiring setting. One may easily modify the classical proofs for schemes (over rings) to prove the correspondence in the context of semirings, but we include sketches of some proofs for the readers who are not familiar with commutative algebra of semirings. 

We introduce a notion of the functor of points of a locally free sheaf $\mathcal{F}$ of a scheme over a semiring $R$ as follows (Definition \ref{definition: functor of points locally free sheaf}): for an $R$-algebra $A$ and a morphism $x:\Spec A \to X$, we define the fiber $\mathcal{F}_x$ as
\[
\mathcal{F}_x:=(x^*\mathcal{F})(\Spec A).
\]
In other words, $\mathcal{F}_x$ is the global sections of the pullback $x^*\mathcal{F}$ of $\mathcal{F}$ along $x$. Then we define the set:
\[
\mathcal{F}(A):=\bigsqcup_{x \in X(A)}\mathcal{F}_x,
\]
where $X(A)=\Hom(\Spec A, X)$, the set of $A$-rational points of $X$. The functor of points of $\mathcal{F}$ is the functor sending an $R$-algebra $A$ to $\mathcal{F}(A)$ and an $R$-algebra morphism $f:A \to B$ to the morphism $\mathcal{F}(A) \to \mathcal{F}(B)$ given by $v \mapsto f^*(v)$, where $f^*(v) \in \mathcal{F}_{f^*(x)}$ is the pullback of the global section $v$ along $f^*:\Spec B \to \Spec A$ (see Remark \ref{remark: pullback global section}). With this definition in mind, we prove our first main result.

\begin{nothma}[Corollary \ref{corollary: equivalence of functor of points}]
The functor of points of a geometric vector bundle on a scheme $X$ over a semiring $R$ is the same as that of the corresponding locally free sheaf on a scheme $X$ over a semiring $R$.    
\end{nothma}

We define equivariant vector bundles on a scheme $X$ over a semiring $R$ by using the functor of points of a locally free sheaf (Definition \ref{definition: equivariant vector bundle}). It is well-known that any vector bundle on an affine toric variety is trivial \cite{Gub87} and that any toric vector bundle on an affine toric variety is equivariantly trivializable \cite[Proposition 2.2]{payne2008moduli}. We prove that a similar result holds for schemes over an idempotent semifield under irreducibility assumption.

\begin{nothmb}[Theorem \ref{theorem: torus-equivariant}]
Let $X$ be an irreducible scheme over an idempotent semifield $K$ and $G$ be an irreducible algebraic group over $K$ acting on $X$.  Let $E$ be a $G$-equivariant vector bundle on $X$ which is trivial as a vector bundle.  Then $E$ is a direct sum of equivariant line bundles. 
\end{nothmb}

The above theorem is a consequence on several results. The first such result is that the functor of points of a scheme over an idempotent semiring is characterized by its values at irreducible idempotent semirings (Proposition \ref{irreduciblefunctorofpoints}). 
Another aspect of the theory that heavily relies on working over idempotent semifields is irreducibility. More precisely, $X$ and $Y$ are irreducible schemes over an idempotent semifield $R$, then $X\times_R Y$ is irreducible (Proposition \ref{proposition: irreducible product}). This clearly fails for schemes over $\mathbb{Z}$. We also make use of the fact that every invertible matrix with entries in a zero-sum-free semiring is a product of a diagonal matrix and a permutation matrix (Corollary \ref{corollary: semidirect}). As a consequence, we prove that any equivariant vector bundle on a toric scheme $X$ over an idempotent semifield, which is trivial as a vector bundle, equivariantly splits as a sum of equivariant line bundles.

In Section \ref{section: classification of equivariant vb}, we classify equivariant vector bundles, which will be used later to prove a version of Klyachko's classification of toric vector bundles in the setting of an idempotent semifield. We first show that analogously to the results for schemes over $\mathbb{Z}$ we have a natural isomorphism for schemes over a semiring or a monoid $K$:
\begin{equation}\label{eq: maps}
\Hom_{\Spec K}(X,\text{GL}_1) \simeq \Gamma(X,\mathcal{O}_X^\times),
\end{equation}
 where $\text{GL}_1=\Spec K[t,t^{-1}]$. This allows us to go back and forth between morphisms and global sections. 

Let $K$ be a semiring or monoid, $X$ be a scheme over $K$, and $G$ be a group scheme over $K$ acting on $X$. We say a map $u:G \times_KX \to \text{GL}_1$ is \emph{principal} if for any $K$-algebra $A$ and for $(g,h,x) \in (G\times_KG\times_KX)(A)$, the following holds:
\[
u_A(gh,x)=u_A(h,x)u_A(g,hx).
\]
Then, we prove the following correspondence.

\begin{nothmc}[Proposition \ref{Equivariant structures on O_X}]
With the same notation as above, there is a one-to-one correspondence between
\[
\{\textrm{$G$-actions on $\mathcal{O}_X$ which make $\mathcal{O}_X$ an equivariant line bundle}  \}
\]
and
\[
\{\textrm{principal morphisms $u: G \times_K X \rightarrow GL_1$.}\}
\]
Moreover, this correspondence is a group isomorphism, where the first set is viewed as a group under tensor product\footnote{Proposition \ref{proposition: equivariant Picard group} ensures that the first set is a group.}, and the second is a group under point-wise multiplication.  
\end{nothmc}

In Section \ref{section: Gluing equivariant vector bundles}, we explore how we can glue equivariant vector bundles on affine subschemes to obtain an equivariant vector bundles on a scheme. It turns out that the only obstruction to gluing is each affine subscheme has to be closed under the $G$-action (Lemma \ref{lemma: restriction of equivariant vector bundle} and Proposition \ref{proposition: gluing equivariant vector bundles}). As a consequence, we prove the following, an equivariant version of our previous result on vector bundles in \cite{JMT20}.
  
\begin{nothmd}[Theorem \ref{theorem: equivariantly split theorem}]
Let $X$ be a toric scheme over an idempotent semifield $R$, and let $G$ be the corresponding torus.   Let $E$ be a $G$-equivariant vector bundle on $X$.  Then there are unique (up to permutation) equivariant line bundles $L_1, \ldots, L_n$ such that $E = L_1 \oplus \ldots \oplus L_n$ (as $G$-equivariant vector bundles).
\end{nothmd}

We prove that the $G$-action on an equivariant vector bundle is determined by the action on the torus (Proposition \ref{proposition: action determined by generic point}). To introduce our next result, let $R$ be an idempotent semifield and let $X$ be a toric scheme over $R$. We denote by $G$ the corresponding torus and by $\Lambda$ the dual lattice and $\Delta$ be the fan corresponding to $X$. Based on the isomorphism \eqref{eq: maps}, we further prove (Proposition \ref{proposition: equivariant line bundles and characters} and Lemma \ref{lemma: characters as elements of the dual lattice}) that there is a group homomorphism from $\Lambda$ to $\Pic_G(X)$, the group of isomorphism classes of $G$-equivariant line bundles on $X$, and characterize its kernel (Proposition \ref{proposition: dual lattice points corresponding to the trivial equivariant line bundle}). As a result, we prove the following.
    
\begin{nothme}[Theorem \ref{theorem: exact sequence}]
With the same notation as above, there is an exact sequence of abelian groups
\[ 
0 \rightarrow \bigcap_{\sigma \in \Delta} (\Lambda \cap \sigma^\vee)^\times \rightarrow \Lambda \rightarrow \mathrm{Pic}_G(X) \rightarrow \mathrm{Pic}(X) \rightarrow 0,
\]
where $\Pic_G(X) \to \Pic(X)$ is the forgetful map. 
\end{nothme}

In Section \ref{section: Tropical Klyachko theorem}, we prove a version of Klyachko's classification on toric vector bundles in our setting. Thanks to Theorem D, classifying toric vector bundles on the toric scheme $X$ over an idempotent  semifield associated to a fan $\Delta$ reduces to classifying toric line bundles on $X$. 

Let $\Lambda$ be the dual lattice. We first prove that toric line bundles on $X$ are in one-to-one correspondence with families of elements $[u_\sigma] \in \Lambda/(\Lambda\cap \sigma^\vee)^\times$ indexed by cones satisfying some compatibility condition (Proposition \ref{proposition: classification toric line bundles}), which we call \emph{Klyachko families} (Definition \ref{definition: Klyachko family}).

To turn this into a version of Klyachko's classification theorem, we define an $n$-dimensional \emph{$\Delta$-Klyachko space} over an idempotent semifield $K$ as a free $K$-module $E$ of rank $n$ with collections of decreasing filtrations $\{E^\rho(i)\}_{i\in \mathbb{Z}}$ indexed by the rays of $\Delta$, satisfying some compatibility condition analogous to Klyachko's theorem (Definition \ref{definition: Delta space}). Then, we prove an equivalence between isomorphism classes of $n$-dimensional $\Delta$-Klyachko spaces over an idempotent semifield $K$ and $S_n$-orbits of $n$-tuples of Klyachko families (Lemma \ref{lemma: correspondence family space}). By appealing to these results, we prove the following. 
    
\begin{nothmf}[Theorem \ref{theorem: tropical Klyachko}]
Let $K$ be an idempotent semifield. Let $X_K$ be the toric scheme over $K$ associated to a fan $\Delta$. The set of isomorphism classes of toric vector bundles on $X_K$ is in one-to-one correspondence with the set of isomorphism classes of $\Delta$-Klyachko spaces over $K$.   
\end{nothmf}

Finally, we note in Remark~\ref{remark: valuated matroids} that a $\Delta$-Klyachko space is a valuated matroid satisfying a certain compatibility condition, and that it is a tropical vector bundle in the sense of Khan and Maclagan in \cite{khan2024tropical} whose underlying matroid is free.

\subsection{Organization of the paper}
We begin in Section \ref{section: preliminaries} by reviewing some basic definitions and propositions needed for the paper. 

In Section \ref{section: Equivalence between locally free sheaves and geometric vector bundles}, we prove that the correspondence between geometric vector bundles and locally free sheaves is still valid in the setting of semirings. We then introduce the functor of points of a locally free sheaf and prove Theorem A. 

In Section \ref{section: equivariant vector bundles}, we define equivariant vector bundles via functor of points perspective.

In Section \ref{section:torus-equivariant}, we study irreducibility and equivariant bundles on irreducible schemes over an idempotent semifield. We then prove Theorem B. 

In Section \ref{section: classification of equivariant vb}, we classify equivariant bundles and prove Theorem C. 

In Section \ref{section: Gluing equivariant vector bundles}, we prove Theorems D and E. 

In Section \ref{section: Tropical Klyachko theorem}, we prove Theorem F, a tropical version of Klyachko's classification theorem.

In Appendix \ref{section: Toric varieties, toric monoid schemes, and tropical toric schemes}, we prove some basic results connecting toric varieties, toric monoid schemes, and toric schemes over $\mathbb{T}$. 

In Appendix \ref{sn as a functor}, we prove that $S_n$ can be seen as a scheme over $\mathbb{N}$.

\bigskip

\textbf{Acknowledgment} J.J. acknowledges the support of a AMS-Simons Research Enhancement Grant for Primarily Undergraduate Institution (PUI) Faculty during the writing of this paper. J.J. is grateful to Alex Borisov for helpful discussions.

\section{Preliminaries}\label{section: preliminaries}

\subsection{Toric vector bundles}

We refer the reader to \cite{Ful93} or \cite{CLS} for standard notation and background on toric varieties. We denote by $N$ a lattice and by $\Lambda= \Hom_\Z(N, \Z)$ its dual, and by $\Lambda_\R$ and $N_\R$ the dual vector spaces $\Lambda\otimes_\mathbb{Z}\R$ and $N\otimes_\mathbb{Z}\R$. We call $\Lambda$ the character lattice. Given an isomorphism $\Lambda  \cong  \Z^n$, every point $u =(u_1, \dots, u_n)$ in $\Lambda$ gives rise to a character (group homomorphism) 
\[
\chi^u: G_k(k) \cong (k^\times)^n \rightarrow k^\times, \quad (t_1,\dots,t_n)\mapsto t_1^{u_1}\dots t_n^{u_n}
\]
where $G_k$ is an algebraic torus over a field $k$. When $k$ is a semifield a point in $\Lambda$ similarly gives rise to a homomorphism of groups $(k^\times)^n \rightarrow k^\times$. 

We will denote by $\sigma$ a strongly convex rational polyhedral cone in $N_\R$, that is,
\[
\sigma = \sum\limits_{i=1}^{r}\R_{\geq 0}v_i, \quad v_i \in N.
\]
All the cones we will consider from now on will be strongly convex rational polyhedral. We denote by $\sigma^\vee$ the dual cone of $\sigma$:
\[
\sigma^\vee = \{u\in \Lambda_\R : \left< u,v \right> \geq 0, \forall v\in \sigma\}.
\]
A fan $\Delta$ is a finite collection of cones $\sigma$ in $N_\R$ such that each face of $\sigma$ is also a cone of $\Delta$ and the intersection of two cones of $\Delta$ is a face of both. 

We recall the definition of a toric vector bundle. Let $X$ be a toric variety (with torus $G$). We say that $\mathcal{E}$ is a \emph{toric vector bundle} on $X$ if it is a vector bundle, together with an action of $G$, compatible with the action on $X$.

We recall a classification of toric vector bundles.

\begin{mytheorem}(Klyachko’s Classification Theorem)
Let $k$ be a field. The category of toric vector bundles on $X(\Delta)_k$ is naturally equivalent to the category of finite-dimensional $k$-vector spaces $E$ with collections of decreasing filtrations $\{E^\rho(n)\}$ indexed by the rays of $\Delta$, satisfying the following compatibility condition:
For each cone $\sigma \in \Delta$, there exists a decomposition $E=\bigoplus \limits_{[u] \in \Lambda_\sigma} E_{[u]}$ such that
\[
E^\rho(i) = \sum_{\left<[u],v_\rho\right> \geq i}E_{[u]},
\]
for every $\rho \preceq \sigma$ and $i \in \mathbb{Z}$, where $\Lambda_\sigma:=\Lambda/(\sigma^\perp \cap \Lambda)$.
\end{mytheorem}

\subsection{Monoids and monoid schemes}
One may mimic the construction of schemes to define monoid schemes. Alternatively, one may take a functor of points approach to define the category of affine monoid schemes as the opposite category of monoids, and monoid schemes as (Zariski) sheaves on the category of monoids. As in the case of schemes, these two definitions are equivalent.

In this paper, by a monoid, we mean a commutative (multiplicative) monoid $M$.By a \emph{pointed monoid}, which means that there exists an absorbing element $0_M \in M$, that is, $0_M\cdot a=0_M$ for all $a \in M$. 

By a \textit{module} over a monoid $A$ we mean a set $M$ together with a map $\phi: A\times M \to M$, defined by $\phi(a, m) = am$, where $1m = m$ and $(ab)m=a(bm)$. 

The \textit{tensor product} $M \otimes_A N$ of two modules $M, N$ of a monoid $A$ is defined to be the quotient of $M \times N$ by the equivalence relation generated by $(am, n) \sim (m, an)$, for all $a\in A$, $m\in M$ and $n\in N$. As usual we denote the equivalence class on $(m, n)$ by $m\otimes n$. If $M$ and $N$ are two pointed modules of $A$ with distinguished points $m_0$ and $n_0$ respectively, then their tensor product is $$(M, m_0) \otimes_A (N, n_0) = (M \otimes_A N, m_0\otimes n_0).$$
For properties of the tensor product of (pointed) monoids we refer the reader to \cite{deitmar2008f1}.

Denote by $\mathbb{F}_1$ the initial object in the category of (pointed) monoids. We interpret it as a pointed multiplicative monoid with one element which we call 1 and absorbing element 0. A (pointed) monoid is automatically an $\mathbb{F}_1$-module.

By an \textit{$\mathbb{F}_1$-algebra}, we mean a pointed monoid. An $\mathbb{F}_1$-algebra (or a monoid) $M$ is said to be \textit{irreducible} if the product of non-nilpotents is not nilpotent. 

For a monoid $M$, an nonempty subset $I\subseteq M$ is said to be an \emph{ideal} if $MI \subseteq I$. An ideal $I$ is said to be \emph{prime} if $M- I$ is a multiplicative nonempty subset of $M$, and \emph{maximal} if $I$ is a proper ideal (i.e., $I \neq M$) which is not contained in any other proper ideal.

We can define the \emph{prime spectrum} $\Spec M$ of a monoid $M$ as the set of all prime ideals of $M$ equipped with the Zariski topology. The set $\{D(f)\}_{f \in M}$ forms an open basis of $\Spec M$, where $D(f):=\{\mathfrak{p} \in \Spec M \mid f \not\in \mathfrak{p}\}$. 

A \emph{monoid scheme} (or \textit{$\mathbb{F}_1$-scheme}) is defined to be a topological space equipped with a structure sheaf of monoids locally isomorphic to affine monoid schemes $(\Spec M, \mathcal{O}_{\Spec M})$. We refer the reader to \cite[\S 2]{jun2019picard} for the precise definitions and details.

A special class of monoid schemes called \emph{toric monoid schemes} (or \textit{toric $\mathbb{F}_1$-schemes) } can be obtained from fans. A toric monoid scheme is locally isomorphic to $\Spec M$, where $M$ is a toric monoid.\footnote{By a toric monoid, we mean the monoid obtained from a cone giving rise to an affine toric variety.} The fan tells us how to glue the affine pieces corresponding to each cone of the fan. For more details we refer the reader to \cite{cortinas2015toric}.

\subsection{Semirings, toric schemes over semirings, and vector bundles} \label{section: tropical toric schemes}

\begin{mydef}
By {\it semiring} we mean a commutative semiring with multiplicative unit, that is a nonempty set $R$ with two binary operations $(+,\cdot)$ satisfying:

\begin{itemize}
\item $(R,+)$ is a commutative monoid with identity element $0_R$
\item $(R,\cdot)$ is a commutative monoid with identity element $1_R$
\item For any $a,b,c \in R$: $a(b+c) = ab+ac$
\item $1_R \neq 0_R$ and $a\cdot 0_R = 0_R$ for all $a \in R$
\end{itemize}

$R$ is said to be \emph{zero-sum free} if $a+b=0$ implies $a=b=0$ for all $a,b \in R$. 

A {\it semifield} is a semiring in which all nonzero elements have a multiplicative inverse. 

\end{mydef}

We will denote by $\mathbb{B}$ the semifield with two elements $\{1,0\}$, where $1$ is the multiplicative identity, $0$ is the additive identity and $1+1 = 1$. 
The {\it tropical semifield}, denoted $\mathbb{T}$, is the set $\mathbb{R}  \cup \{-\infty\} $ with the $+$ operation to be the maximum and the $\cdot$ operation to be the usual addition, with $-\infty = 0_\mathbb{T}$.

\begin{mydef}
An \emph{affine semiring scheme} is the prime spectrum $X=\Spec A$ of some semiring $A$, equipped with a structure sheaf $\mathcal{O}_X$. A locally semiringed space is a topological space with a sheaf of semirings such that the stalk at each point has a unique maximal ideal. A \emph{semiring scheme} is a locally semiringed space which is locally isomorphic to an affine semiring scheme.\footnote{As we noted in the introduction, this is equivalent to the functor of points approach.}
\end{mydef}

\begin{rmk}
In \cite{giansiracusa2016equations} J. Giansiracusa and N. Giansiracusa propose a special case of the semiring schemes called tropical schemes in such as way that the tropicalization of an algebraic variety can be understood as a set of $\TT$-rational points of a tropical scheme. 
\end{rmk}

Toric schemes over idempotent semirings are intimately related to monoid schemes, in fact they arise from monoid schemes via base change. We make this relation precise. 

\begin{mydef}\label{definition: based change functor}
 For any ring (resp. semiring) $R$ we define $M\otimes_{\mathbb{F}_1}R$ to be the monoid ring (resp. semiring) $R[M]$, where we identify $0_M$ and $0_R$. This defines a functor from the category of monoids to the category of $R$-algebras, sending $M$ to $ R[M]$. There is also a natural functor (forgetful functor) $\mathcal{F}$ from the category of rings (resp. semirings) to the category of monoids sending a ring (resp. semiring) $A$ to $(A,\times)$, the underlying multiplicative monoid with $0_A$. There is an adjunction:
	\[
	\Hom_{\textbf{Monoids}}(M,\mathcal{F}(A)) \simeq \Hom_{\textbf{R-algebras}}(M\otimes_{\mathbb{F}_1}R,A).
	\]
	This gives rise to a functor $-\otimes_{\mathbb{F}_1}R$ from the category of monoid schemes to the category of schemes (resp. semiring schemes) sending $X$ to $X_R:=X\otimes_{\mathbb{F}_1}R$. One can find the details of that construction in \cite[Section 2.3]{giansiracusa2016equations}. 
\end{mydef}

\begin{mydef}\label{definition: tropical toric schemes}
By a \emph{toric scheme $X$ over a semiring} $S$, we mean a semiring scheme obtained from a toric monoid scheme via the base change to $S$. 
\end{mydef}

\begin{myeg}
Let $A = \mathbb{T}[x, y]$. Then $\Spec A$ is a toric scheme over $\mathbb{T}$, which is the tropical analogue of the affine 2-space. In this case, $A = \mathbb{T}[M]$, where $M = \mathbb{N}^2 \cup \{0\}$.  
\end{myeg}

\begin{mydef}\label{definition: linearly independent}
		Let $R$ be a semiring and $M$ be an $R$-module. 
		\begin{enumerate}
			\item 
			A subset $\{x_1,\dots,x_n\} \subseteq M$ is \emph{linearly independent} if
			\[
			\sum_{i=1}^na_ix_i = \sum_{i=1}^n b_ix_i
			\]
			implies that $a_i=b_i$ for all $i=1,\dots,n$. 
			\item 
			$M$ is \emph{free of rank $n$} if there exists a set $\{x_1,\dots,x_n\} $ of linearly independent generators of $M$ over $R$. 
		\end{enumerate}
	\end{mydef}
	
	One can easily see that if $M$ is a free $R$-module of rank $n$, then $M\simeq R^n$ (as an $R$-module).

We recall some results from \cite{JMT20} which will be used in this paper. We first recall the definition of idempotent pairs. 

\begin{mydef}\cite[Definition 3.6]{JMT20}\label{definition: idempotent pair}
Let $R$ be a semiring. By an \emph{idempotent pair} of $R$, we mean a pair $(e,f)$ such that $e+f=1$ and $ef=0$. An idempotent pair $(e,f)$ is said to be \emph{trivial} if $\{e,f\} = \{0,1\}$.
\end{mydef}

\begin{pro}[Proposition 3.10 \cite{JMT20}]\label{proposition: connected implies no idempotet paris}
		Let $R$ be a semiring, $\Spec R$ is connected if and only if any idempotent pair of $R$ is trivial.
	\end{pro}

\begin{lem}[Lemma 3.14 \cite{JMT20}]\label{lemma: direct summand of free module}
		Let $R$ be a zero-sum free semiring with only trivial idempotent pairs, and $M$ be a free module over $R$ with basis $S\subseteq M$.  Suppose $M=P\oplus Q$.  Then there exists a subset $S'\subseteq S$ such that $P$ is free with basis $S'$ and $Q$ is free with basis $S \setminus S'$.
	\end{lem}

 \begin{pro}[Proposition 3.18 \cite{JMT20}]\label{proposition: exact sequence R, GL, S_n}
		Let $R$ be a zero-sum free semiring. If $R$ has only trivial idempotent pairs, then one has the following split short exact sequence of groups which is natural in $R$:
		\begin{equation}\label{eq: exact sequence R,GL, S_n}
			\begin{tikzcd}
				0 \arrow[r] &
				(R^\times)^n \arrow[r,"f"]&
				\emph{GL}_n(R) \arrow[r,"g"]
				& S_n\arrow[r]
				& 0
			\end{tikzcd}
		\end{equation}
		where $f$ is the diagonal map and $g$ sends a matrix $A$ to the unique permutation $\sigma$ such that $A_{\sigma(i)i}\neq 0$ for all $i$. 
\end{pro}

\begin{pro}[Corollary 3.19 \cite{JMT20}]\label{corollary: semidirect}
		Let $R$ be a zero-sum free semiring. If $R$ has only trivial idempotent pairs, then one has the following isomorphism of groups:
		\[
		(R^\times)^n \rtimes S_n \simeq \emph{GL}_n(R),
		\]
  	where $S_n$ acts on $(R^\times)^n$ by permuting factors. 
\end{pro}

\begin{pro}[Theorem 4.7 \cite{JMT20}]\label{theorem: bundles split}
Let $X$ be an irreducible semiring scheme which is locally isomorphic to $\Spec R$, where $R$ is a zero-sum free semiring. Then any vector bundle of rank $n$ on $X$ is a direct sum of $n$ copies of line bundles on $X$. Moreover, this decomposition is unique up to permuting summands.
\end{pro}

\begin{rmk}
In \cite{JMT20}, Proposition \ref{theorem: bundles split} assumes that $R$ has only trivial idempotent pairs, however, this follows from the irreducibility of $X$.
\end{rmk}

\begin{condition}\label{condition: condition on open cover}
	Let $X$ be an irreducible monoid scheme. Suppose that $X$ has an open affine cover $\mathcal{U}=\{U_\alpha\}$ such that any finite intersection of the sets $U_\alpha$ is isomorphic to the prime spectrum of a cancellative monoid.\footnote{This condition is always satisfied with toric monoid schemes as in this case monoids are subsets of dual lattices. Also, all of our monoids are finitely generated by Gordan's lemma.}
\end{condition}

For instance, when $X$ is a toric scheme over $\mathbb{T}$, then the above condition holds.

\begin{pro}[Corollary 4.8 \cite{JMT20}]\label{corollary: vector bundle on affine is trivial}
Let $M$ be a cancellative monoid and $K$ be an idempotent semifield. Then any vector bundle on $X=\Spec K[M]$ is trivial. 
\end{pro}

\begin{pro}[Theorem 4.11 \cite{JMT20}]\label{theorem: bundles stable under scalar extension}
	Let $X$ be an irreducible monoid scheme satisfying Condition \ref{condition: condition on open cover}, and $K$ be an idempotent semifield. Then, there exists a natural bijection between $Vect_n(X)$ and $Vect_n(X_K)$. 
\end{pro}

\section{Equivalence between locally free sheaves and geometric vector bundles}\label{section: Equivalence between locally free sheaves and geometric vector bundles}

In this section, we show that the one-to-one correspondence between isomorphism classes of geometric vector bundles and locally free sheaves is still valid for semirings. The proofs are almost identical to the case of rings, but we include them here for completeness. We then provide a definition of the functor of points of a locally free sheaf, which is equivalent to the functor of points of the corresponding geometric vector bundle. 

In what follows, instead of saying a semiring scheme, we will simply say a \emph{scheme over $\mathbb{N}$}. Equivalently, one can think of a semiring scheme as a scheme $X$ with a structure map $X \to \Spec \mathbb{N}$.

Let $X$ be a scheme over $\mathbb{N}$. By $\mathbb{A}^n_X$, we mean the scheme $X\times_{ \mathbb{N}} \mathbb{A}^n_{\mathbb{N}}$, where $\mathbb{A}^n_{\mathbb{N}}=\Spec \mathbb{N}[x_1,\dots,x_n]$.

\begin{mydef}\label{definition: geometric vector bundle}
Let $Y$ be a scheme over $\mathbb{N}$. 
A \emph{geometric vector bundle of rank $n$} over $Y$ is a scheme $X$ over $Y$ and a morphism $f:X \to Y$, together with an open covering $\{U_i\}$ of $Y$, and isomorphisms:
\begin{equation}\label{eq: local iso}
\psi_i:f^{-1}(U_i)=X\times_YU_i \to \mathbb{A}^n_{U_i},
\end{equation}
such that for any $i, j$ and for any open affine subset $V=\Spec A \subseteq U_i \cap U_j$, the automorphism $\psi=\psi_j \circ \psi_i^{-1}$ of $\mathbb{A}_V^n=\Spec A[x_1,\dots,x_n]$ is given by a linear automorphism $\theta$ of $A[x_1,\dots,x_n]$. 
\end{mydef}

Let $(X,f,\{U_i\},\{\psi_i\})$ and $(X',f',\{U_j'\},\{\psi_j\})$ be geometric vector bundles on a scheme $Y$ over $\mathbb{N}$. An isomorphism is an isomorphism $g:X \to X'$ of underlying schemes $X$ and $X'$ such that 
\begin{equation}\label{eq: vb iso}
f'\circ g =f, 
\end{equation}
and $X$ and $f:X \to Y$ together with the covering of $Y$ consisting of $\{U_i\}\cup\{U_j'\}$ and the isomorphisms $\psi_i$, $\psi'_j\circ g$ also gives a geometric vector bundle structure on $X$. 

\begin{lem}\label{lemma: geo to locally free}
Let $Y$ be a scheme over $\mathbb{N}$ and $f:X \to Y$ be a geometric vector bundle of rank $n$. Then the following defines a locally free sheaf of rank $n$ on $Y$: for each open $U \subseteq Y$
\[
\mathcal{X}(U):= \{\text{morphisms } s:U \to X \mid f\circ s = \text{id}_U\}. 
\]
\end{lem}
\begin{proof}
The proof follows mutatis mutandis as the one for schemes over rings.
\end{proof}

Tensor algebras and symmetric algebras can be also constructed for semirings in the same way of those for rings. More precisely, let $R$ be a semiring. Then tensor product $M\otimes_RN$ of $R$-modules $M$, $N$ are defined as for rings. For more details, see \cite{borger2016witt} or \cite{giansiracusa2018grassmann}. Now, we let $\text{T}^0(M)=R$ and we inductively define
\[
\text{T}^d(M)=\text{T}^{d-1}(M)\otimes_R M. 
\]
Define the tensor algebra of $M$ to be 
$\text{T} (M) = \bigoplus_{d\geq 0} \text{T}^d(M)$, where multiplication is determined by the following natural isomorphism (concatenation):
\[
\text{T}^n(M)\otimes_R\text{T}^m(M) \to \text{T}^{n+m}(M). 
\]
We proceed to define the symmetric algebra $\text{S} (M)$ as the quotient of $ \text{T} (M)$ by the equivalence relation generated by $ m_i \otimes m_j \sim m_j \otimes m_i$, for all $m_i$, $m_j \in M$. As in the case of rings we have that $\text{Sym}^d(M)$ is the quotient of $\text{T}^d(M)$ by the equivalence relation generated by the elements $m_1 \otimes \dots \otimes m_d \sim m_{\sigma(1)} \otimes \dots \otimes m_{\sigma(d)}$, for all $m_i\in M$ and $\sigma$ a permutation on $d$ elements. Thus we have $\text{Sym} (M) = \bigoplus_{d\geq 0} \text{Sym}^d(M)$.

Likewise, for an $\mathcal{O}_Y$-module $\mathcal{F}$ of a scheme $Y$ over $\mathbb{N}$, the tensor algebra $\text{T}(\mathcal{F})$ and symmetric algebra $\text{Sym}(\mathcal{F})$ can be defined similarly for rings, as the sheafification of the obvious presheaves.

Next, let $\mathcal{F}$ be a quasi-coherent sheaf of algebras on a scheme $Y$ over $\mathbb{N}$. We let $\Spec \mathcal{F}$ be the scheme over $\mathbb{N}$ defined as follows: for each open subset $U \subseteq Y$, we have an affine scheme $\Spec \mathcal{F}(U)$. The scheme $\Spec \mathcal{F}$ is obtained by gluing these affine pieces. Note that there is a natural projection map 
\begin{equation}\label{eq: projection sheaves}
f:X \to Y \quad \textrm{such that} \quad f^{-1}(U)\simeq \Spec \mathcal{F}(U).
\end{equation}
By $\mathbf{V}(\mathcal{F})$, we will mean $\Spec(\text{Sym}(\mathcal{F}))$.

\begin{rmk}
Let $\mathcal{F}$ be a locally free sheaf of rank $n$ on a scheme $Y$ over $\mathbb{N}$. Then the symmetric algebra $\text{Sym}(\mathcal{F})$ is also a locally free sheaf on $Y$. In fact, the question is local. Hence it is enough to show it when $M=R^n$ for some semiring $R$. But, in this case, clearly the symmetric algebra $\text{Sym}(M)$ is a free $R$-module.
\end{rmk}

In what follows by a scheme, we always mean a scheme over $\mathbb{N}$ (equivalently semiring schemes) unless otherwise stated.

Let $\mathcal{F}$ be a locally free sheaf of rank $n$ on a scheme $Y$. Let $X=\mathbf{V}(\mathcal{F})$, with projection morphism $f:X \to Y$ as in \eqref{eq: projection sheaves}. For each open affine subset $U \subseteq Y$ for which $\mathcal{F}|_U$ is free, choose a basis of $\mathcal{F}$, and let 
\begin{equation}\label{eq: local iso sheaves}
\psi_U:f^{-1}(U) \to \mathbb{A}_U^n
\end{equation}
be the isomorphism resulting from the identification of $\text{Sym}(\mathcal{F}(U))$ with $\mathcal{O}_Y(U)[x_1,\dots,x_n]$. 

Let $R$ be a semiring, $T$ be a multiplicative subset of $R$, and $M$ be an $R$-module. As in the case for rings, one has the following isomorphism:
\begin{equation}\label{eq: sym eq}
T^{-1}\text{Sym}(M) \simeq \text{Sym}(T^{-1}M),
\end{equation}
Now, the following Lemma \ref{lemma: covering lemma} and Proposition \ref{proposition: sheaf to geometric} boil down to \eqref{eq: sym eq}.

\begin{lem}\label{lemma: covering lemma}
Let $(X,f,\{U_i\},\{\psi_i\})$ be a geometric vector bundles on a scheme $Y$ {over $\mathbb{N}$}. Let $\{U_j'\}$ be another covering of $Y$. Let $U_{ij}:=U_i \cap U_j'$ and $\psi_{ij}$ be the restriction of $\psi_i$ to $f^{-1}(U_{ij})$. Then, $(X,f,\{U_{ij}\},\{\psi_{ij}\}\}$ is a geometric vector bundle which is isomorphic to $(X,f,\{U_i\},\{\psi_i\})$.
\end{lem}
\begin{proof}
To show that $(X,f,\{U_{ij}\},\{\psi_{ij}\}\}$ is a geometric vector bundle, it suffices to prove that $\psi_{ij}:f^{-1}(U_{ij}) \to \mathbb{A}^n_{U_{ij}}$ is an isomorphism. In fact, we may further assume that $U_{i}=\Spec R_i$ and $U_{ij}=\Spec (R_i)_{f_j}$ for some $f_j \in R_i$. Then, it reduces to show that for a free $R_i$-module $M_i$, 
\begin{equation}\label{eq: sym commutes}
(\text{Sym}(M_i))_{f_j}\simeq \text{Sym}((M_i)_{f_j}),
\end{equation}
which is true for modules over semirings. 

Finally, it is clear that $(X,f,\{U_{ij}\},\{\psi_{ij}\}\}$ is isomorphic to $(X,f,\{U_i\},\{\psi_i\})$.
\end{proof}

\begin{pro}\label{proposition: sheaf to geometric}
Let $\mathcal{F}$ be a locally free sheaf of rank $n$ on a scheme $Y$ over $\mathbb{N}$ and $X=\Spec (\text{Sym}(\mathcal{F}))$. With $f$ and $U_i$ as in \eqref{eq: projection sheaves}, $(X,f,\{U_i\},\{\psi_i\})$ is a geometric vector bundle of rank $n$ over $Y$, which (up to isomorphism) does not depend on the bases of $\mathcal{F}|_{U_i}$ chosen. 
\end{pro}
\begin{proof}
The proof follows mutatis mutandis as the one for schemes over rings.
\end{proof}

\begin{lem}\label{lemma: universal property sym}(Universal property of symmetric algebras)
Let $Y$ be a scheme and $\mathcal{F}$ be an $\mathcal{O}_Y$-module. Let $i:\mathcal{F} \to \text{Sym}(\mathcal{F})$ be a natural map. For each sheaf $\mathcal{G}$ of $\mathcal{O}_Y$-algebras and map $f:\mathcal{F} \to \mathcal{G}$ of $\mathcal{O}_Y$-modules, there exists a unique map $g:\text{Sym}(\mathcal{F}) \to \mathcal{G}$ of $\mathcal{O}_Y$-algebras such that $g=f\circ i$. 
\end{lem}
\begin{proof} 
The proof follows mutatis mutandis as the one for schemes over rings.
\end{proof}

Let $Y$ be a scheme. For $\mathcal{O}_Y$-modules $\mathcal{F}$ and $\mathcal{G}$, we let $\underline{\mathrm{Hom}}_{\mathcal{O}_Y}(\mathcal{F},\mathcal{G})$ be the sheaf: for each open $U \subseteq Y$
\[
\underline{\mathrm{Hom}}_{\mathcal{O}_Y}(\mathcal{F},\mathcal{G})(U):=\mathrm{Hom}_{\mathcal{O}_Y(U)}(\mathcal{F}(U),\mathcal{G}(U)).
\]
Note that as in the case for rings, $\underline{\mathrm{Hom}}_{\mathcal{O}_Y}(\mathcal{F},\mathcal{G})$ is indeed a sheaf. We let $\mathrm{Hom}_{\mathcal{O}_Y}(\mathcal{F},\mathcal{G})$ be the monoid of maps of $\mathcal{O}_Y$-modules. 

We will later need an explicit description of the functor of points of a tensor product of locally free sheaves.  For this we will need the Hom-tensor adjunction. 

\begin{lem}
Let $(X, \mathcal{O}_X)$ be a semiringed space.  Let $\mathcal{E}, \mathcal{F}, \mathcal{G}$ be $\mathcal{O}_X$-modules.  Then there is an isomorphism
\[
\mathrm{Hom}_{\mathcal{O}_X}(\mathcal{E}\otimes_{\mathcal{O}_X}\mathcal{F}, \mathcal{G})\cong \mathrm{Hom}_{\mathcal{O}_X}(\mathcal{E}, \underline{\mathrm{Hom}}_{\mathcal{O}_X}(\mathcal{F}, \mathcal{G}))
\]
which is natural in all arguments, where $\underline{\mathrm{Hom}}$ is the sheaf-valued Hom.
\end{lem}
\begin{proof}
The proof follows mutatis mutandis as the one for schemes over rings.
\end{proof}

As in the case for rings, one also has the following. 

\begin{lem}\label{lemma: pullback pushforward adjunction}
 Let $f:X \to Y$ be a morphism of schemes over $\mathbb{N}$. Let $\mathcal{F}$ (resp.~$\mathcal{G}$) be an $\mathcal{O}_X$-module (resp.~$\mathcal{O}_Y$-module). Then one has the following pullback-pushforward adjunction:
 \[
 \Hom_{\mathcal{O}_X}(f^*\mathcal{G},\mathcal{F})\simeq \Hom_{\mathcal{O}_Y}(\mathcal{G},f_*\mathcal{F}).
 \]
 \end{lem}
\begin{proof}
The proof follows mutatis mutandis as the one for schemes over rings.
\end{proof}

\begin{pro}\label{proposition: naturality}
Let $Y$ be a scheme and $\mathcal{F}$ be a locally free $\mathcal{O}_Y$-module of finite rank. The dual $\mathcal{F}^\vee$ is defined to be the sheaf $\underline{\mathrm{Hom}}_{\mathcal{O}_Y}(\mathcal{F},\mathcal{O}_Y)$. Let $X=\mathbf{V}(\mathcal{F})$ be the geometric vector bundle over $Y$ associated to $\mathcal{F}$ with $f:X \to Y$. Let $\mathcal{E}$ be the locally free sheaf of sections of $X$. Then $\mathcal{E}$ is isomorphic to $\mathcal{F}^\vee$.
\end{pro}
\begin{proof}
The proof follows mutatis mutandis as the one for schemes over rings.
\end{proof}

From Lemma \ref{lemma: geo to locally free} and Propositions \ref{proposition: sheaf to geometric} and \ref{proposition: naturality}, one obtains the following.

\begin{pro}\label{proposition: equivalence of categories}
Let $Y$ be a scheme {over $\mathbb{N}$}. Then there is a one-to-one correspondence between isomorphism classes of geometric vector bundles on $Y$ and isomorphism classes of locally free sheaves on $Y$. 
\end{pro}

From the above proposition, over semirings, we may identify geometric vector bundles and locally free sheaves. In the following, we provide a description of a locally free sheaf as a functor of points. This perspective will be used in later sections. We start with the following lemma. 

\begin{lem}
Let $f: X\rightarrow Y$ be a morphism of schemes {over $\mathbb{N}$}.  Let $\pi: E\rightarrow Y$ be a geometric vector bundle.  Then the pullback $p: E \times_Y X \rightarrow X$ of $E$ along $f$ is a geometric vector bundle.
\end{lem}
\begin{proof}
The proof follows mutatis mutandis as the one for schemes over rings.
\end{proof}

Let $Y$ be a scheme over a semiring $R$ and $\pi:E \to Y$ be a geometric vector bundle. Let $A$ be an $R$-algebra and $Y(A)=\Hom_{R}(\Spec A, Y)$ and $E(A)=\Hom_{R}(\Spec A, E)$. For each $y:\Spec A \to Y$ of $Y(A)$, we let $E_y$ be the pullback of the following diagram (taken in the category of sets):
\[
\begin{tikzcd}
E_y \arrow[dashed,r] \arrow[dashed,d] & E(A)\arrow[d,"\tilde{\pi}"] \\
\{y\} \arrow[hookrightarrow,r] & Y(A)
\end{tikzcd}
\]
where $\tilde{\pi}(x)=\pi\circ x$.

For an $R$-algebra morphism $\phi:A \to B$ and the induced map $\phi^*:\Spec B \to \Spec A$, we have
\[
\phi_*:E(A) \to E(B), \quad x \mapsto x \circ \phi^*. 
\]

\begin{pro}\label{proposition: Functor of points of geometric vector bundle via pullbacks}
With the same notation as above, the functor of points $E(A)$ has the following properties:
\begin{enumerate}
\item 
The pullback $E_y$ is in canonical bijection with the set of global sections of the pullback bundle $p:E \times_Y \Spec A \to \Spec A$, where $\Spec A$ is viewed as a scheme over $Y$ via $y$. 
\item 
For an $R$-algebra morphism $\phi: A \rightarrow B$ and $y \in Y(A)$, the restriction of $\phi_*:E(A) \to E(B)$ to $E_y$ is the map 
\[
\phi_*: E_y \rightarrow E_{x}, \quad \textrm{where } x=y\circ \phi^*,
\]
given by sending a section $s\in E_y$ (which may be viewed as a map $\Spec A\rightarrow E \times_Y \Spec A$) to the pullback map $\Spec B\rightarrow E \times_Y \Spec B$ viewed as an element of the fiber $E_{x}$.
\end{enumerate}  
\end{pro}
\begin{proof}
Let $y\in Y(A)$.  Elements of $E_y$ are morphisms $v: \Spec A\rightarrow E$ such that $\pi\circ v = y$.  These are equivalent to pairs of morphisms $v: \Spec A\rightarrow E$ and $u: \Spec A\rightarrow \Spec A$ such that $u$ is the identity map and $\pi\circ v = y \circ \mathrm{id}$.  Without the condition that $u$ is the identity, such pairs correspond to morphisms $\eta: \Spec A \rightarrow E\times_Y \Spec A$, so with this condition included, such pairs correspond to morphisms $\eta: \Spec A \rightarrow E\times_Y \Spec A$ whose composition with the projection $p: E\times_Y \Spec A \rightarrow \Spec A$ is the identity.  These are the same as global sections of the sheaf associated with the pullback bundle $p: E\times_Y \Spec A \rightarrow \Spec A$.  This establishes the first claim.

Let $v\in E_y$.  We identify $v$ with the map $s: \Spec A\rightarrow E \times_Y \Spec A$ obtained by applying the universal property of the pullback to $v$ and $\mathrm{id}$. By applying $- \times_{\Spec A} \Spec B$ we obtain a map $t: \Spec B\rightarrow E \times_Y \Spec B$. So, we have the following diagram:
\[
\begin{tikzcd}
\Spec B \arrow[r,"\phi^*"] \arrow[dd,dashed, bend right=70,swap,"\text{id}"] \arrow[d,swap,"t"] & \Spec A \arrow[d,"s"] \arrow[dr,"v"] &  \\
E\times_Y \Spec B \arrow[r,"p_1"] \arrow[d,swap,"p_2"] & E\times_Y\Spec A \arrow[d,"p"]  \arrow[r]& E \arrow[d,"\pi"]  \\
\Spec B \arrow[r,"\phi^*"] \arrow[rr,bend right=15,swap,"x"] & \Spec A \arrow[r,"y"] & Y
\end{tikzcd}
\]
Since all squares are pullback squares, one can easily see that $t$ is the global section corresponding to the element $\phi_*(v)=v\circ \phi^*$. 
\end{proof}

Proposition \ref{proposition: Functor of points of geometric vector bundle via pullbacks} motivates our definition of the functor of points of a locally free sheaf, which will be given shortly.  But first we need to define the pullback of a global section of a sheaf.

\begin{rmk}\label{remark: pullback global section}
Let $f: Y\rightarrow X$ be a morphism of schemes over $\mathbb{N}$ and $\mathcal{F}$ be an $\mathcal{O}_X$-module.  Then any global section $s\in \Gamma(X, \mathcal{F})$ induces a global section of $f^*\mathcal{F}$ by functoriality of pullbacks: Specifying the global section $s$ is equivalent to specifying a morphism $\mathcal{O}_X\rightarrow \mathcal{F}$, which may be pulled back to a morphism $\mathcal{O}_Y = f^*\mathcal{O}_X \rightarrow f^*\mathcal{F}$.
\end{rmk}

Let $\mathcal{F}$ be a locally free sheaf on a scheme $X$ over a semiring $R$. For an $R$-algebra $A$ and a morphism $x: \Spec A\rightarrow X$, the fiber $\mathcal{F}_x$ is defined as the global sections of the pullback sheaf $x^*\mathcal{F}$, i.e., 
\[
\mathcal{F}_x:=(x^*\mathcal{F})(\Spec A)
\]
The set $\mathcal{F}(A)$ of $A$-valued points is defined as the union of $\mathcal{F}_x$ over all $x\in X(A)$. 

Given an $R$-algebra morphism $f: A\rightarrow B$ and an element $v\in \mathcal{F}_x$ of some fiber, let $f^*(v) \in \mathcal{F}_{f^*(x)}$ as the pullback of the global section $v$ along $f^*:\Spec B \to \Spec A$ as explained in Remark \ref{remark: pullback global section}. Now, the functor of points of $\mathcal{F}$ is defined as follows.

\begin{mydef}\label{definition: functor of points locally free sheaf}
Let $\mathcal{F}$ be a locally free sheaf on a scheme $X$ over a semiring $R$. The functor of points of $\mathcal{F}$ is the functor sending an $R$-algebra $A$ to $\mathcal{F}(A)$ and an $R$-algebra morphism $f:A\rightarrow B$ to the morphism $\mathcal{F}(A)\rightarrow\mathcal{F}(B)$ given by $v\mapsto f^*(v)$.
\end{mydef}

We show that the functor of points of a locally free sheaf is equal to the functor of points of the corresponding geometric vector bundle.  The only difficulty is that the definition of the functor of points of a locally free sheaf involves the pullback sheaf rather than the sheaf corresponding to the pullback of the geometric vector bundle.

\begin{lem}\label{lemma: equivalence}
Let $\pi: E\rightarrow X$ be a geometric vector bundle, and let $f:Y\rightarrow X$ be a morphism of schemes.  Let $\mathcal{F}$ be the sheaf of sections of $E$.  Then the sheaf $\mathcal{G}$ of sections of $E\times_X Y\rightarrow Y$ is canonically isomorphic to $f^*\mathcal{F}$.
\end{lem}
\begin{proof}
Let $U\subseteq X$ be an open subset. A section of $E$ over $U$ may be pulled back to a section of $E\times_X Y$ over $f^{-1}(U)$.  Since $f_*\mathcal{G}$ sends $U$ to the sections of $E\times_X Y$ over $f^{-1}(U)$, we have obtained a homomorphism $\mathcal{F}\rightarrow f_*\mathcal{G}$.  By adjunction we obtain a homomorphism $f^*\mathcal{F}\rightarrow \mathcal{G}$.  

To check this is an isomorphism, we may work locally (or even on stalks), and so we may assume $E$ is the trivial bundle.  Choose a basis $v_1,\ldots, v_n$ of sections of $E$, and equip $E\times_X Y$ with the pulled back basis $w_1, \ldots, w_n$.  We may now identify $\mathcal{F}$ with $\mathcal{O}_X^n$ and $\mathcal{G}$ with $\mathcal{O}_Y^n$.  The map $\mathcal{F}\rightarrow f_*\mathcal{G}$ we constructed above is the map $\mathcal{O}_X^n \rightarrow f_*\mathcal{O}_Y^n$ sending one basis to the other (i.e. the direct sum of many copies of the canonical map $\mathcal{O}_X\rightarrow f_*\mathcal{O}_Y$.  The map $f^*\mathcal{F}\rightarrow \mathcal{G}$ is then the map $f^*\mathcal{O}_X^n\rightarrow \mathcal{O}_Y^n$ obtained as the direct sum of copies of the canonical isomorphism $f^*\mathcal{O}_X\rightarrow \mathcal{O}_Y$. 
\end{proof}

Note that Proposition \ref{proposition: Functor of points of geometric vector bundle via pullbacks} looks almost identical to the definition of the functor of points of a locally free sheaf in Definition \ref{definition: functor of points locally free sheaf}; the only major difference is that it involves pullbacks of schemes instead of pullbacks of sheaves, which is shown to be equivalent by Lemma \ref{lemma: equivalence}. To be precise, we have following.

\begin{cor}\label{corollary: equivalence of functor of points}
The functor of points of a geometric vector bundle is the same as that of the corresponding locally free sheaf.
\end{cor}
\begin{proof}
First, we consider what the functor of points looks like on objects via the first statement of Proposition {\ref{proposition: Functor of points of geometric vector bundle via pullbacks}}. 

Let $R$ be a semiring and $A$ be an $R$-algebra. For any geometric vector bundle $E$ on a scheme $Y$ over $R$, the set $E(A)$ is the disjoint union of $E_y$ over all $ y \in Y(A)$. Also, for each $y \in Y(A)$, Proposition \ref{proposition: Functor of points of geometric vector bundle via pullbacks} says that $E_y$ is the set of global sections of the pullback of $E$ along $y$. 

Moreover, from Lemma \ref{lemma: equivalence}, we know that the sections of this pullback bundle are just global sections of the pullback sheaf. It follows that $E_y$ is the set of global sections of $y^*\mathcal{F}$ where $\mathcal{F}$ is the sheaf of sections of $E$.  By the definition of the functor of points of a locally free sheaf, we can write this as $E_y = \mathcal{F}_y$.  Then taking the union over all $y \in Y(A)$ gives 
\[
E(A) = \mathcal{F}(A).
\]
The case of morphisms follow by a similar argument. 
\end{proof}

Because the functor of points is defined in terms of the pullback sheaf, it will be useful to understand the pullback of a tensor product.  For this, the following lemma will be helpful.
\begin{lem}Let $f:X\rightarrow Y$ be a morphism of schemes over $\mathbb{N}$. Let $U\subseteq Y$ be an open subset.
\begin{enumerate}
    \item 
    Let $\mathcal{F}$ be an $\mathcal{O}_Y$-module and $\mathcal{G}$ be a module over $\mathcal{O}_Y\mid_U$.  Define $\mathcal{H}$ on $Y$ by $\mathcal{H}(V) = \mathcal{G}(U\cap V)$.  Then $\Hom_{\mathcal{O}_Y\mid_U}(\mathcal{F}\mid_U, \mathcal{G}) \cong \Hom_{\mathcal{O}_Y}(\mathcal{F}, \mathcal{H})$.
    \item 
    Let $\mathcal{F}$ be an $\mathcal{O}_X$-module.  Let $j: f^{-1}(U)\rightarrow X$ be the inclusion.  Then $f_*j_*(\mathcal{F}\mid_{f^{-1}(U)})$ sends an open subset $V\subseteq X$ to $(f_*\mathcal{F})(U\cap V)$.
    \item Let $\mathcal{F}$ be an $\mathcal{O}_Y$-module and $\mathcal{G}$ be an $\mathcal{O}_X$-module. There is a natural isomorphism 
    \begin{equation}\label{eq: sheaf version of pushforward-pullback adjunction}
        \underline{\Hom}_{\mathcal{O}_Y}(\mathcal{F}, f_*\mathcal{G}) \cong f_*\underline{\Hom}_{\mathcal{O}_X}(f^*\mathcal{F}, \mathcal{G}).
    \end{equation}
\end{enumerate}
\end{lem}
\begin{proof}
The proof follows mutatis mutandis as the one for schemes over rings.
\end{proof}

\begin{lem}\label{lemma: pullback of tensor}
Let $f:X\rightarrow Y$ be a morphism of schemes over $\mathbb{N}$.  Let $\mathcal{F}, \mathcal{G}$ be locally free sheaves on $Y$.  Then there is a natural isomorphism $f^*(\mathcal{F})\otimes_{\mathcal{O}_X} f^*(\mathcal{G})\rightarrow f^*(\mathcal{F}\otimes_{\mathcal{O}_Y} \mathcal{G})$.
\end{lem}
\begin{proof}

Let $\mathcal{H}$ be a sheaf on $X$.  Then by the hom-tensor adjunction and the pullback-pushforward adjunction, we obtain
\begin{align}
    &\mathrm{Hom}_{\mathcal{O}_X}(f^*\mathcal{F}\otimes_{\mathcal{O}_X} f^*\mathcal{G}, \mathcal{H}) \cong \mathrm{Hom}_{\mathcal{O}_X}(f^*\mathcal{F}, \underline{\mathrm{Hom}}_{\mathcal{O}_X}(f^*\mathcal{G}, \mathcal{H})) \cong \\
    &\mathrm{Hom}_{\mathcal{O}_Y}(\mathcal{F}, f_*\underline{\mathrm{Hom}}_{\mathcal{O}_X}(f^*\mathcal{G}, \mathcal{H}))
\end{align}
and
\begin{equation}
\mathrm{Hom}_{\mathcal{O}_X}(f^*(\mathcal{F}\otimes_{\mathcal{O}_Y} \mathcal{G}), \mathcal{H}) \cong \mathrm{Hom}_{\mathcal{O}_Y}(\mathcal{F}\otimes_{\mathcal{O}_Y} \mathcal{G}, f_*\mathcal{H}) \cong \mathrm{Hom}_{\mathcal{O}_Y}(\mathcal{F}, \underline{\mathrm{Hom}}_{\mathcal{O}_Y}(\mathcal{G}, f_*\mathcal{H})).
\end{equation}
The result follows by using \eqref{eq: sheaf version of pushforward-pullback adjunction} (with $\mathcal{G}, \mathcal{H}$ in place of $\mathcal{F}, \mathcal{G}$) to compare the right sides of the two equations above.
\end{proof}

\begin{pro}\label{proposition: functor of points of tensor product}Let $X$ be a scheme over a semiring $R$ and $\mathcal{F}, \mathcal{G}$ be locally free sheaves on $X$.  The functor of points of $\mathcal{F}\otimes_{\mathcal{O}_X}\mathcal{G}$ is given by $A\mapsto \bigsqcup_{x\in X(A)}\mathcal{F}_x\otimes_{A}\mathcal{G}_x$.
\end{pro}
\begin{proof}This is just a restatement of Lemma \ref{lemma: pullback of tensor}.  Let $A$ be an $R$-algebra and let $x\in X(A)$.  Then by definition, the fiber of $(\mathcal{F}\otimes_{\mathcal{O}_X}\mathcal{G})(A)\rightarrow X(A)$ over $x$ is the $A$-module corresponding to $x^*(\mathcal{F}\otimes_{\mathcal{O}_X}\mathcal{G})$, which is isomorphic to $x^*(\mathcal{F})\otimes_A x^*(\mathcal{G})$.  This is the tensor product of the fibers of $\mathcal{F}$ and $\mathcal{G}$.
\end{proof}

\begin{pro}\label{proposition: addition morphism}
Let $E$ be a locally free sheaf on a scheme $X$ over a semiring $R$. Consider the map $+:E\times_X E \to E$ which is simply vector addition on the level of functors of points. Then, $+$ is a morphism of schemes when $E$ is viewed as a geometric vector bundle. Similarly there is a morphism of schemes $\mathbb{A}^1_X \times_XE \to E$ whose functor of points is the scalar multiplication map $(A \times X(A))\times_{X(A)} E(A) \cong A \times E(A) \to E(A)$.
\end{pro}
\begin{proof}
Since the fibers are $A$-modules, we have an addition operation
\[
+: E(A)\times_{X(A)} E(A) \to E(A),
\]
which is natural in $A$ because pulling back sections yields an additive map $E(A)\rightarrow E(B)$ for any $f: A\rightarrow B$.  The universal property of pullback yields a natural isomorphism 
\[
(E\times_X E)(A) \cong E(A)\times_{X(A)} E(A).
\]
Composing gives a natural addition operation $(E\times_X E)(A) \rightarrow E(A)$.  By naturality, this must be the functor of points of a morphism of schemes.  

The proof for scalar multiplication is similar; the only substantial difference is that we need to check that $A \times E(A)\rightarrow E(A)$ is natural.  If $x\in X(A)$ and $f:A\rightarrow B$, this amounts to showing that the composition of $A\times E_x \rightarrow E_x$ with the map $E_x \rightarrow E_{f(x)} \cong E_x \otimes_A B$ (given by $v\mapsto v\otimes 1$)  is equal to the composition of the map $A\times E_x \rightarrow B \times E_{f(x)}$ (given by $a,v \mapsto f(a), v\otimes 1$) with $B \times E_{f(x)}\rightarrow E_{f(x)}$, which is clear.
\end{proof}

\section{Equivariant vector bundles}\label{section: equivariant vector bundles}

From now on, by a vector bundle, we mean a locally free sheaf. 

Let $E$ be a vector bundle on a scheme $X$ over a semiring $R$. Let $A$ be an $R$-algebra. Consider any element of $X(A)$:
\[
x:\Spec A \to X.
\]
As we noted in Remark \ref{remark: pullback global section}, any global section $s$ of $E$ determines a global section of the pullback bundle $x^*E$. Since $E_x:=x^*(E)(\Spec A)$, the element $s$ determines an element $s(x)$ of $E_x$. Hence, $s$ also determines a section of the projection $\pi_A:E(A) \to X(A)$.

Let $R$ be a semiring. By a group scheme over $R$, we mean a group object in the category of schemes over $R$. In particular, for any $R$-algebra $A$, the set $G(A)$ of $A$-points of $G$ is a group. Also, recall that for any vector bundle $E$ on $X$ and $x \in X(A)$, since $E_x:=x^*E(\Spec A)$, the set $E_x$ is indeed an $A$-module.

\begin{mydef} \label{definition: equivariant vector bundle}
Let $X$ be a scheme over a semiring $R$ and $G$ be a group scheme over $R$ acting on $X$. We define an \textit{equivariant vector bundle} to be a vector bundle $E$ on $X$ together with an action of $G(A)$ on $E(A)$ for each $R$-algebra $A$ satisfying the following:
\begin{enumerate}
    \item
the action is natural in $A$. \label{item:naturality of action on bundle}
\item 
the action makes $\pi_A$ equivariant. 
\item
the action makes the induced map $E_x \to E_{gx}$ $A$-linear.
\end{enumerate}
\end{mydef}

The first condition means the following: for $R$-algebras $A$ and $B$ with a map $f^*:\Spec B \to \Spec A$ induced by an $R$-algebra map $f:A \to B$, the following diagram commutes:
\[
\begin{tikzcd}
G(A) \times E(A) \arrow[r,"\rho_A"] \arrow[d," \varphi \times \psi",swap] & E(A) \arrow[d,"\psi"] \\
G(B) \times E(B) \arrow[r,"\rho_B"] & E(B)
\end{tikzcd}
\]
where $\varphi(\alpha)=\alpha \circ f^*$ and $\psi(\beta)=f^*(\beta)$ (as in Definition \ref{definition: functor of points locally free sheaf}) and $\rho_A$ 
(resp.~$\rho_B$) is an action of $G(A)$ (resp.~$G(B)$) on $E(A)$ (resp.~$E(B)$). 

The second condition means that the following diagram commutes:
\[
\begin{tikzcd}
G(A) \times E(A) \arrow[r,"\rho_A"] \arrow[d,"\textrm{id} \times \pi_A",swap] & E(A) \arrow[d,"\pi_A"] \\
G(A) \times X(A) \arrow[r,"\rho'_A"] & X(A)
\end{tikzcd}
\]
where the top row is the action of $G(A)$ on $E(A)$ and the bottom row is the action of $G(A)$ on $X(A)$. 

The third condition is clear: each $g \in X(A)$ induces an automorphism of $E(A)$ and it has to be $A$-linear for each fiber.

\begin{myeg}\label{example: trivial eq}
Let $X$ be a scheme over a semiring $R$ with a $G$-action, where $G$ is a group scheme over $R$. Consider a free sheaf $\mathcal{F} = \mathcal{O}_X^n$.  Let $v_1, \ldots, v_n$ be a basis for $\mathcal{F}$.  Let $x\in X(A)$ for some $R$-algebra $A$.  Then $\mathcal{F}_x$ is free with basis $(v_1)_x, \ldots, (v_n)_x$ where $(v_i)_x\in \mathcal{F}_x$ corresponds to $x^*v_i\in \Gamma(\Spec A, x^*\mathcal{F})$.

Define an action $\rho_A: G(A)\times \mathcal{F}(A)\rightarrow \mathcal{F}(A)$ by the condition that for $g\in G(A)$, $x\in X(A)$ and $v\in \mathcal{F}_x$,
\[
\rho_A((g,v)) = \sum_i a_i (v_i)_{gx},
\]
where $v = \sum_i a_i (v_i)_x$ is the basis expansion of $v$.  $\mathcal{O}_X^n$ with this action is called the \emph{trivial $G$-equivariant vector bundle}.  It is clear that $\rho_A$ is compatible with the projection since it maps the fiber over $x$ to the fiber over $gx$. It is also clear that $\rho_A$ is an action and that it induces linear isomorphisms between fibers.  To show the trivial $G$-equivariant vector bundle is in fact an equivariant vector bundle, we only need to check naturality.

Let $f^*:\Spec B \to \Spec A$ be a map induced by a map $f:A \to B$ of $R$-algebras. We want the following diagram to commute:
\[
\begin{tikzcd}
G(A) \times \mathcal{F}(A) \arrow[r,"\rho_A"] \arrow[d," \varphi \times \psi",swap] & \mathcal{F}(A) \arrow[d,"\psi"] \\
G(B) \times \mathcal{F}(B) \arrow[r,"\rho_B"] & \mathcal{F}(B)
\end{tikzcd}
\]
For $(g,v) \in G(A) \times \mathcal{F}(A)$, we have
\[
\psi\circ\rho_A((g,v))=\psi(\sum_i a_i (v_i)_{gx})= \sum_i f(a_i) \psi((v_i)_{gx}) = \sum_i f(a_i) (f^*v_i)_{f^*(gx)}
\]
On the other hand, with $v=\sum_i a_i(v_i)_x$, we have
\[
\psi(v)=\psi(\sum_ia_i(v_i)_x) = \sum_i f(a_i) \psi((v_i)_x) = \sum_i f(a_i)f^*((v_i)_x).
\]
Hence, we have
\[
\rho_B\circ (\varphi \times \psi)(g,v)= \rho_B(f^*(g), \sum_i f(a_i)(f^*v_i)_{f^*x}))=\sum_i f(a_i) (f^*v_i)_{f^*(g)f^*x}=\sum_i f(a_i)(f^*v_i)_{f^*(gx)},
\]
showing that $\psi\circ \rho_A = \rho_B\circ (\varphi \times \psi)$.
\end{myeg}

Obviously a direct sum of equivariant vector bundles is an equivariant vector bundle.  We now show that the same is true for tensor products.

\begin{lem}\label{lemma: pullback of dual sheaf}
Let $f: Y\rightarrow X$ be a morphism of schemes over a semiring $R$ and let $\mathcal{F}$ be a locally free sheaf on $X$.  Then there is a canonical isomorphism $f^*\underline{\Hom}_{\mathcal{O}_X}(\mathcal{F}, \mathcal{O}_X) \simeq \underline{\Hom}_{\mathcal{O}_Y}(f^*\mathcal{F}, \mathcal{O}_Y)$.
\end{lem}
\begin{proof}
The proof follows mutatis mutandis as the one for schemes over rings.
\end{proof}

Let $X$ be a scheme over a semiring $R$ with an action of a group scheme $G$ over $R$.  Let $\mathcal{E},\mathcal{F}$ be equivariant vector bundles on $X$. Let $A$ be an $R$-algebra. We apply Proposition \ref{proposition: functor of points of tensor product} to identify $(\mathcal{E}\otimes\mathcal{F})_x$ with $\mathcal{E}_x \otimes_A \mathcal{F}_x$ for each $x\in X(A)$. 

\begin{pro}\label{proposition: equivariant Picard group}
With the same notation as above, $\mathcal{E}\otimes\mathcal{F}$ is an equivariant vector bundle via the action obtained by linearly extending the following:
\[
g(v\otimes w) = gv \otimes gw, \quad \textrm{where } g\in G(A),~x\in X(A), ~v\in \mathcal{E}_x,~ w\in \mathcal{F}_x.
\]
Moreover, the set of isomorphism classes of equivariant line bundles is a group under tensor product.
\end{pro}
\begin{proof}If $v\otimes w$ is in the fiber over $x$, then it is clear that $gv\otimes gw$ is in the fiber over $gx$, which establishes compatibility with the projections, showing (2) of Definition \ref{definition: equivariant vector bundle}.  

The linearity of the actions on $\mathcal{E}_x$ and $\mathcal{F}_x$ implies $gv\otimes gw$ depends bilinearly on $v$ and $w$.  By the universal property of the tensor product, there is a unique linear map $\mathcal{E}_x \otimes \mathcal{F}_x \rightarrow \mathcal{E}_{gx} \otimes \mathcal{F}_{gx} $ sending $v\otimes w$ to $gv\otimes gw$.  This establishes that the map of interest is well defined and linear on fibers. This shows (3) of Definition \ref{definition: equivariant vector bundle}.

One can easily observe that (1) of Definition \ref{definition: equivariant vector bundle} follows immediately from the naturality of the actions on $\mathcal{E}$ and $\mathcal{F}$ plus naturality of the multiplication map $G(A)\times G(A)\rightarrow G(A)$.  Thus the tensor product is an equivariant vector bundle.

Now let $\mathcal{O}_X$ be the trivial equivariant line bundle (with basis element denoted $1$).  We have an isomorphism of vector bundles 
\[
\mathcal{O}_X\otimes \mathcal{F}\rightarrow\mathcal{F}
\]
which on the level of functors of points is given by $1_x \otimes v \mapsto v$. To check it is an isomorphism of $G$-equivariant vector bundles (so that $\mathcal{O}_X$ is the tensor product identity), we need to check that $g(1_x \otimes v)$ maps to $gv$.  But this follows immediately from $g(1_x \otimes v) = g1_x \otimes gv = 1_{gx} \otimes gv$.

It remains to show that equivariant line bundles are invertible under tensor product.  Let $L$ be an equivariant line bundle.  There is a canonical morphism $\underline{\Hom}_{\mathcal{O}_X}(L, \mathcal{O}_X) \otimes L \rightarrow \mathcal{O}_X$.  This is clearly an isomorphism for $L = \mathcal{O}_X$, and so for any line bundle it induces isomorphisms on stalks, and hence is an isomorphism.

By Lemma \ref{lemma: pullback of dual sheaf}, the functor of points of $\underline{\Hom}_{\mathcal{O}_X}(L, \mathcal{O}_X)$ has fiber over a point $x\in X(A)$ given by $\Hom_A(L_x, A)$.  We may then define the action by saying $g\in G(A)$ acts on $\ell\in \Hom_A(L_x, A)$ by
\[
(g\ell)(v) = g(\ell(g^{-1}v)) \textrm{ for } v\in L_{gx}.
\]
Let $L^\vee=\underline{\Hom}_{\mathcal{O}_X}(L, \mathcal{O}_X)$. We claim that $L^\vee$ is an equivariant line bundle. In fact, from the definition, one can easily check that $g\ell$ is in the fiber over $gx$ and $(L^\vee)_{x} \to (L^\vee)_{gx}$ is $A$-linear. The only nontrivial part is the naturality. In other words, for $f:A \to B$, we have to show that the following diagram commutes 
\[
\begin{tikzcd}
G(A) \times L^\vee(A) \arrow[r,"\rho_A"] \arrow[d," \varphi \times \psi",swap] & L^\vee(A) \arrow[d,"\psi"] \\
G(B) \times L^\vee(B) \arrow[r,"\rho_B"] & L^\vee(B)
\end{tikzcd}
\]
First, for $\ell \in \Hom_A(L_x,A)$, we have
\[
L_{f^*x}=L_x\otimes_A B \textrm{ and }\psi(\ell)=f^*\ell=\ell\otimes \text{id} :L_x\otimes_AB \to B. 
\]
Hence, we have
\[
\psi\circ\rho_A (g,\ell)=f^*(g\ell)=(g\ell)\otimes \text{id}.
\]
So, for $w=\sum v_i\otimes b_i \in L_{f^*x}$ from the linearity, we have
\[
\psi\circ\rho_A (g,\ell)(w)=\sum_i g(\ell(g^{-1}v_i)) \otimes b_i.
\]
On the other hand, we have
\[
f^*g(w)=f^*g(\sum_iv_i\otimes b_i)=\sum_i gv_i \otimes b_i.
\]
Hence, we have
\[
\rho_B\circ (\varphi\times\psi)(g,\ell)(w) = (f^*gf^*\ell)(w) = f^*g(f^*\ell(f^*g^{-1}w))=f^*g(f^*\ell(\sum_ig^{-1}v_i\otimes b_i))
\]
\[
=f^*g (\sum_i \ell(g^{-1}v_i)\otimes b_i) = \sum_i g(\ell(g^{-1}v_i))\otimes b_i.
\]
This shows that $\psi\circ\rho_A =\rho_B\circ (\varphi\times\psi)$.

The canonical isomorphism
\[
L^\vee \otimes L\rightarrow \mathcal{O}_X
\]
sends $(\ell, v)$ to $\ell(v)$.  Since $(g\ell)(gv) = g(\ell(v))$, this isomorphism is equivariant, so the line bundle $L^\vee$ is the inverse under tensor product. 
\end{proof}

Later, we will consider toric vector bundles in the affine setting and then glue them to obtain toric vector bundles in the general setting. An arbitrary vector bundle is trivial over an affine toric variety \cite{Gub87}. For toric vector bundles one has the following statement.

\begin{pro}\cite[Proposition~2.2]{payne2008moduli}
Every toric vector bundle on an affine toric variety splits equivariantly as a sum of toric line bundles whose underlying line bundles are trivial. 
\end{pro}

The statement also holds in the semiring case, however, one cannot use a similar argument as in the classical proof. In fact, we prove an analogous theorem (Theorem \ref{theorem: torus-equivariant}) in Section \ref{section:torus-equivariant} saying that under the irreducibility assumption, any equivariant vector bundle which is trivial as a vector bundle splits equivariantly over an idempotent semifield.

Note that we can understand $S_n$ in our exact sequence in Proposition \ref{corollary: semidirect} as a scheme over $\mathbb{N}$. See Appendix \ref{sn as a functor}. This fact be used in the next section.

\section{Splitting of $G$-equivariant vector bundles on irreducible schemes}\label{section:torus-equivariant}

In this section, we study properties of $G$-equivariant vector bundles on irreducible schemes over an idempotent semifield $K$, where $G$ is an irreducible algebraic group over $K$. In \cite{JMT20}, we proved that any vector bundle on a scheme over an idempotent semifield, satisfying a certain local condition splits. In this section, we prove that for an irreducible algebraic group $G$ over an idempotent semifield $K$, any $G$-equivariant vector bundle, which is trivial as a vector bundle, on an irreducible scheme $X$ over $K$ equivariantly splits. Along with other results, this result will be used to study toric vector bundles on toric schemes in later sections.

In what follows, by a scheme we mean a scheme over $\mathbb{N}$, i.e., semiring scheme, unless otherwise stated. 

\begin{mydef}\label{definition: irreducible}
We say a semiring $A$ is \textit{irreducible} if the following condition holds: for any $x,y\in A$ if $xy$ is nilpotent, then $x$ or $y$ is nilpotent.
\end{mydef}

The usual argument for rings can be modified to prove the following. 

\begin{lem}
A semiring $A$ is irreducible 
if and only if $\Spec A$ is irreducible. 
\end{lem}

Because we prefer to work with irreducible schemes, we will need the following variant of Yoneda's lemma.  To any scheme $X$, we associate a functor of points $X(A) = \Hom(\mathrm{Spec} A, X)$ mapping the category of irreducible idempotent semirings to the category of sets.

\begin{pro}\label{irreduciblefunctorofpoints}
Let $X, Y$ be schemes over $\mathbb{B}$ (i.e., over some an idempotent semiring). Suppose that $X$ irreducible. There is a bijective correspondence between morphisms $f: X \rightarrow Y$ and natural transformations $f_A: X(A) \rightarrow Y(A)$ of functors of points on the category of irreducible idempotent semirings.
\end{pro}
\begin{proof}
Suppose $f, g: X \rightarrow Y$ are such that $f_A = g_A$ for all irreducible idempotent semirings $A$.  Take an open affine cover $\{U_i=\Spec A_i\}_{i \in I}$ of $X$, where $A_i$ are idempotent semirings. Observe that $A_i$ is irreducible, because $X$ is irreducible.  The inclusion $\iota_i: U_i \rightarrow X$ is an element $\iota_i \in X(A_i)$.  Then
\[ 
f|_{U_i} = f_{A_i}(\iota_i) = g_{A_i}(\iota_i) = g|_{U_i} \]
Since this holds for all $i$, we have $f = g$.

Now suppose we are given functions $f_A: X(A) \rightarrow Y(A)$ for all irreducible idempotent semirings $A$, and that this is natural in $A$.  Let $U_i$, $A_i$, and $\iota_{i}$ be as before.  For each $i, j$, let $\{V_{ijk}\}$ be a cover of $U_i \cap U_j$ by affine open sets, and write $V_{ijk} = \mathrm{Spec}(B_{ijk})$.  Let $\kappa_{ijk} \in X(B_{ijk})$ be the inclusion $V_{ijk}\rightarrow X$.  Let $\phi_{ijk}: V_{ijk} \rightarrow U_i$ be the inclusion.  Note that $\iota_{i} \circ \phi_{ijk} = \kappa_{ijk}$. 

We obtain maps 
\[
a_i := f_{A_i}(\iota_i): U_i \rightarrow Y, \quad b_{ijk}:= f_{B_{ijk}}(\kappa_{ijk}): V_{ijk} \rightarrow Y.
\]
Naturality implies that
\[ 
a_i |_{V_{ijk}} = a_i \circ \phi_{ijk} = f_{B_{ijk}}(\iota_i \circ \phi_{ijk}) = b_{ijk}. 
\]
Similarly, the restriction of $a_j$ to $V_{ijk}$ is $b_{ijk}$.  So $a_i$ and $a_j$ agree on $V_{ijk}$ for all $k$, which implies they agree on $U_i \cap U_j$.  Since we have a collection of morphisms $a_i: U_i \rightarrow Y$ that agree on overlaps, we obtain a morphism $g:X \to Y$ such that $g|_{U_i} = a_i$. In particular, we have
\[
g_{A_i}=f_{A_i}.
\]
It remains to show that $g_A = f_A$ for all irreducible $A$.  We fix a morphism $x: \Spec A \rightarrow X$ and show 
that
\[
f_A(x) = g_A(x).
\]
Consider first the case where $x$ factors through some $U_i$, i.e., $x = \iota_i \circ y$ for some $y \in U_i(A)$.  By naturality, we have
\[
f_A(x) = f_A(\iota_i \circ y) = f_{A_i}(\iota_i) \circ y, 
\]
and similarly for $g$.  Since $g_{A_i} = f_{A_i}$ by construction, $f_A(x) = g_A(x)$ for such $x$.

We now consider the general case.  Cover each $x^{-1}(U_i)$ by open affines, to get a cover $\{W_{ij}=\Spec C_{ij}\}$ of $\Spec A$ such that $x(W_{ij})\subseteq U_i$.  Let $\mu_{ij}: W_{ij} \rightarrow \Spec A$ be the inclusion.  Let $x_{ij}: W_{ij} \rightarrow X$ be given by
\[
x_{ij} = x \circ \mu_{ij},\quad \text{i.e., } x_{ij} = x|_{W_{ij}}.
\]
  We have that $x_{ij}$ factors through $U_i$, and hence 
 \[
 f_{C_{ij}}(x_{ij}) = g_{C_{ij}}(x_{ij}).
 \] 
By rewriting in terms of $x$ and $\mu_{ij}$ and by using naturality, we have
\[ 
f_A(x) \circ \mu_{ij} = f_{C_{ij}}(x \circ \mu_{ij}) = g_{C_{ij}}(x \circ \mu_{ij}) = g_A(x) \circ \mu_{ij}.
\]
It follows that the maps $f_A(x), g_A(x): \Spec A \rightarrow Y$ agree on each $W_{ij}$, and thus are equal.
\end{proof}

Our next task is to construct some examples of irreducible schemes over any additively idempotent semiring.  We will start with some affine examples.  For this, we will need the following lemma.
\begin{lem}\label{lemma: nilradical}
Let $A$ be an idempotent semiring and $I$ be the nilradical of $A$, i.e., 
\[
I=\{x \in A \mid x^n=0\}.
\]
Then $I$ is a subtractive ideal, i.e. if $a+b \in I$, then $a \in I$.  Moreover, if $A$ is irreducible then $I$ is prime.
\end{lem}
\begin{proof}
$I$ is an ideal by the same proof as the classical case.

Since $a+b \in I$, there exists $n \in \mathbb{N}$ such that $(a+b)^n=0$. Since $A$ is idempotent, we have
\[
(a+b)^n = \sum a^kb^{n-k} =0, 
\]
which implies that $a^kb^{n-k}=0$ for all $k$, as any idempotent semiring is zero-sum-free. In particular $a^n=0$, showing that $a \in I$.  Thus $I$ is subtractive.

The claim about primality follows immediately from the definition of an irreducible semiring.
\end{proof}

\begin{rmk}For many of our results on irreducible idempotent semirings, the only place we need idempotence is to establish that if $I$ is the nilradical then $A / I$ is zero-sum-free and $x + y \in I$ implies $x \in I$.  These facts are trivially true for zero-sum-free semirings with no nontrivial nilpotents.  Thus many of the below results apply in this setting as well.  In particular if $A$ is zero-sum-free (e.g. $A = \mathbb{N}$), then any monoid algebra $A[M]$ has no nontrivial zero-divisors.
\end{rmk}

The only tool for establishing a semiring is irreducible that will be needed for our focus on the toric case is Proposition~\ref{proposition: irreducible affine}. However, we will prove other irreducibility results for the sake of making the theory more general.

Recall that for an $\mathbb{F}_1$-algebra $M$ and a semiring $A$, by the base change $A\otimes_{\mathbb{F}_1}M$ we mean the monoid semiring $A[M]/\sim$ where $\sim$ identifies $0 \in M$ and $0 \in A$. See Section \ref{section: tropical toric schemes}. 

\begin{pro}\label{proposition: irreducible affine}
Let $A$ be an irreducible idempotent semiring, and let $M$ be an irreducible $\mathbb{F}_1$-algebra.  Then the semiring $A \otimes_{\mathbb{F}_1} M$  is irreducible.   Moreover, if neither $A$ nor $M$ have nontrivial zero-divisors, then neither does $A \otimes_{\mathbb{F}_1} M$.  In particular for any monoid $M$, $A[M]$ is irreducible, and if $A$ has no nontrivial zero-divisors, $A[M]$ does not either. 
\end{pro}
\begin{proof}In the monoid case, we can adjoin a zero element to get to the case of an $\mathbb{F}_1$-algebra with no non-trivial zero-divisors.

Let $x, y\in A \otimes_{\mathbb{F}_1} M$ be such that $xy$ is nilpotent.  We wish to show either $x$ or $y$ is nilpotent.  Write
\[ 
x = \sum_{m\in M} x_m [m], 
\]
where $x_m \in A$, and we have a similar expansion for $y$.  Then
\[ \sum_{m, n \in M} x_m y_n [mn] \]
is nilpotent. Thus, from Lemma \ref{lemma: nilradical}, $x_m y_n [mn]$ is nilpotent for all $m, n\in M$.  Then there is some $k$ such that $(x_m y_n)^k [(mn)^k] = 0$.  Since non-zero elements of $M$ form a basis for $A \otimes_{\mathbb{F}_1} M$, either $[(mn)^k] = 0$ or $(x_m y_n)^k = 0$.  

If for all $n\in M$ either $y_n$ is nilpotent  or $[n]$ is nilpotent, then $y = \sum y_n [n]$ lies in the nilradical, and we are done.  Otherwise choose $n\in M$ such that neither $y_n$ nor $[n]$ is nilpotent and choose $m\in M$ arbitrarily.  In the case that $x_m y_n$ is nilpotent, since $y_n$ is not, irreducibility of $A$ tells us $x_m$ is nilpotent.  In the other case where $[mn]$ is nilpotent, irreducibility of $M$ tells us $[m]$ is nilpotent.  Thus for all $m\in M$, $x_m [m]$ is nilpotent and thus $x = \sum x_m [m]$ is nilpotent. This shows irreducibility of $A\otimes_{\mathbb{F}_1}M$.

For the second assertion, if neither $A$ nor $M$ have non-trivial zero-divisors (in particular all nilpotents are trivial), suppose $xy = 0$ (in particular $xy$ is nilpotent).  We may follow the argument above, and in first case of the previous paragraph we obtain that for all $n\in M$, either $y_n$ is nilpotent (so $y_n = 0$) or $[n]$ is nilpotent (so $[n] = 0$).  Either way $y = \sum y_n [n] = 0$.  In the remaining case, as in the previous paragraph we get that for each $m$, $x_m$ or $[m]$ is nilpotent, and as in the previous case we get $x = \sum x_m [m] = 0$.
\end{proof}

Note that the above result is false over rings, as seen in the case $A = \mathbb{C}$ and $M = \mathbb{Z}/2\mathbb{Z}$.  Also note that while we generally assume commutativity, the above result holds for noncommutative $\mathbb{F}_1$-algebras as well.

The following lemma will be useful for studying irreducibility of quotient semirings.
\begin{lem}\label{lemma: irreducible quotient}
Let $A$ be an irreducible idempotent semiring.  Let $\sim$ be a congruence generated by a collection of pairs $(x_i, y_i)$ with the property that for each $i$ in the index set, either both $x_i$ and $y_i$ are nilpotent or both are non-nilpotent.  Then if $x \sim y$, either both $x$ and $y$ are nilpotent or both are non-nilpotent.  Moreover, the quotient $A / \sim$ is irreducible.  If $A$ has no nontrivial zero-divisors, then neither does $A / \sim$.
\end{lem}
\begin{proof}
Let $P$ be the nilradical, which is a subtractive prime ideal by Lemma \ref{lemma: nilradical}.  Then $A / P$ is a zero-sum-free semiring with no nontrivial zero-divisors.  This implies the map $\phi: A/ P \rightarrow \mathbb{B}$ given by $\phi(0) = 0$ and $\phi(x) = 1$ for $x \neq 0$ is a homomorphism.  We consider instead the epimorphism $\psi: A \rightarrow \mathbb{B}$ obtained by composing with the quotient map; explicitly $\psi(x)$ is $0$ or $1$ according to whether $x$ is nilpotent.

Write $x \equiv y$ if $x$ and $y$ are both nilpotent or both non-nilpotent.  Since this is equivalent to $\psi(x) = \psi(y)$, $\equiv$ is a congruence.  By assumption it contains all generators of the congruence $\sim$.  Thus if $x \sim y$ then $x \equiv y$, which is the first part of the result.

For the second part of the result, suppose $\bar{x} \bar{y} \in A / \sim$ is nilpotent.  Pick lifts $x, y\in A$, and observe that for some $k$, $(xy)^k \sim 0$.  Thus $(xy)^k$ is nilpotent, which implies $xy$ is nilpotent.  By irreducibility, either $x$ or $y$ is nilpotent, which implies the same for $\bar{x}$ or $\bar{y}$.

If $A$ has no nontrivial zero-divisors, suppose $\bar{x} \bar{y} = 0$.  It is nilpotent, so we may follow the previous paragraph to obtain that $x$ or $y$ is nilpotent.  But $A$ has no nontrivial nilpotents, so $x = 0$ or $y = 0$ and hence $\bar{x} = 0$ or $\bar{y} = 0$.
\end{proof}

One application of the above result is to the tropicalization of a variety not contained in the union of the coordinate hyperplanes. Note that a monoid does not have an absorbing element $0$, and hence it is always irreducible as it cannot have nilpotent elements. 

\begin{pro}\label{proposition: tropical hyperplane}
Let $K$ be a valued field, $M$ be a monoid, and $I\subseteq K[M]$ be an ideal.  Let $R = K[M] / I$, where $I$ does not contain any element of $M$.  Let $A$ be the tropicalization of $R$, i.e. the quotient of $\mathbb{T}[M]$ by bend relations corresponding to elements of $I$.  Then $A$ has no nontrivial zero-divisors, in particular $A$ is irreducible.
\end{pro}
\begin{proof}
Observe that $\mathbb{T}[M]$ has no nontrivial zero-divisors by Proposition \ref{proposition: irreducible affine}.  Let $\sim$ be the congruence of $\mathbb{T}[M]$ generated by bend relations.  Let $v_K$ be the valuation on $K$.

Let $f \in I \subseteq K[M]$.  Write $f = \sum a_i [m_i]$ with the $a_i$ all nonzero. If $f$ has no terms (i.e., $f = 0$), then the set of bend relations for $f$ is empty and the bend congruence is just the diagonal congruence. If there is only one term then $[m_1] = f / a_1 \in I \cap M$ contradicting the hypothesis. Thus we may assume there are at least two nonzero terms.  

For each $i$, consider the corresponding bend relation
\[ 
\sum_j v_K(a_j) [m_j] \sim \sum_{j \neq i} v_K(a_j) [m_j]. 
\]
Since $f$ contains more than one term, both sums are nonempty.  Moreover each term is nonzero.  So the bend relation identifies a non-zero element with another non-zero element.  Since $\mathbb{T}[M]$ is reduced, the bend relation identifies two non-nilpotents.  Thus the result follows from Lemma \ref{lemma: irreducible quotient}.
\end{proof}

\begin{rmk}
Note that one may also prove Proposition \ref{proposition: tropical hyperplane} by using \cite[Corollary 6.11]{JMT20}.    
\end{rmk}

Next we show that a tensor product of irreducible idempotent semirings is irreducible.  We do so by viewing the tensor product as a quotient of a monoid algebra.

\begin{pro}\label{proposition: tensor product is irreducible}Let $R$ be an idempotent semiring, and let $A, B$ be irreducible $R$-algebras.   Assume that the maps $R \rightarrow A$ and $R \rightarrow B$ have trivial kernels (which is automatic if $R$ is a semifield).\footnote{Here, a kernel does not mean an equalizer, but it means simply the inverse image of $0$.} Then $A \otimes_R B$ is irreducible. 
\end{pro}
\begin{proof}If we forget the addition structure on $B$, we can view it as an irreducible $\mathbb{F}_1$-algebra.  We can then construct the $A$-algebra $A \otimes_{\mathbb{F}_1} B$.  This is irreducible by Proposition \ref{proposition: irreducible affine}. By abuse of notation we denote it $A[B]$, since this makes the asymmetry in how we treat $A$ and $B$ more clear.

Define a congruence $\sim$ on $A[B]$ generated by $[b_1] + [b_2] \sim [b_1 + b_2]$ for $b_1, b_2 \in B$ and $r \sim [r]$ for $r \in R$ (i.e., $\sim$ identifies the copy of $R$ for $B$ with the one for $A$).  Because the structure maps have trivial kernel, $r^k = 0$ if and only if $[r]^k = 0$, so $r$ is nilpotent if and only if $[r]$ is nilpotent.  If $[b_1 + b_2]$ is nilpotent, there is some $k$ such that $[(b_1 + b_2)^k] = 0$ and hence $(b_1 + b_2)^k = 0$, i.e., $b_1 + b_2$ is nilpotent.  By Lemma \ref{lemma: nilradical}, this implies $b_1$ and $b_2$ are nilpotents in $B$ so $[b_1], [b_2] \in A[B]$ are nilpotent, and hence $[b_1] + [b_2]$ is nilpotent.  Conversely if $[b_1] + [b_2]$ is nilpotent, Lemma \ref{lemma: nilradical} implies each term is nilpotent, so $b_1, b_2$ are nilpotent, and hence $[b_1 + b_2]$ is nilpotent.  So no generator of $\sim$ identifies a nilpotent element with a non-nilpotent.  Thus $A[B] / \sim$ is irreducible by Lemma \ref{lemma: irreducible quotient}.

$A[B]$ has the universal property that a homomorphism $A[B] \rightarrow S$ corresponds to a pair of a semiring homomorphism $A \rightarrow S$ and an $\mathbb{F}_1$-algebra homomorphism $B \rightarrow S$ (viewed $S$ as a multiplicative monoid with $0$).  A homomorphism $(A[B] / \sim) \rightarrow S$ corresponds to a homomorphism $\phi: A[B] \rightarrow S$ such that $\phi(r) = \phi([r])$ for all $r\in R$ and $\phi([b_1]) + \phi([b_2]) = \phi([b_1 + b_2])$ for all $b_1, b_2 \in B$.  This in turn corresponds to a pair of a homomorphism $\phi_A: A \rightarrow S$ and an $\mathbb{F}_1$ algebra homomorphism $\phi_B: B\rightarrow S$ such that $\phi_A$ and $\phi_B$ agree on $R$ and such that $\phi_B$ preserves addition (i.e. $\phi_B$ is in fact a semiring homomorphism).

The universal property of tensor products of semirings says that a homomorphism $A \otimes_R B \rightarrow S$ corresponds to a pair of homomorphisms $A \rightarrow S$ and $B\rightarrow S$ which agree on $R$.  Thus $A \otimes_R B \cong A[B] / \sim$.  Since $A[B] / \sim$ is irreducible, so is $A \otimes_R B$.
\end{proof}

\begin{rmk}
    The above result fails for classical schemes, e.g., $\mathbb{C} \otimes_\mathbb{R} \mathbb{C}$ is not even connected.
\end{rmk}

The following lemma is useful in order to prove similar irreducibility results for non-affine schemes.

\begin{lem}\label{lemma: irreducible graph}
Let $X$ be a sober space with an open cover $\{U_i\}$ such that each $U_i$ is irreducible. Construct a graph whose vertices correspond to sets of the open cover. Two vertices $i$ and $j$ are connected by an edge if $U_i \cap U_j \neq \emptyset$. This graph is connected if and only if $X$ is irreducible. 
\end{lem}
\begin{proof}
If $X$ is irreducible, clearly the graph is connected. Conversely, suppose that the graph is connected. If $U_i$ and $U_j$ are nondisjoint, then $U_i \cap U_j$ is an open subset of the irreducible space $U_i$, and therefore contains the generic point of $U_i$.  Similarly, it contains the generic point of $U_j$.  But $U_i \cap U_j$ is irreducible (as it is an open subset of an irreducible space), so contains a unique generic point.  Thus, if $i$ and $j$ are connected by an edge, $U_i$ and $U_j$ share the same generic point.  Since the graph is connected, all of the $U_i$ share the same generic point $\eta$.  

Let $U$ be an open set.  There is some $i$ such that $U \cap U_i \neq \emptyset$.  Since $U \cap U_i$ is an open set in the irreducible space $U_i$, it contains the generic point $\eta$.  Since all open sets contain $\eta$, no two open sets can be disjoint.
\end{proof}

As an application of this result we show that the product of irreducible schemes is irreducible.

\begin{pro}\label{proposition: irreducible product}
Let $X$ and $Y$ be irreducible schemes over an idempotent semifield $K$.  Then $X \times_K Y$ is irreducible.
\end{pro}
\begin{proof}
First, we note that the construction of the fibered product in the semiring setting is the same as in the classical case of rings. One can check the same arguments go through mutatis mutandis.

Now, from the construction of the fibered product $X\times_KY$, if $U$ is an open subset of $X$, then the preimage of $U$ under $\pi_1: X \times_K Y \rightarrow X$ is isomorphic to $U \times_K Y$.  Moreover the product of affine schemes corresponds to the tensor product of semirings.

Let $U_i = \Spec A_i$ be an open affine cover of $X$ and $V_j = \Spec B_j$ be an open affine cover of $Y$. $X \times_K Y$ has an open cover by open subschemes of the form $\pi_1^{-1}(U_i) = U_i \times_K Y$, and these in turn may be covered by open subschemes of the form $U_i \times_K V_j = \Spec A_i \otimes_K B_j$.  By Proposition \ref{proposition: tensor product is irreducible}, each $U_i \times_K V_j$ is irreducible.

Next, we consider the graph associated to this open cover (with an edge linking open subsets if they are not disjoint).  This graph has vertices of the form $(i, j)$ where $i$ and $j$ belong to the index sets of the covers $\{ U_i \}$ and $\{ V_j \}$ respectively.

If $U_i$ and $U_k$ are not disjoint, then $\pi_1^{-1}(U_i \cap U_k) \cong (U_i \cap U_j) \times_K Y$, so the left side is nonempty.  In particular, $U_i \times_K Y$ and $U_j \times_K Y$ are not disjoint open subschemes.  We may apply this with $V_j$ in place of $Y$ to obtain that $U_i \times_K V_j$ and $U_k \times_K V_j$ are not disjoint open subschemes of $X \times_K V_j \subseteq X \times_K Y$.     So for each $j$ we get an edge linking $(i, j)$ to $(k, j)$. Similarly if $V_j$ and $V_l$ are not disjoint then for each $i$, $U_i \times_K V_j$ intersects $U_i \times_K V_l$.  So we get edges linking $(i, j)$ to $(i, l)$.  

Thus the product graph of the graphs associated to the covers $\{ U_i \}$ and $\{ V_j \}$ is a spanning subgraph of the graph associated to the product cover.  Since the product of connected graphs is connected, the result follows follows from Lemma \ref{lemma: irreducible graph}.
\end{proof}

Next we will establish irreducibility of toric schemes over an idempotent semiring. 

Let $\Delta$ be a fan. Construct a graph $G_\Delta$ whose vertices correspond to cones in $\Delta$ in which two vertices $\sigma, \tau \in \Delta$ are linked by an edge if $\sigma$ is a face of $\tau$. Then, $G_\Delta$ is connected since $\{0\}$ is a face of any cone in $\Delta$.

\begin{pro}\label{proposition: irreducible toric}
Let $X$ be the toric scheme associated to a fan $\Delta$ over an irreducible idempotent semiring $R$.   Then $X$ is irreducible.
\end{pro}
\begin{proof}
$X$ has an open cover $\{ U_\sigma \}$ (by nonempty sets) indexed by the cones in $\Delta$ in which $U_\sigma \subseteq U_\tau$ whenever $\sigma$ is a face of $\tau$. Now, the result directly follows from Lemma \ref{lemma: irreducible graph} applied to $G_\Delta$. 
\end{proof}

\begin{lem}
Let $X$ be a connected scheme over $\mathbb{N}$.  Let $R$ and $S$ be semirings.  Then any morphism $X \to \Spec (R\oplus S)$ factors through $\Spec R$ or $\Spec S$.
\end{lem}
\begin{proof}
We begin with the affine case where $X = \Spec A$, where we are looking at a homomorphism $\phi: R\oplus S\rightarrow A$.  Let $e = \phi(1, 0)$ and $f = \phi(0, 1)$.  Since $(1, 0) + (0, 1) = (1, 1)$ and $(1, 0)(0, 1) = (0, 0)$, it follows that $(e, f)$ forms an idempotent pair.  By connectedness either $e = 0$ or $f = 0$.  Without loss of generality consider the latter case.  Then $\phi(r, s) = re + sf = re$ so $\phi$ factors through the projection $R\oplus S\rightarrow R$.  Note that unless $X$ is the empty scheme, $\phi$ cannot factor through both projections because then we would have $\phi(1, 1) = \phi(0, 0)$ contradicting that homomorphisms are unital.

Now consider the non-affine case.  Construct a cover by affine open subsets $\{U_i\}$. Let $\phi:X \to \Spec(R\oplus S)$. Then, the restriction $\phi|_{U_i}$ factors through either $\Spec R$ or $\Spec S$. Suppose that for two indices $i$ and $j$, $\phi|_{U_i}$ factors through $\Spec R$, $\phi|_{U_j}$ factors through $\Spec S$, and $U_i\cap U_j \neq \emptyset$. Then, we obtain a contradiction by letting $W$ be an affine open subset of $U_i\cap U_j$ and observing that $\phi|_W$ factors through both $\Spec R$ and $\Spec S$.  

Let $U$ be the union of all $U_i$ such that $\phi|_{U_i}$ factors through $\Spec R$ and $V$ be the union of all $U_i$ such that $\phi|_{U_i}$ factors through $\Spec S$.  We just showed $U$ and $V$ are disjoint, and clearly they cover $X$.  So one of them is empty; without loss of generality it is $V$, and so all $\phi|_{U_i}$ factor through $R$.  Hence so does $\phi$.
\end{proof}

With these preliminaries on irreducible sets out of the way, we turn to studying the obstruction to an equivariant vector bundle being a sum of one-dimensional bundles equivariantly. We will start by considering a sort of a ``vector bundle over a set'' rather than a scheme, and will later use the functor of points viewpoint to reduce the scheme-theoretic case to this one.

Let $A$ be an irreducible idempotent semiring, $G$ a group, and $X$ a $G$-set, and let $n \in \mathbb{N}$.  Let $E = X \times A^n$, and for each $x \in X$, let $E_x = \{x\} \times A^n$, which is an $A$-module in the obvious way.  Let $e_i\in A^n$ be the standard basis and $e_i(x) = (x, e_i)$ be the corresponding basis of $E_x$. Fix a $G$-action on $E$ such that the projection $E\rightarrow X$ is equivariant and for each $g\in G$ and $x\in X$, the induced maps $g: E_x\rightarrow E_{gx}$ on fibers are $A$-linear.

\begin{lem}\label{settheoreticbundledecomposition}
With the same notation as above, there is a unique map $\phi: G \times X\rightarrow S_n$ such that $g e_i(x)$ is a scalar multiple of $e_{\phi(g, x)(i)}(gx)$ for all $(g, x)\in G\times X$.
\end{lem}
\begin{proof}
Fix $g\in G$ and $x\in X$.  We consider two bases of $E_{gx}$: we have the basis $\{e_i(gx)\}$ and the basis $\{g e_i(x)\}$.  Because $A$ is an irreducible (hence connected) zero-sum-free semiring and $E_{gx}$ is a free $A$-module, the basis is unique up to reordering and scalar multiplication (see Remark \ref{remark: exact sequence remark}). Thus there is some permutation $\sigma$ and some units $a_i\in A^\times$ such that $g e_i(x) = a_i e_{\sigma(i)}(gx)$.  The desired map $\phi$ is the map that sends $(g, x)$ to this $\sigma \in S_n$. Note that $\phi$ is unique because a basis vector cannot be a multiple of another element of the basis.
\end{proof}

To understand the significance of Lemma \ref{settheoreticbundledecomposition}, note that if the map $\phi$ in the lemma sends every element to $1\in S_n$, then $g e_i(x)$ is a scalar multiple of $e_i(gx)$.  Thinking of $E$ as essentially an equivariant vector bundle this would mean that $X \times \mathrm{span}(e_i)$ is an equivariant line subbundle, and $E$ would decompose into equivariant line bundles.  Of course this is all taking place over sets rather than schemes, so our definition of equivariant vector bundle does not quite apply.

Morally we expect that in the scheme-theoretic case, we should have a similar map $G \times X \rightarrow S_n$ acting as the obstruction to decomposing an equivariant vector bundle which is trivial as a vector bundle into line bundles.  We will do this only in the case where $X$ and $G$ are irreducible.  Note that furthermore, under these irreducibility assumptions, we should intuitively expect the map to be trivial for connectedness reasons.

Now, fix an irreducible scheme $X$ over an idempotent semifield $K$ and a vector bundle $E$ on $X$. For any irreducible idempotent semiring $A$ and $x:Y=\Spec A \to X$, we let $E_x=x^*E(\Spec A)$ as before. Let's recall some more notations. We have
\[
E(A)=\bigsqcup_{x \in X(A)}E_x. 
\]
Also, there is a canonical projection
\[
\pi_A:E(A) \to X(A), 
\]
where $\pi_A(y)=x$ if and only if $y \in E_x$.

\begin{mythm}\label{theorem: torus-equivariant}
Let $X$ be an irreducible scheme over an idempotent semifield $K$ and $G$ be an irreducible algebraic group over $K$ acting on $X$.  Let $E$ be a $G$-equivariant vector bundle on $X$ which is trivial as a vector bundle.  Then $E$ is a direct sum of equivariant line bundles.
\end{mythm}
\begin{proof}
For any irreducible idempotent $K$-algebra $A$, we consider the group $G(A)$ of $A$-valued points of $G$ and its action on the sets $X(A)$ and $E(A)$. Pick a basis $e_i$ for $E$, and observe that this induces a basis $x^*(e_i)$ of $E_x$ for each $x\in X(A)$. Note that this is essentially due to the fact that $x^*$ and $x_*$ are adjoint functors as in the case for rings (Lemma \ref{lemma: pullback pushforward adjunction}), implying that if $E=\mathcal{O}_X^n$, then $x^*E=\mathcal{O}_{\Spec A}^n$ since a left adjoint functor preserves colimits. It follows that $E_x=A^n$, and hence we can identify $E(A)$ with $X(A) \times A^n$. 

Now, equivariance of the bundle $E$ means that $G(A)$ acts on $E(A)$ in a manner that is linear on fibers (Definition \ref{definition: equivariant vector bundle} (3)) and that it makes the projection $\pi_A:E(A)\rightarrow X(A)$ equivariant (Definition \ref{definition: equivariant vector bundle} (2)).  Apply Lemma \ref{settheoreticbundledecomposition} to the $G(A)$-action on $E(A)\rightarrow X(A)$.  We obtain for each $A$, a map $\phi_A: G(A) \times X(A) \rightarrow S_n$ such that $g e_i(x)$ is a scalar multiple of $e_{\phi_A(g, x)(i)}(gx)$ for all $(g, x)\in G(A) \times X(A)$.

First, we want to show this map is natural. Fix a morphism $f: \Spec B \rightarrow \Spec A$.  The induced map $f^*: X(A)\rightarrow X(B)$ sends a point $x:\Spec A\rightarrow X$ to $x\circ f$, and the induced map $f^*: G(A)\rightarrow G(B)$ is described similarly.  

Let $\alpha:G \times_K X \to X$ be a $G$-action on $X$. Then for each $K$-algebra $A$, one has
\[
\alpha_A:G(A) \times X(A) \to X(A).
\]
Since the action $\alpha$ is a morphism of schemes, it induces a natural transformation of functors of points, i.e., for $g \in G(A)$ and $x \in X(A)$, one has
\begin{equation}\label{eq: natural}
f^*(gx) = f^*(\alpha_A(g, x)) = \alpha_B(f^*(g,x)) = \alpha_B(f^*(g), f^*(x)) = f^*(g)f^*(x).
\end{equation}

The induced map $f^*: E(A)\rightarrow E(B)$ sends $e_i(x)$ to
\begin{equation}
f^*(e_i(x))= f^* x^*(e_i) = e_i(x\circ f) = e_i(f^*(x)).
\end{equation}

By \eqref{item:naturality of action on bundle} of Definition \ref{definition: equivariant vector bundle}, similar to \eqref{eq: natural}, for any $e\in E(A)$ and $g\in G(A)$ we have 
\[
f^*(g) f^*(e) = f^*(ge).
\]
Observe that
\begin{equation} \label{eq: eq17}
f^*(g) f^*(e_i(x)) = f^*(g e_i(x))
\end{equation}
is a scalar multiple of
\begin{equation} \label{eq: eq18}
f^*(e_{\phi_A(g, x)(i)}(gx)) = e_{\phi_A(g, x)(i)}(f^*(gx)) = e_{\phi_A(g, x)(i)}(f^*(g)f^*(x))
\end{equation}
because $ge_i(x)$ is a multiple of $e_{\phi_A(g, x)(i)}(gx)$. We claim the following:
\begin{equation}\label{eq: naturality}
\phi_B(f^*(g), f^*(x)) = \phi_A(g, x),
\end{equation}
which is precisely the naturality result we want to prove. Note that $f^*$ does not occur on the right side because $S_n$ is a constant functor, so $f^*$ is the identity here. 

To see \eqref{eq: naturality}, define for $g' \in G(B)$ and $x' \in X(B)$,
\begin{equation}
\psi(g',x')=\begin{cases}
    \phi_B(g',x') & \textrm{ if } (g', x')  \neq   (f^*(g), f^*(x)),\\
    \phi_A(g, x)  & \textrm{ if } (g', x')  =  (f^*(g), f^*(x)).
\end{cases}
\end{equation}
Then for any $g' \in G(B)$ and $x' \in X(B)$, we see that $g' e_i(x')$ is a multiple of $e_{\psi(g', x')(i)}(g'x')$; when $(g', x') = (f^*(g), f^*(x))$ this is because \eqref{eq: eq17} is a multiple of \eqref{eq: eq18}, and otherwise this holds because it is true for $\phi_B$.  Now the uniqueness in the defining property of $\phi_B$ (in Lemma \ref{settheoreticbundledecomposition}) implies $\phi_B = \psi$, so in particular $\phi_B(f^*(g), f^*(x)) = \phi_A(g, x)$, as claimed.

We thus have constructed a natural map $(G \times_K X)(A) \rightarrow S_n = S_n(A)$.  Moreover, it follows from Proposition \ref{proposition: irreducible product} that $G \times_K X$ is irreducible. Therefore, from Lemma \ref{irreduciblefunctorofpoints}, this natural transformation comes from a morphism of schemes $\phi: G \times_K X \rightarrow S_n$.  Since $G\times_K X$ is connected, and $S_n$ is the spectrum of a direct sum of $n!$ copies of $K$, $\phi$ is constant, i.e., there is some $\sigma \in S_n$ such that $\phi$ factors through the inclusion $\Spec K\rightarrow S_n$ corresponding to the point $\sigma$.  The same must be true of the corresponding maps on functors of points.  Thus $\phi_A(g, x) = \sigma$ for all $(g, x)\in G(A) \times X(A)$. So we obtain that for any $(g, x)\in G(A) \times X(A)$, $g e_i(x)$ is a scalar multiple of $e_{\sigma (i)}(gx)$. By applying this to $g = 1$, we must have $\sigma = 1$, so $g e_i(x)$ is a scalar multiple of $e_i(gx)$. 

Let $L_i$ be the line bundle contained in $E$ given as a sheaf by $L_i(U) = \mathrm{span}(e_i|_U)$.  Its functor of points is such that $L_i(A)$ is the set of pairs $(x, e)$ such that $x\in X(A)$, and $e\in \mathrm{span}(e_i(x))\subseteq E_x$.  We saw $g e_i(x)$ is a scalar multiple of $e_i(gx)$ so belongs to $L_i(A)$, and the same is true for $ge$ where $e\in L_i(A)$ since any such $e$ is a multiple of $e_i(x)$ for some $x$.  Since $ge\in L_i(A)$ for $g\in G(A)$ and $e\in L_i(A)$, it follows that the $G(A)$ action on $E(A)$ descends to an action on $L_i(A)$, making $L_i(A)$ equivariant.  Even without equivariance, we of course have $E = \bigoplus L_i$, and since $L_i$ is equivariant, this implies the result.
\end{proof}

\begin{cor}\label{corollary: torus-equivariant trivial vector bundles are sums of line bundles}
Let $X$ be a toric scheme over an idempotent semifield $K$ and $G$ be the algebraic torus over $K$. Let $E$ be an equivariant vector bundle which is trivial as a vector bundle.  Then $E$ is a direct sum of equivariant line bundles.
\end{cor}

\begin{rmk}
In the non-affine case there are non-trivial (toric) line bundles both in the classical case \cite{PS14} and the tropical case \cite{jun2019picard}.
\end{rmk}

\section{Classification of equivariant vector bundles}\label{section: classification of equivariant vb}

Throughout this section, unless stated otherwise, we let $K$ denote either a semiring or a monoid (with an absorbing element). As in the case for rings and semirings, by a monoid scheme over $K$, when $K$ is a monoid, we mean a monoid scheme $X$ with a map to $\Spec K$. Most of the results below hold for both semiring schemes and monoid schemes, with the same proof for both cases.  And we will need both cases in order to compare equivariant line bundles on toric schemes over idempotent semirings with those on the underlying monoid schemes.

The following result justifies thinking of morphisms $X\rightarrow \text{GL}_1$ as units of $X$.

\begin{lem}\label{units as morphisms}
Let $X$ be a scheme (resp.~monoid scheme) over a semiring (resp.~monoid) $K$. Let $\text{GL}_1 = \Spec K[t, t^{-1}]$.  Then there is a natural isomorphism $\mathrm{Hom}_{\Spec K}(X, \text{GL}_1) \simeq \Gamma(X, \mathcal{O}_X^\times)$.
\end{lem}
\begin{proof}
We prove the case for semirings. The case for monoids is similar. 

Note that we have a canonical global section $t \in \Gamma(GL_1, \mathcal{O}_{\text{GL}_1}^\times)$, so $f: X\rightarrow \text{GL}_1$ induces a global section $f^*t \in  \Gamma(X, \mathcal{O}_X^\times)$.  It is natural, because given $g: Y\rightarrow \text{GL}_1$, the induced map on $\mathrm{Hom}_{\Spec K}(-, \text{GL}_1)$ sends $f$ to $f\circ g$, while the induced map on $\Gamma(-, \mathcal{O}_-^\times)$ sends $f^* t$ to $g^* f^* t = (f\circ g)^* t$.

To show it is an isomorphism, we begin with the affine case, where $X = \Spec A$ for some semiring $A$ and so $\Gamma(X, \mathcal{O}_X^\times) = A^\times$. The universal property of $K[t, t^{-1}]$ asserts that morphisms
\[
f: K[t, t^{-1}]\rightarrow A
\]
are in bijection with units of $A$, and that the corresponding element of $A^\times$ is $f(t) = f^* t$.  So this bijection is the natural map from before.

Now consider the general case.  Fix an open affine cover $\{ U_i \}$ and let $\iota_i: U_i\rightarrow X$ be the inclusion.  For each $i, j$, fix an open affine cover $\{ V_{ijk} \}$ of $U_i \cap U_j$ with $V_{ijk} = V_{jik}$, and let $\phi_{ijk}: V_{ijk} \rightarrow U_i$ be the inclusion.  Given $u \in \Gamma(X, \mathcal{O}_X^\times)$, we obtain elements $u_i = u\mid_{U_i} \in \Gamma(U_i, \mathcal{O}_{U_i}^\times)$ which satisfy the following condition:
\[
\phi_{ikj}^* u_i = \phi_{jik}^* u_j.
\]
We let $f_i$ be the morphism corresponding to $u_i$; note that the restrictions of $f_i$ and $f_j$ to each $V_{ijk}$ agree.  Our goal is to show there is a unique $f: X\rightarrow GL_1$ satisfying $u = f^* t$.

Given a morphism $f: X\rightarrow GL_1$, we have
\[
(f^* t)\mid_{U_i} = \iota_i^* f^* t = (f\mid_{U_i})^* t.
\]
Thus $f^* t = u$ if and only if $(f\mid_{U_i})^* t = u_i$.  By the affine case of this result, this holds if and only if $f\mid_{U_i}$ is the morphism corresponding to $u_i$, i.e., if and only if $f\mid_{U_i} = f_i$.  But since the $f_i$ agree on overlaps they glue together to give a unique morphism satisfying this property.
\end{proof}

Let $K$ be a semiring or a monoid. Let $X$ and $Y$ be schemes over $K$. For each $K$-algebra $A$, there is an one-to-one correspondence between the points $t \in (X\times_KY)(A)$ and the pairs $(x,y) \in X(A) \times Y(A)$ satisfying the following condition
\begin{equation}\label{eq: pairs}
g_X \circ x = g_Y \circ y,    
\end{equation}
where $g_X:X \to \Spec K$ and $g_Y:Y \to \Spec K$ are the structure maps. From now on, we replace each $t \in (X\times_KY)(A)$, with a pair $(x,y) \in X(A) \times Y(A)$ satisfying \eqref{eq: pairs}. 

The following result shows how any equivariant structure on the trivial line bundle $\mathcal{O}_X$ gives rise to a unit on $G \times_K X$.  This will be the key to classifying equivariant structures on $\mathcal{O}_X$.  Note that this unit could potentially depend on the choice of basis vector.  We can just choose the basis vector to be $1\in \mathcal{O}_X$, but since that $1$ is not preserved by line bundle automorphisms, the resulting unit is not completely canonical.

To proceed, we need a notion of equivariant vector bundles on a monoid scheme $X$.
\begin{mydef}
Let $K$ be a monoid, $X$ be a scheme over $K$, and $G$ be a group scheme over $K$ acting on $X$. By an equivariant vector bundle $E$ on $X$, we mean a vector bundle $E$ (locally free sheaf) on $X$ together with an action of $G(A)$ for each $K$-algebra $A$ satisfying the same conditions as in Definition \ref{definition: equivariant vector bundle}. 
\end{mydef}
\begin{rmk}
One can observe that all statements in Section \ref{section: Equivalence between locally free sheaves and geometric vector bundles} except Proposition \ref{proposition: addition morphism} are valid for locally free sheaves on monoid schemes as proofs do not use something specific to semirings. In fact, the proofs for rings can be modified to proof the semiring case and the monoid case. 
\end{rmk}

\begin{lem}\label{construction of unit associated to an equivariant line bundle}
Let $K$ be a semiring or monoid.  Let $X$ be a scheme over $K$ and $G$ be a group scheme over $K$ acting on $X$.  Let $L$ be an equivariant line bundle on $X$ which is trivial as a line bundle and let $v$ be a basis vector for the global sections of  $L$. Then there is a unique morphism $u: G \times_K X \rightarrow GL_1$ such that for each $K$-algebra $A$, each $(g, x) \in (G\times_KX)(A)$, the morphism $u_A: (G\times_KX)(A) \rightarrow GL_1(A)$ satisfies $u_A(g, x) v_{gx} = g v_x$.
\end{lem}

\begin{proof}
We prove the case for semirings. The case for monoids is similar. 

We will construct $u$ via its functor of points. Let $v$ be a basis vector for global sections for $L$.  Since $v_x$ is a basis vector for the free module $L_x$, $gv_x$ is a basis vector for $L_{gx}$.  But $v_{gx}$ is also a basis vector, so there is a unique $a\in GL_1(A)$ such that $a v_{gx} = g v_x$. We let $u_A(g,x)=a$.

We have thus constructed a unique map $u_A$ such that $u_A(g, x) v_{gx} = g v_x$ holds.  To obtain the map $u$ we must show $u_A$ is natural in $A$.  Fix a morphism $f:\Spec B \to \Spec A$. As in the proof of Theorem \ref{theorem: torus-equivariant}, we let the induced maps be $f^*:X(A) \to X(B)$, $f^*:G(A) \to G(B)$, and $f^*:L(A) \to L(B)$. Since $f^*(gx)=f^*(g)f^*(x)$, we have
\[
u_B(f^* g, f^* x)  v_{f^*(gx)} = (f^*g) v_{f^* x}.
\]
Recall that $v_{f^* x}$ is the result of pulling $v$ back along $f^* x : \Spec B \rightarrow X$, while $v_x$ is the pullback along $x: \Spec A \rightarrow X$.  Thus $v_{f^* x}$ is the pullback of $v_x$ along $f: \Spec B \rightarrow \Spec A$.  In other words, $v_{f^*x} = f^* v_x$, where the latter $f^*$ denotes the induced map $f^*:L(A)\rightarrow L(B)$.

Thus we have
\[
u_B(f^* g, f^* x)  f^*v_{gx} = (f^*g) f^* v_{x}.
\]
But applying $f^*$ to both sides of
\[
u_A(g, x) v_{gx} = g v_x
\]
gives 
\[
u_A(g, x) f^* v_{gx} = (f^* g) f^* v_x.
\]
Thus $u_B(f^* g, f^* x) = u_A(g, x)$, which is the desired naturality.
\end{proof}

The next step is to determine which units of $G \times_K X$ arise this way.

\begin{lem}\label{lemma: principal}
With the notation and assumptions of Lemma \ref{construction of unit associated to an equivariant line bundle}, for $(g, h, x) \in (G\times_K G \times_K X)(A)$, we have $u_A(gh, x) = u_A(h, x) u_A(g, hx)$ and $u_A(1, x) = 1$.
\end{lem}
\begin{proof}
To show the first part of the statement, observe that
\[ 
u_A(gh, x) v_{ghx} = gh v_{x} = u_A(h, x) g v_{hx} = u_A(h, x) u_A(g, hx) v_{ghx}. 
\]
Since $v_{ghx}$ is a basis vector, the first part of the statement follows.
Similarly, the second half of the statement follows from $u_A(1, x) v_x = 1 v_x = v_x$.
\end{proof}

\begin{mydef}\label{definition: principal}
With the same notation as above, a map $u:G \times_K X \to \text{GL}_1$ is said to be \emph{principal} if for any $K$-algebra $A$ and for $(g, h, x) \in (G\times_K G \times_K X)(A)$, one has
\begin{equation}\label{eq: eqimp}
u_A(gh, x) = u_A(h, x) u_A(g, hx).
\end{equation}
\end{mydef}

Note that Equation~\eqref{eq: eqimp} implies that $u_A(1, x) = 1$. 

\begin{lem}\label{lemma: from principal to equivariant}
Let $K$ be either a semiring or a monoid. Let $X$ be a scheme over $K$ and $G$ be a group scheme over $K$ acting on $X$.  Let $u: G \times_K X \rightarrow \text{GL}_1$ be a principal morphism.  Then there is an equivariant line bundle and a basis vector $v$ such that $u$ arises from the construction of Lemma \ref{construction of unit associated to an equivariant line bundle}. In other words,
\[
u_A(g, x) v_{gx} = g v_x.
\]
\end{lem}
\begin{proof}
We choose any trivial line bundle $L$ and basis vector $v$.  To construct the equivariant structure, we define the action $\rho_A:G(A) \times L(A) \rightarrow L(A)$ by
\[
g(av_x) = a u_A(g, x) v_{gx}.
\]
By construction this action is linear and satisfies $u_A(g, x) v_{gx} = g v_x$.  We must check that it is actually a group action, that it is compatible with the action on $X$, and that it is natural in $A$.

To see that it is a group action, observe that
\[ g(h(a v_x)) = g(a u_A(h, x) v_{hx}) = a u_A(g, hx) u_A(h, x) v_{ghx} = a u_A(gh, x) v_{ghx} = (gh)(a v_x), \]
which is the desired homomorphism property.  Also $1$ acts via $1(a v_x) = u_A(1, x) a v_x = a v_x$.  The action is compatible with the action on $X$ because $g(a v_x) = a u_A(g, x) v_{gx}$ lies over $gx$.

To show that the action is natural in $A$, let $f:\Spec B \to \Spec A$, and $f^*$ be as in the proof of Lemma~\ref{construction of unit associated to an equivariant line bundle}. We need to show that the following diagram commutes:
\[
\begin{tikzcd}
G(A) \times L(A) \arrow[r,"\rho_A"] \arrow[d," f^*",swap] & L(A) \arrow[d,"f^*"] \\
G(B) \times L(B) \arrow[r,"\rho_B"] & L(B)
\end{tikzcd}
\]
Now for $(g,av_x) \in G(A) \times L(A)$, we have
\[
f^*\rho_A(g,av_x)=f^*(au_A(g,x)v_{gx})=au_A(g,x)f^*(v_{gx})
\]
\[
=au_A(g,x)v_{f^*({gx})}=au_A(g,x)v_{f^*gf^*x}.
\]
On the other hand, 
\[
\rho_B(f^*g,f^*(av_x))=\rho_B(f^*g,af^*(v_x))=\rho_B(f^*g,av_{f^*x})=au_B(f^*g,f^*x)v_{f^*gf^*x}
\]
From the functoriallity of $u$, we have
\[
u_B(f^*g,f^*x)=u_A(g,x), 
\]
implying that
\[
f^*\rho_A(g,av_x)=\rho_B(f^*g,f^*(av_x))=\rho_Bf^*(g,av_x).
\]
\end{proof}

We can now classify actions on the trivial bundle $\mathcal{O}_X$ which make it into an equivariant line bundle.

\begin{pro}\label{Equivariant structures on O_X}
Let $K$ be a semiring or a monoid. Let $X$ be a scheme over $K$ and $G$ be a group scheme over $K$ acting on $X$.  Then there is a one-to-one correspondence between $G$-actions on $\mathcal{O}_X$ which make $\mathcal{O}_X$ an equivariant line bundle and principal morphisms $u: G \times_K X \rightarrow GL_1$.  Moreover, this correspondence is a group isomorphism, where the first set is viewed as a group under tensor product, and the second is a group under pointwise multiplication.
\end{pro}

\begin{rmk}
    Proposition \ref{proposition: equivariant Picard group} shows that the tensor product of equivariant line bundles is an equivariant line bundle, that $\mathcal{O}_X$ with the trivial action is the tensor product identity and that the dual line bundle is the inverse.  Since $\mathcal{O}_X \otimes \mathcal{O}_X = \mathcal{O}_X = \Hom(\mathcal{O}_X, \mathcal{O}_X)$, the set of equivariant line bundles whose underlying line bundle is $\mathcal{O}_X$ is closed under tensor products and duals, and hence forms a group.
\end{rmk}

\begin{proof}[Proof of Proposition~\ref{Equivariant structures on O_X}]
We have already proved that equivariant structures on $\mathcal{O}_X$ uniquely give rise to principal morphisms $u$ (Lemmas \ref{construction of unit associated to an equivariant line bundle} and \ref{lemma: principal}), and that every such $u$ arises in this way (Lemma \ref{lemma: from principal to equivariant}).

Next we prove the uniqueness. Consider two actions $\alpha, \beta: G \times_K \mathcal{O}_X \rightarrow \mathcal{O}_X$ which give rise to the same morphism $u$.  Any element of $L(A)$ can be written $a 1_x$ for some $a\in A$ and $x\in X(A)$.  Under both actions, linearity and the definition of $u$ give 
\[
g(a 1_x) = a g(1_x) = a u_A(g, x) 1_x,
\]
and hence so both actions are the same. In other words, such $u$ arises from a unique equivariant structure on $\mathcal{O}_X$.

Finally, we show this bijective correspondence is a homomorphism.  Let $L, L'$ be equivariant line bundles with basis sections $v$ and $w$ respectively.  Let $u, u', \phi: G \times_K X \rightarrow \mathrm{GL}_1$ correspond to equivariant structures of $L, L', L\otimes L'$ respectively.  Then by definition of $\phi$ and of the action on the tensor product,
\begin{equation}\phi_A(g, x)(v_{gx} \otimes w_{gx}) = g(v_x \otimes w_x) = gv_x \otimes gw_x.
\end{equation}
On the other hand, by definition of $u$ and $u'$,
\begin{equation}u_A(g, x)u'_A(g, x)(v_{gx} \otimes w_{gx}) = gv_x \otimes gw_x.
\end{equation}
So $\phi = uu'$, which implies the correspondence in question is a group homomorphism.
\end{proof}

The relationship in Lemma \ref{units as morphisms}, between units of the global sections of $G\times_K X$ and morphisms $G\times_K X \to \text{GL}_1$, lets us phrase Proposition \ref{Equivariant structures on O_X} entirely in terms of units.

Let $K$ be a semiring or a monoid and $X$ be a scheme over $K$ and $G$ be a group scheme over $K$ acting on $X$.   Let $m, a, p: G\times_K G \times_K X \rightarrow X$ be the maps induced by the multiplication on $G$, the action of $G$ on $X$, and the projection which drops the first factor respectively, and let $i: X \rightarrow G \times_K X$ be the map which inserts the identity element in the first component.  Let $\mu: GL_1 \times_K GL_1\rightarrow GL_1$ denote multiplication.

\begin{cor}\label{cor: units homs}
With the same notation as above, there is a group isomorphism between the group of equivariant line bundle structures on $\mathcal{O}_X$ and the group of units $u \in \Gamma(G\times_K X, \mathcal{O}_{G\times_K X}^\times)$ satisfying $i^*u = 1$ and $\mu \circ (a^* u, p^*u) = m^*(u)$.
\end{cor}

In the case of toric varieties, we can use the above result to relate the tropical setting to the $\mathbb{F}_1$-setting as follows. We will interchangeably use the terms $\mathbb{F}_1$-schemes and monoid schemes. By a group scheme over $\mathbb{F}_1$ we mean a group object in the category of $\mathbb{F}_1$-schemes. In the following, for monoid schemes $X, Y$ we simply write $X\times Y$ instead of $X\times_{\mathbb{F}_1}Y$. Likewise, for semiring schemes $X_R$ and $Y_R$ over a semiring $R$, we write $X_R\times Y_R$ for $X_R\times_{R}Y_R$.

\begin{lem}\label{lemma: tropical thing}
Let $X$ be a toric $\mathbb{F}_1$-scheme and $G$ be a group scheme over $\mathbb{F}_1$. Let $K$ be an idempotent semifield. Then, we have
\begin{equation}\label{eq: split}
\Gamma(G_K\times X_K, \mathcal{O}_{G_K\times X_{K}}^\times) = \Gamma(G\times X, K^\times) \times \Gamma(G\times X, \mathcal{O}_{G\times X}^\times).
\end{equation}
\end{lem}
\begin{proof}
We first consider the case when $X$ and $G$ are affine. Let $X=\Spec M$ and $G=\Spec N$. Then,
\[
G\times X = \Spec (N\otimes_{\mathbb{F}_1}M), \quad G_K=\Spec K[N], \quad X_K=\Spec K[M], 
\]
and since tensor products preserves colimits
\[
G_K \times X_K= (G\times X)_K.
\]
Hence, we have
\[
\Gamma(G_K\times X_K, \mathcal{O}_{G_K\times X_K}^\times) = (K[N\otimes_{\mathbb{F}_1}M])^\times.
\]
It follows from \cite[Proposition 3.4]{jun2019picard} that
\[
(K[N\otimes_{\mathbb{F}_1}M])^\times = K^\times \times (N\otimes_{\mathbb{F}_1}M)^\times.
\]
On the other hand, one can easily check that
\[
\Gamma (G\times X, K^\times)=K^\times, \quad \Gamma(G\times X, \mathcal{O}_{G\times X}^\times) = (N\otimes_{\mathbb{F}_1}M)^\times,
\]
showing the claimed identity for affine $G$ and $X$. 

For general $G$ and $X$, we cover $G_K \times X_K$ by affine open cover $\mathcal{U}_K$ obtained by affine open subsets of $\mathcal{U}=\{U_i\times V_j\}$, where $U_i$ (resp.~$V_j$) are affine open subsets of $G$ (resp.~$X$). Note that this is possible by \cite[Lemma 3.1]{jun2019picard}. Then, as in \cite[Theorem 3.8]{jun2019picard}, by using the affine case, one can obtain isomorphisms of \v{C}ech cochains
\[
\mathbf{C}^k(\mathcal{U}_K,\mathcal{O}_{G_K\times X_K}^\times)    = \mathbf{C}^k(\mathcal{U},K^\times \times \mathcal{O}_{G\times X}^\times),
\]
showing the desired identity \eqref{eq: split}.
\end{proof}

\begin{pro}\label{Correspondence for equivariant structures on O_X}
Let $X$ be a toric $\mathbb{F}_1$-scheme and $G$ be a group scheme over $\mathbb{F}_1$. Let $K$ be an idempotent semifield.  Suppose $G \times X$ is connected.  Then there is an isomorphism between the group of $G$-equivariant line bundle structures on $\mathcal{O}_X$ and the group of $G_K$-equivariant line bundle structures on $\mathcal{O}_{X_K}$.
\end{pro}

\begin{proof}
Let $u \in \Gamma(G_K\times X_K, \mathcal{O}_{G_K\times X_K}^\times)$. We write $u=(u_1,u_2)$ under the isomorphism in Lemma \ref{lemma: tropical thing}.
Under a connectedness assumption, $u_1$ is constant, and hence the identity
\[
u(gh, x) = u(g, hx) u(h, x)
\]
together with the fact that $K^\times$ is a group (so that $x^2 = x$ implies $x =1$) implies that $u_1$ is $1$. It follows that $u$ actually lies in $\Gamma(G\times X, \mathcal{O}_{G\times X}^\times)$. Now, the desired group isomorphism follows from Corollary \ref{cor: units homs}. 
\end{proof}

In the classical setting, a result of Rosenlicht states if $X$ and $Y$ are irreducible varieties, then every morphism $X \times Y \rightarrow GL_1$ is the product of a morphism $X \rightarrow GL_1$ and a morphism $Y \rightarrow GL_1$ as shown in \cite[Proposition 4.1.3]{brionlinearization}.\footnote{
Rosenlicht only proved a corollary (which is pretty different from what we are referencing). } The expository text \cite{brionlinearization} explains how this is used to study equivariant line bundles in the classical setting.  We will prove an analogue for $\mathbb{F}_1$-schemes.

\begin{lem}\label{lemma: factorization of units}
Let $X, Y$ be $\mathbb{F}_1$-schemes which are either affine or  irreducible.  Let $u \in \Gamma(X \times Y, \mathcal{O}_{X\times Y}^\times)$.  Then there exist unique $\alpha \in \Gamma(X, \mathcal{O}_X^\times)$ and $\beta \in \Gamma(Y, \mathcal{O}_Y^\times)$ such that $u = \alpha\beta$.
\end{lem}
\begin{proof}We consider first the case where $X = \Spec A$ and $Y = \Spec B$ are affine.  As a monoid, $A \otimes_{\mathbb{F}_1} B$ is obtained from the (unpointed) monoid $A \times B$ by identifying $(0, b)$ and $(a, 0)$ with $(0, 0)$ for all $a \in A$ and $b \in B$; this is an $\mathbb{F}_1$-algebra with zero element $(0, 0)$. The quotient map $\pi:A \times B \to (A\otimes_{\mathbb{F}_1}B)$ gives rise to a map 
\[
\theta: A^\times \times B^\times = (A \times B)^\times \rightarrow (A \otimes_{\mathbb{F}_1} B)^\times. 
\]
We claim that $\theta$ is bijective. In fact, $\theta$ is injective because the only fiber of the quotient map $\pi$ containing more than one element is $A\times \{0\} \cup \{0\} \times B$, and this fiber contains no units.  Surjectivity is similar: if $x \in (A \otimes_{\mathbb{F}_1} B)^\times$ then we can let $y$ be the unique element of $\pi^{-1}(x)$ and $z$ be the unique element of $\pi^{-1}(x^{-1})$ and observe that $\pi(yz) = 1$.  Since the fiber over $1$ contains a single element, $yz = 1$ so $x = \pi(y)$ for some unit $y$ and $x \in \mathrm{im} \theta$.  Thus $\theta$ is bijective.  Since we can view $\pi: A \times B \rightarrow (A \otimes_{\mathbb{F}_1} B)$ as the multiplication map, $\theta$ is the multiplication map on units, which establishes the result in the affine case.

Next, suppose $Y$ is such that the result holds for any pair $(Z, Y)$ with $Z$ affine. Suppose $X$ is irreducible.  Let $U \subseteq X$ be open affine.  Then by the affine case, for $u \in \Gamma(X\times Y, \mathcal{O}_{X\times Y}^\times)$, there exist unique elements $\alpha_U \in \Gamma(U, \mathcal{O}_X^\times)$ and $\beta_U \in \Gamma(Y, \mathcal{O}_Y^\times)$ such that
\[
u \mid_{U \times Y} = \alpha_U \beta_U.
\]
Let $U_1, U_2$ be open affine subsets of $X$; they are nondisjoint by irreducibility.  Let $V \subseteq U_1 \cap U_2$ be a nonempty open affine subset.  Then restricting the factorization of $u\mid_{U_1\times Y}$ to $V \times Y$ gives
\[
u\mid_{V\times Y} = \alpha_{U_1}\mid_V \beta_{U_1}.
\]
By applying the argument to $U_2$ and using uniqueness of $\alpha_V, \beta_V$, we obtain
\[
\alpha_{U_1}\mid_V = \alpha_V = \alpha_{U_2}\mid_V \textrm{ and } \beta_{U_1} = \beta_V = \beta_{U_2}.
\]
Since $\beta_U$ does depend on $U$, we denote it $\beta$.  Fix the open cover $\{ U_i \}$ of $X$ containing all open affine sets.  We have seen that the sections $\alpha_{U_i}\in \Gamma(U_i, \mathcal{O}_X^\times)$ agree on overlaps, so give rise to $\alpha \in \Gamma(X, \mathcal{O}_X^\times)$.

We now must show $\alpha$ and $\beta$ are the unique elements of $\Gamma(X, \mathcal{O}_X^\times)$ and $\Gamma(Y, \mathcal{O}_Y^\times)$ such that $u = \alpha \beta$.  For any open affine $U\subseteq X$, we have
\[ u \mid_{U\times Y} = \alpha_U \beta_U = \alpha\mid_U \beta = (\alpha\beta)\mid_{U \times Y}\]
Since the sets $U\times Y$ where $U$ ranges over open affine sets forms a cover of $X\times Y$, $u = \alpha\beta$.  If $u = \alpha' \beta'$ then restricting to $U\times Y$ gives $\alpha'\mid_U \beta' = u\mid_{U\times Y} = \alpha\mid_U \beta$.  By the uniqueness of $\alpha_U, \beta_U$, we obtain $\beta' = \beta$ and $\alpha'\mid_U = \alpha\mid_U$.  Since $\alpha$ and $\alpha'$ agree on all open affine subsets, they are equal.  

This establishes the result when $X$ is irreducible and $Y$ is such that the result holds for all pairs $(Z, Y)$ with $Z$ affine.  Since we have proven the affine case, any affine $\mathbb{F}_1$-scheme $Y$ satisfies this hypothesis; hence the result holds when $X$ is irreducible and $Y$ is affine or by symmetry when $X$ is affine and $Y$ is irreducible.  This implies any irreducible scheme $Y$ satisfies the above hypothesis, so the result holds when $X$ and $Y$ are irreducible.
\end{proof}

In the $\mathbb{F}_1$-setting, the above can be used to relate equivariant line bundles to characters of $G$.  We remark that using Proposition \ref{Correspondence for equivariant structures on O_X}, we can also apply Proposition \ref{proposition: equivariant line bundles and characters} to the semiring case. We will need the following lemma. 

\begin{lem}\label{lemma: constant lemma}
Let $X$ be a scheme over $\mathbb{F}_1$ and $f:X \to \text{GL}_1$ be a morphism of $\mathbb{F}_1$-schemes.  If $f_A(x)$ does not depend on $x \in X(A)$ for any $A$, then $f_A(x) = 1$ for all $x \in X(A)$.  
\end{lem}
\begin{proof}
The constancy of $f_A$ implies it factors through the terminal map (as a set map)
\[
t_A: X(A) \rightarrow (\Spec \mathbb{F}_1)(A),
\]
yielding a map $\hat{f}_A: (\Spec \mathbb{F}_1)(A) \rightarrow \text{GL}_1(A)$.  
  
We claim that $\hat{f}$ is natural in $A$. In fact, let $\psi:\Spec B \to \Spec A$, and $\psi_i^*$ be the induced maps on $X$, $\Spec \mathbb{F}_1$, and $\text{GL}_1$. Consider the following diagram: 
\[
\begin{tikzcd}
X(A) \arrow[r,"t_A"] \arrow[d," \psi_1^*",swap] & (\Spec \mathbb{F}_1)(A) \arrow[d,"\psi_2^*"]  \arrow[r,"\hat{f}_A"]& \text{GL}_1(A) \arrow[d,"\psi_3^*"]  \\
X(B) \arrow[r,"t_B"] & (\Spec \mathbb{F}_1)(B) \arrow[r,"\hat{f}_B"] & \text{GL}_1(B)
\end{tikzcd}
\]
The whole square and the left square commute by the definition of $t$ and $f$, i.e., 
\[
\psi_3^*\hat{f}_At_A=\hat{f}_Bt_B\psi_1^* \textrm{ and } \psi_2^*t_A=t_B\psi_1^*.
\]
Hence we have
\[
\psi_3^*\hat{f}_At_A=\hat{f}_B\psi_2^*t_A. 
\]
But, since $t_A$ is surjective, we have
\[
\psi_3^*\hat{f}_A=\hat{f}_B\psi_2^*, 
\]
which is precisely the naturality for $\hat{f}$. Hence, we obtain a morphism $\hat{f}: \Spec \mathbb{F}_1 \rightarrow GL_1$.  But $\mathbb{F}_1^\times $ has only the single element $1$, so $\hat{f}_A$ is the constant map $1$, and hence so is $f$.
\end{proof}

\begin{pro}\label{proposition: equivariant line bundles and characters}
Let $G$ be a group scheme over $\mathbb{F}_1$ and $X$ be an $\mathbb{F}_1$-scheme. Suppose $G$ and $X$ are either affine or irreducible.  Then there is a group isomorphism between
\[
\{\textrm{$G$-actions on the trivial line bundle $\mathcal{O}_X$ which make it an equivariant line bundle}  \}
\]
and
\[
\{\textrm{homomorphisms $\chi: G \rightarrow \text{GL}_1$.}\}
\]
This correspondence is characterized by the property that $G$ acts on the element $1_x$ of the fiber over $X(A)$ via $g 1_x = \chi_A(g) 1_{gx}$.
\end{pro}
\begin{proof}
By Proposition \ref{Equivariant structures on O_X}, it suffices to construct a group isomorphism between homomorphisms $G \rightarrow \text{GL}_1$ and principal maps $u: G \times X \rightarrow \text{GL}_1$ (as in Definition \ref{definition: principal}). In other words, 
\[
u_A(gh, x) = u_A(g, hx) u_A(h, x).
\]
for all $\mathbb{F}_1$-algebra $A$ and all $(g, h, x) \in G(A) \times G(A) \times X(A)$.  

First, let $u$ be a principal map. By Lemmas \ref{units as morphisms} and \ref{lemma: factorization of units}, there exist unique $\alpha: G \rightarrow \text{GL}_1$ and $\beta: X \rightarrow \text{GL}_1$ such that $u_A(g, x) = \alpha_A(g) \beta_A(x)$ for all $(g, x)\in G(A) \times X(A)$.  Then, we have
\begin{equation}\label{eq: cancel}
\alpha_A(gh) \beta_A(x) = u_A(gh, x) = u_A(g, hx) u_A(h, x) = \alpha_A(g) \beta_A(hx) \alpha_A(h) \beta_A(x). 
\end{equation}
Since $\beta_A(x) \in \text{GL}_1(A)$ is invertible, we cancel it from both sides of \eqref{eq: cancel} to obtain 
\begin{equation}\label{eq: group}
\alpha_A(gh) = \alpha_A(g) \alpha_A(h) \beta_A(hx).
\end{equation}
Since the image of $\alpha_A$ consists of units, this equation gives
\begin{equation}\label{eq: fraction}
\beta_A(hx) = \frac{\alpha_A(gh)}{\alpha_A(g)\alpha_A(h)}. 
\end{equation}
By taking $h = 1$ in the above, we obtain $\beta_A(x) = \alpha_A(1)^{-1}$, so $\beta_A$ is constant.  By Lemma~\ref{lemma: constant lemma}, $\beta_A(x) = 1$ for all $x$. Then, from \eqref{eq: group}, we have
\[
\alpha_A(gh) = \alpha_A(g) \alpha_A(h) \beta = \alpha_A(g) \alpha_A(h),
\]
showing that $\alpha$ is a homomorphism.  

Thus we have constructed a map from principal maps $u$ to the set of homomorphisms $\alpha: G \rightarrow GL_1$.  Moreover the equation $u_A(g, x) = \alpha_A(g) \beta_A(x) = \alpha_A(g)$ implies $u$ is determined by $\alpha$, i.e. the map is injective.  

For surjectivity, fix a homomorphism $\alpha: G \rightarrow GL_1$ and define $u$ by $u_A(g, x) = \alpha_A(g)$, i.e. $u$ is the composition $\alpha \circ \pi: G \times X \rightarrow G \rightarrow GL_1$.  Then, we have
\[
u_A(gh, x) = \alpha_A(gh) = \alpha_A(g) \alpha_A(h) = u_A(g, hx) u_A(h, x),
\]
which establishes surjectivity.  

Because the correspondence we are considering is given by $u_A(g, x) = \alpha_A(g)$, we show it is a homomorphism as follows.  Let $\alpha, \alpha'$ correspond to $u, u'$.  Then, we have
\[
(uu')_A(g, x) = u_A(g, x)u'_A(g, x) = \alpha_A(g)\alpha'_A(g),
\]
and hence $uu'$ corresponds to $\alpha\alpha'$.

Finally, by Proposition \ref{Equivariant structures on O_X}, the equivariant line bundle is characterized by the property that $G$ acts on $1_x$ via $g1_x = u_A(g, x)1_{gx} = \alpha(g) 1_{gx}$.
\end{proof}

Note that two different equivariant structures on the line bundle $\mathcal{O}_X$ may be isomorphic under an automorphism of the line bundle.  Our next goal is to determine which equivariant structures are isomorphic.

A first step towards this goal is the following easy lemma.

\begin{lem}
Let $K$ be either a semiring or a monoid. Let $X$ be a scheme over $K$ and $G$ be a group scheme over $K$ acting on $X$.  Fix a morphism $f: X \rightarrow GL_1$.  Let $u$ be the ratio of the two morphisms $G\times X \rightarrow GL_1$ obtained by composing $f$ with either the action or the projection $G\times X \rightarrow X$.  In other words, the functor of points of $u$ is given by
\[
u_A(g, x) = f_A(gx) / f_A(x).
\]
Then $u$ is principal, i.e., $u$ satisfies $u_A(gh, x) = u_A(h, x) u_A(g, hx)$.
\end{lem}

In what follows, let $K$ be either a semiring or a monoid, $X$ be a scheme over $K$ and $G$ be a group scheme over $K$ acting on $X$ unless otherwise stated. 

\begin{lem}\label{lemma: isomorphic equivariant structures on O_X}
Let $L$ and $L'$ be two equivariant line bundles whose underlying line bundle is $\mathcal{O}_X$ and let $u, u': G\times X \to \text{GL}_1$ be the associated principal maps from the construction of Lemma \ref{Equivariant structures on O_X}.  Then $L \simeq L'$ as equivariant line bundles if and only if there is some $f: X \rightarrow \text{GL}_1$ such that 
\begin{equation}\label{equation for trivial equivariant line bundle isomorphism}u_A(g, x) f_A(gx) = u'_A(g, x) f_A(x)\end{equation}
for all $K$-algebra $A$ and all $(g, x)\in G(A) \times X(A)$.
\end{lem}
\begin{proof}
Let $v$ be a basis element of $L$ and $w$ be a basis element of $L'$ (for concreteness, we take both to be $1\in \Gamma(X, \mathcal{O}_X^\times)$, but giving them individual names makes the argument more clear).
Given an isomorphism $\phi: L \rightarrow L'$, we obtain a basis element $\phi(v)$ of $L'$.  Such a basis element must have the form $\phi(v) = \alpha w$ for some $\alpha \in \Gamma(X, \mathcal{O}_X^\times)$.  Let $f: X\rightarrow GL_1$ correspond to $\alpha$ under Lemma \ref{units as morphisms}.  For $x\in X(A)$ the naturality part of Lemma \ref{units as morphisms} tells us that $f_A(x) = f \circ x$ corresponds to $x^*(\alpha) \in GL_1(A)$. So, we identify $x^*(\alpha)$ with $f_A(x)$.

$G(A)$ acts on the fiber $L'_x$ via $g(w_x) = u'_A(g, x) w_{gx}$. So, we have
\[ 
g(\phi(v_x)) = g(f_A(x) w_x) = f_A(x) g(w_x) = f_A(x) u'_A(g, x) w_{gx}.
\]
On the other hand, since $\phi$ is an isomorphism of equivariant bundles,
\[ g(\phi(v_x)) = \phi(g v_x) = \phi(u_A(g, x) v_{gx}) = u_A(g, x) f_A(gx) w_{gx} \]
Combining the two equations yields equation \eqref{equation for trivial equivariant line bundle isomorphism} as desired.

Conversely, suppose we are given $f: X \rightarrow GL_1$ satisfying \eqref{equation for trivial equivariant line bundle isomorphism}.  $f$ corresponds to a unit $f \in \Gamma(X, \mathcal{O}_X^\times)$, and we obtain a basis vector $\eta \in \Gamma(X, L')$ by multiplying $w$ with $f$.  For each $K$-algebra $A$ and each $x\in X(A)$, define $\phi_x: L_x \rightarrow L'_x$ via
\[
\phi_x(a v_x) = a \eta_x = a f_A(x) w_x
\]
for all $a\in A$.  Combining the maps on different fibers gives rise to maps
\[
\phi_A: L(A) \rightarrow L'(A)
\]
compatible with the projections of each bundle onto $X(A)$.  It remains to show $\phi$ is an isomorphism of equivariant line bundles.  We have seen compatibility with projections, and the fact that $\phi_x$ is a linear isomorphism for each point $x$ is clear.

Next we check that $\phi$ is equivariant.  By the same calculations as the other direction of this proof, $g(\phi_A(v_x)) = f_A(x) u'_A(g, x) w_{gx}$ and $\phi_A(g v_x) = u_A(g, x) f_A(gx) w_{gx}$.  Combining with \eqref{equation for trivial equivariant line bundle isomorphism}, we see that $g \circ \phi_x$ and $\phi_{gx} \circ g$ agree on a basis vector, and hence are equal, which implies $g \circ \phi_A = \phi_A \circ g$.

Finally, we show that $\phi$ is natural in $A$. To show that the action is natural in $A$, let $h:\Spec B \to \Spec A$, and $h^*$ be the induced maps. We need to show that the following diagram commutes:
\[
\begin{tikzcd}
L(A)\arrow[r,"\phi_A"] \arrow[d," h^*",swap] & L'(A) \arrow[d,"h^*"] \\
L(B) \arrow[r,"\phi_B"] & L'(B)
\end{tikzcd}
\]
Let $av_x \in L_x \subseteq L(A)$. Then we have
\[
h^*\phi_A(av_x)=h^*(af_A(x)w_x)=h^*(af_A(x))h^*(w_x)=h^*(a)h^*(f_A(x))w_{h^*x}
\]
and 
\[
\phi_Bh^*(av_x)=\phi_B(h^*(a)h^*(v_x))=\phi_B(h^*(a)(v_{h^*x}))=h^*(a)f_B(h^*x)w_{h^*x}
\]
From the naturality of $f$, we have $h^*f_A = f_Bh^*$, it follows that we have
\[
h^*\phi_A(av_x)=\phi_Bh^*(av_x),
\]
showing that $\phi$ is natural in $A$. 
\end{proof}

\begin{pro}
Let $X$ be a toric $\mathbb{F}_1$-scheme and $G$ be a group scheme over $\mathbb{F}_1$. Let $K$ be an idempotent semifield. Suppose $G \times X$ is connected.  Then there is a one-to-one correspondence between isomorphism classes of $G_K$-equivariant line bundles on $X_K$ which are trivial as line bundles and isomorphism classes of $G$-equivariant line bundles on $X$ which are trivial as line bundles.
\end{pro}
\begin{proof}
From Proposition \ref{Correspondence for equivariant structures on O_X}, we have a one-to-one correspondence between $G$-equivariant line bundle structures on $\mathcal{O}_X$ and $G_K$-equivariant line bundle structures on $\mathcal{O}_{X_K}$.
Now, with the same notation as in Lemma \ref{lemma: isomorphic equivariant structures on O_X}, two equivariant structures for the trivial line bundle on $X$ are isomorphic over $X_\mathbb{T}$ if and only if the ratio of the corresponding units has the form $f(gx) / f(x)$.  But a constant factor in front of $f$ does not affect this expression, so we may assume that $f$ has no component in $K^\times$, so it also yields an isomorphism of equivariant line bundles over $X$ from Proposition \ref{Correspondence for equivariant structures on O_X}.
\end{proof}

\section{Gluing equivariant vector bundles}\label{section: Gluing equivariant vector bundles}

It would be useful to be able to restrict equivariant vector bundles on $X$ to open sets $U\subseteq X$.  This is not always possible: if $U$ is not closed under the $G$-action, it should not inherit a $G$-action, so the notion of an equivariant vector bundle on $U$ does not even make sense.  The next lemma shows that this is the only obstruction. 

Let $R$ be a semiring or monoid. Let $G$ be a group scheme over $R$.  Let $X$ be a scheme over $R$ equipped with a $G$-action.  Let $E$ be a $G$-equivariant vector bundle on $X$.  Let $U\subseteq X$ be an open set with the property that the action $G \times_R X \rightarrow X$ restricts to an action $G \times_R U \rightarrow U$ and let $i: U\rightarrow X$ be the inclusion.  Let $E_U$ be the vector bundle obtained by restricting $E$ to $U$. As in the previous section, for schemes $X$ and $Y$ over $R$, we will simply write $X\times Y$ instead of $X\times_RY$ when there is no possible confusion. 

\begin{lem}\label{lemma: restriction of equivariant vector bundle}
With the same notation as above, the following hold.
\begin{enumerate}
    \item 
For any $R$-algebra $A$ and any $u \in U(A)$, the fiber of $E_U$ over $u$ is equal to the fiber of $E$ over $i(u)$. 
\item 
There is a unique equivariant vector bundle structure on $E_U$ such that for each $(g, u) \in G(A) \times U(A)$, the induced map on fibers $g: (E_U)_u \rightarrow (E_U)_{gu}$ is the same as the map $g: E_{i(u)} \rightarrow E_{g i(u)}$ induced by the equivariant structure on $E$.
\end{enumerate}
\end{lem}
\begin{proof}
$(1)$ $E_U$ is the pullback of $E$ along $i$.  The fiber $(E_U)_u$ is the module of global sections of the pullback $u^* E_U = u^* i^* E$.  Since $i(u) = i \circ u$, $(E_U)_u$ equals the module of global sections of $i(u)^* E$, and this module is $E_{i(u)}$.

$(2)$ The action map $\rho_A^U:G(A) \times E_U(A) \rightarrow E_U(A)$ is determined by how $G(A)$ acts on each fiber, and this is specified in the statement of the lemma.  To show this is an action, we must show that for any $v \in (E_U)_u$ and $g, h \in G(A)$, $(gh)v = g(hv)$. But this is true because by construction $G$ acts on the fiber in the same way if we view it as $E_{i(u)}$, and the map $\rho_A:G(A) \times E(A) \rightarrow E(A)$ really is a group action.  The action map induces linear isomorphisms on fibers, since fiberwise it looks the same as the action on $E$.   

It remains to check naturality. Let $f:\Spec B \to \Spec A$, and $f^*$ be the induced maps. We need to show that the following diagram commutes:
\[
\begin{tikzcd}
G(A) \times E_U(A) \arrow[r,"\rho_A^U"] \arrow[d," f^*",swap] & E_U(A) \arrow[d,"f^*"] \\
G(B) \times E_U(B) \arrow[r,"\rho_B^U"] & E_U(B)
\end{tikzcd}
\]
But, under the inclusion $i$, we may view $E_U(A)$ as a subset of $E(A)$ and $\rho_A^U=\rho_A\mid_U$. In particular, the naturality directly follows from that of $\rho$. 
\end{proof}

The following proposition shows that under suitable conditions, we may obtain equivariant vector bundles from gluing equivariant vector bundles on an open cover.

Let $R$ be a semiring or monoid. Let $G$ be a group scheme over $R$.
Let $G$ be a group scheme over $R$, let $X$ be a scheme over $R$ equipped with a $G$-action, and let $E$ be a vector bundle over $X$.  Let $\{ U_i \}$ be an open cover of $X$, and suppose each open set $U_i$ is $G$-invariant in the sense that the action on $X$ restricts to maps $G \times U_i \rightarrow U_i$. 

\begin{pro}\label{proposition: gluing equivariant vector bundles} 
With the same notation as above, there is a one-to-one correspondence between equivariant structures on $E$ and collections of equivariant structures on each $E_{U_i}$ with the property that restricting the action from $E_{U_i}$ or $E_{U_j}$ to $E_{U_{ij}}$ yield the same morphism $G \times E_{U_{ij}} \rightarrow E_{U_{ij}}$, where $E_{U_{ij}}:=E_{U_i} \cap E_{U_j}$.
\end{pro}
\begin{proof}
We first prove the semiring case.
We will use the geometric vector bundle perspective to view $E$ as a scheme equipped with a map $\pi: E \rightarrow X$. From the definition, one can easily see that $E_U$ is an open subscheme of $E$ for a $G$-invariant open subset $U$.

If $E$ is an equivariant vector bundle, we can give a geometric description of the action on the bundle $E_U$ from Lemma \ref{lemma: restriction of equivariant vector bundle}.  It is clear that the functor of points for the action on $E_{U}$ is exactly the same as the functor of points from the map $G \times E_{U} \rightarrow E_{U}$ whose composition with the open immersion $E_{U}\rightarrow E$ is the restriction $G \times E_{U} \rightarrow E$ of the action because both are given by $g, v \mapsto \beta(g, v)$ where $\beta$ is the action on $E$, and where we identify $U(A)$ with a subset of $X(A)$.  In other words, the equivariant vector bundle structure on $E_U$ is the same one obtained by restriction of the action viewed as a morphism of schemes.

Suppose we are given actions $\beta_i: G \times E_{U_i} \rightarrow E_{U_i}$ which agree on overlaps and which make each $E_{U_i}$ into an equivariant vector bundle.  Then we may glue them to obtain a map $\beta: G \times E \rightarrow E$.  Since each $\beta_i$ is compatible with the action $\alpha_{U_i}: G \times U_i \rightarrow U_i$, $\beta$ must be compatible with $\alpha: G\times X \rightarrow X$.  Consider the maps
\[
\eta, \theta: G \times G \times E \rightarrow E
\]
given on functors of points by
\[
\eta(g, h, v) = \beta(gh, v) \textrm{ and } \theta(g, h, v) = \beta(g, \beta(h, v)).
\]
Consider maps $\eta_i, \theta_i$ defined similarly using the actions on $E_{U_i}$.  Since $\beta_i$ is an action, we obtain $\eta_i = \theta_i$.  Clearly if $v \in E_{U_i}(A)$ and $g, h \in G(A)$ then $\eta(g, h, v) = \eta_i(g, h, v)$ and similarly for $\theta$,  and hence $\eta = \theta$, showing that $\beta$ is a group action.

Consider the map $+: E \times E \rightarrow E$ which is simply vector addition on the level of functors of points. From Proposition \ref{proposition: addition morphism}, $+$ is a morphism of schemes when we view $E$ as a geometric vector bundle. To show that $\beta$ is additive on fibers we need to show for $g \in G(A)$, $x \in X(A)$ and $v, w \in E_x$ that $\beta(g, v + w) = \beta(g, v) + \beta(g, w)$.  We may rephrase this as
\[
\beta \circ (\mathrm{id}_G \times +) = + \circ (\beta \times \beta) \circ \varphi\circ (\Delta \times \mathrm{id}_{E\times E}),
\]
where $\Delta: G \to G \times G$ is the diagonal and $\varphi:G\times G \times E \times E \to G\times E \times G \times E$ switches the second and the third coordinates. In view of the assumptions on the open cover, it suffices to check that the additivity holds on each $E_{U_i}$ and by retracing our steps, this is equivalent to $\beta_i$ being additive on fibers.  Since $E_{U_i}$ is an equivariant vector bundle, $\beta$ is additive on fibers. We may now show linearity on fibers by using the scalar multiplication map $\cdot: \mathbb{A}^1_X \times E \rightarrow E$, which is also a morphism of schemes when we view $E$ as a geometric vector bundle from Proposition \ref{proposition: addition morphism}.

Finally, $\beta_A: G(A) \times X(A) \rightarrow X(A)$ is natural because it is a morphism of schemes.  Thus $E$ is an equivariant vector bundle.

Conversely, suppose $E$ is an equivariant vector bundle.  Then we have seen that restricting the action $G \times E \rightarrow E$ to $E_{U_i}$ gives an equivariant vector bundle structure on $E_{U_i}$.  Because these maps $\beta_i: G \times E_{U_i} \rightarrow E_{U_i}$ are all given by restricting a morphism of schemes, they must agree on overlaps.

Now, one can easily see that this gives one-to-one correspondence. 

For the monoid case, one may modify Proposition \ref{proposition: addition morphism} just for the scalar multiplication map since in this case one does not have an additive map. With this modification, the same proof goes through in this case.
\end{proof}

We now show that in the toric case, all equivariant vector bundles are the sum of equivariant line bundles.  This reduces the problem of classifying toric vector bundles to the line bundle case.

\begin{mydef}\label{definition: direct sum of equivarant bundles}
Let $X$ be a scheme over a semiring $R$ and $G$ be a group scheme over $R$ acting on $X$. Let $E_1$ and $E_2$ be $G$-equivariant line bundles. The direct sum $E=E_1\oplus E_2$ is a $G$-equivariant vector bundle whose underlying vector bundle is $E_1\oplus E_2$ (viewed $E_1$ and $E_2$ as vector bundles) and the action is given as follows:
\[
\beta:G \times E \to E, \quad (g,v_1+v_2)=\beta_1(g,v_1)+\beta_2(g,v_2),
\]
where $\beta_i:G\times E_i \to E_i$ is the action of $E_i$ for $i=1,2$. In other words, the action of $E$ is componentwise. 
\end{mydef}

\begin{mythm}\label{theorem: equivariantly split theorem}
Let $X$ be a toric scheme over an idempotent semifield $R$, and let $G$ be the corresponding torus.   Let $E$ be a $G$-equivariant vector bundle on $X$.  Then there are unique (up to permutation) equivariant line bundles $L_1, \ldots, L_n$ such that $E = L_1 \oplus \ldots \oplus L_n$ (as $G$-equivariant vector bundles).
\end{mythm}
\begin{proof}
First we prove uniqueness; the toric assumption is not needed here.  By Proposition~\ref{theorem: bundles split}, any direct sum decomposition of a vector bundle into line bundles is unique.  So, we just need to show the action $\beta_i$ on each $L_i$ is determined by the action $\beta$ on $E$.  But by definition of a direct sum of equivariant line bundles, the action on $E = L_1 \oplus \ldots \oplus L_n$ is given by
\[
\beta(g, v_1 + \ldots + v_n) = \beta_1(g, v_1) + \ldots +\beta_n(g, v_n), \quad v_i \in L_i~\forall i=1,\dots,n.
\]
In particular, for $v_i \in L_i$, $\beta_i$ is given by $\beta_i(g, v_i) = \beta(g, v)$.

As a vector bundle, again by Proposition \ref{theorem: bundles split}, $E$ is a sum of line bundles. Call these line bundles $L_1, \ldots, L_n$.  We need to specify the group action on the $L_i$.

For any cone $\sigma$, we let $E_\sigma$ denote the restriction to $U_\sigma$ and similarly for $(L_i)_\sigma$.  Observe that
\[
E_\sigma = (L_1)_\sigma \oplus\ldots\oplus (L_n)_\sigma
\]
is the unique decomposition into line bundles (by Proposition \ref{theorem: bundles split}). Note $E_\sigma$ is trivial by Proposition \ref{corollary: vector bundle on affine is trivial}.  Furthermore, Corollary \ref{corollary: torus-equivariant trivial vector bundles are sums of line bundles} implies that $E_\sigma$ is a sum of equivariant line bundles.  By uniqueness of the decomposition into line bundles, these equivariant line bundles must equal $(L_1)_\sigma, \ldots, (L_n)_\sigma$.  In other words, we have obtained equivariant line bundle structures on each $(L_i)_\sigma$ such that $E_\sigma = (L_1)_\sigma \oplus\ldots\oplus (L_n)_\sigma$ as equivariant vector bundles.

Given two cones $\sigma_1, \sigma_2$, we let $E_{\sigma_1, \sigma_2} = E_{U_{\sigma_1} \cap U_{\sigma_2}}$ and similarly for the $L_i$.  We obtain two equivariant vector bundle structures on each $(L_i)_{\sigma_1, \sigma_2}$ such that $E_{\sigma_1, \sigma_2} = (L_1)_{\sigma_1, \sigma_2} \oplus\ldots\oplus (L_n)_{\sigma_1, \sigma_2}$: we can either restrict the action on $(L_i)_{\sigma_1}$ or restrict the action on $(L_i)_{\sigma_2}$.  By the uniqueness part of the theorem (which does not really require $U_{\sigma_1} \cap U_{\sigma_2}$ to be toric), these must give the same action on $(L_i)_{\sigma_1, \sigma_2}$.  Thus the actions on the $(L_i)_\sigma$ for different values of $\sigma$ agree on overlaps.  By Proposition \ref{proposition: gluing equivariant vector bundles}, these glue together to give an action on $L_i$.

There are now two actions on $E$ which make $E$ into an equivariant vector bundle. The first is the action that comes from the assumption that $E$ is a $G$-equivariant vector bundle. The other, we can obtain from the actions on the $L_i$, which allow us to put an equivariant vector bundle structure on $E = L_1 \oplus \ldots \oplus L_n$.  We had $E_\sigma = (L_1)_\sigma \oplus\ldots\oplus (L_n)_\sigma$ as equivariant vector bundles.  Hence the two actions on $E$ agree when restricted to any $U_\sigma$.  By Proposition \ref{proposition: gluing equivariant vector bundles}, this means the two actions are equal.  The result now follows.
\end{proof}

Next we relate equivariant vector bundles on toric varieties to equivariant vector bundles on the torus.  Conceptually this should use the density of the torus.  But we need something stronger, because it is not clear that regular functions that agree on a dense set should agree everywhere.  In the classical case $\{ x: f(x) = g(x) \}$ is closed as it is the set of prime ideals containing $\langle f - g \rangle$, but in the tropical case we cannot use subtraction.

Let $X$ be a toric scheme over a semiring $R$. As in the case over $\mathbb{Z}$, one has $X = \bigcup U_{\sigma_i}$, with each cone $\sigma_i$ in a fan, and $U_{\{0\}}\subseteq X$ is the torus.  

\begin{lem}\label{lemma: regular function is determined on torus}
Let $X$ be a toric scheme over a semiring $R$.  Let $f, g \in \Gamma(X, \mathcal{O}_X)$. Suppose $f = g$ on $U_{\{0\}}$.  Then $f = g$.
\end{lem}
\begin{proof}
As a preliminary, observe that an injection of sets $S \rightarrow T$ induces an injection $R^S \rightarrow R^T$ of free modules.  In particular, an injection of monoids $M_1 \rightarrow M_2$ induces a monomorphism $R[M_1] \rightarrow R[M_2]$.

First consider the case where $X$ is an affine scheme, corresponding to a cone $\sigma$, i.e. $X = \Spec R[\Lambda \cap \sigma^\vee]$ for some lattice $\Lambda$ and some cone $\sigma$. 
The torus is $\Spec R[\Lambda]$.  Because the inclusion $\Lambda \cap \sigma^\vee$ into $\Lambda$ is injective, the map $\phi: R[\Lambda \cap \sigma^\vee] \rightarrow R[\Lambda]$ is a monomorphism.  Consider two regular functions $f, g$ on $X$ which agree on the torus.  Algebraically, this means $\phi(f) = \phi(g)$, and injectivity gives $f = g$ as desired.

Now we consider the general case.  Assume $f = g$ on the torus.  For each cone $\sigma$ in the fan corresponding to $X$, we may apply this result to $U_\sigma$ to obtain $f \mid_{U_\sigma} = g \mid_{U_\sigma}$.  Since the $U_\sigma$ form an open cover, the result follows.
\end{proof}

We may view regular functions as sections of the trivial line bundle.  In fact, the above result can be generalized to any vector bundle.

\begin{lem}\label{lemma: section is determined by values on torus}
Let $X$ be a toric scheme over an idempotent semifield $R$. 
Let $E$ be a vector bundle on $X$.  Suppose $s, t \in \Gamma(X, E)$ are such that $s \mid_{U_{\{0\}}} = t \mid_{U_{\{0\}}}$.  Then $s = t$.
\end{lem}
\begin{proof}
We have already proven in Lemma~\ref{lemma: regular function is determined on torus} the statement when $E = \mathcal{O}_X$.

Next we consider the case where $E$ is a line bundle.  For any cone $\sigma$, $E_{U_\sigma}$ is trivial by Proposition \ref{corollary: vector bundle on affine is trivial}.  Since $s \mid_{U_\sigma}$ and $t \mid_{U_\sigma}$ agree on the torus, they are equal by the case of the trivial line bundle.  Since this holds for all $\sigma$ and since $E$ is a sheaf, we have $s = t$.

Finally, we use Proposition \ref{theorem: bundles split} to write $E = L_1 \oplus \ldots \oplus L_n$.  We may decompose $s, t$ into sections $s_i, t_i \in \Gamma(X, L_i)$ for each $i$.  Since $s$ and $t$ agree on the torus, the same is true for $s_i$ and $t_i$.  By the line bundle case $s_i = t_i$ for all $i$, and hence $s = t$.
\end{proof}

\begin{lem}\label{lemma: functor of points of a section}
Let $X$ be a scheme over a semiring $R$ and let $E$ be a vector bundle. Let $s_A: X(A) \rightarrow E(A)$ be a natural transformation such that $s_A(x) \in E_x$ for all $x$, where $A$ are $R$-algebras. Then there exists a unique element $\check{s} \in \Gamma(X, E)$ which induces $s_A$.
\end{lem}
\begin{proof}
Let $\pi:E \to X$ be a projection map, where we view $E$ as a geometric vector bundle. Let $s:X \to E$ be a morphism of schemes corresponding to the given natural transformation. Since $s_A(x) \in E_x$, we have that $\pi\circ s = \text{id}_X$, and hence $s$ is a section of $\pi$. The claim follows from the fact that the sections of $\pi$ bijectively corresponds to the elements of $\Gamma(X,E)$. 
\end{proof}

We may use the above results to show that the $G$-action on an equivariant vector bundle is determined by the action on the part of the vector bundle lying over the torus.

\begin{pro}\label{proposition: action determined by generic point}
Let $X$ be a toric scheme over an idempotent semifield $R$.  Let $G = U_{\{0\}} \subseteq X$ be the torus.  Let $E$ be a vector bundle on $X$.  Let $\beta_1, \beta_2: G \times E \rightarrow E$ be actions which make $E$ into an equivariant vector bundle.  If $\beta_1$ and $\beta_2$ induce the same action on $E_{U_{\{0\}}}$, then $\beta_1 = \beta_2$.
\end{pro}
\begin{proof}
Suppose first that $E$ is a trivial vector bundle. Fix a global section $v$ of $E$.  We consider the pullback $\alpha^* E$ along the action $\alpha:G \times X \rightarrow X$. 

Let $A$ be an $R$-algebra. Because fibers are constructed via pullback along a point $(g, x) \in G(A) \times X(A)$, the fibers of $\alpha^* E$ are given by 
\begin{equation}\label{eq: identification}
 (\alpha^* E)_{g, x} = E_{gx}.   
\end{equation}
Note that $(\beta_1)_A(g, v_x), (\beta_2)_A(g, v_x) \in E_{gx}$.  Define
\begin{equation}
(s_i)_A(g, x) := (\beta_i)_A(g, v_x),
\end{equation}
where we use the identification \eqref{eq: identification}. Hence, we have $(s_i)_A:G(A) \times X(A) \to \alpha^*E(A)$. Apply Lemma \ref{lemma: functor of points of a section} to obtain global sections $\check{s_1}, \check{s_2} \in \Gamma(G\times X, \alpha^*E)$.  By assumption, when $x$ is a point of the torus, we have
\[
(\beta_1)_A(g, v_x) = (\beta_2)_A(g, v_x).
\]
Hence the functors of points of $\check{s_1}, \check{s_2} $ agree on the torus $G\times G$. By Lemma \ref{lemma: section is determined by values on torus}, $\check{s_1} = \check{s_2}$, and hence they induce the same element of each fiber, i.e., $(\beta_1)_A(g, v_x) = (\beta_2)_A(g, v_x)$ for all $x$.  In particular, letting $v$ range over a basis $v_1, \ldots, v_n$ of global sections of the trivial bundle $E$, for any $x \in X(A)$ there is a basis $(v_1)_x, \ldots, (v_n)_x\in E_x$ such that $(\beta_1)_A(g, (v_i)_x) = (\beta_2)_A(g, (v_i)_x)$.  By linearity, this implies $(\beta_1)_A = (\beta_2)_A$, and hence $\beta_1=\beta_2$.

Now consider the general case.   For any cone $\sigma$, $E_{U_\sigma}$ is trivial by Proposition \ref{corollary: vector bundle on affine is trivial}.  The actions on $E_{U_\sigma}$ induced by $\beta_1$ and $\beta_2$ agree on the torus, so are equal.  Since this is true for all $\sigma$, the result follows from Proposition \ref{proposition: gluing equivariant vector bundles}.
\end{proof}

Another technique for relating equivariant line bundles over toric schemes to line bundles over the torus is the following result.
\begin{lem}\label{lemma: extension of equivariant line bundle structure in the affine case}
Let $X_R$ be an affine toric scheme over an idempotent semifield $R$ and let $G_R = U_{\{0\}}$ be the torus.  Let $L$ be a line bundle on $X_R$.  Then restriction of the action yields a one-to-one correspondence between $G$-actions on $L$ which make it an equivariant line bundle and $G$-actions on $L_{U_{\{0\}}}$ which make it an equivariant line bundle.
\end{lem}
\begin{proof}
By Proposition \ref{corollary: vector bundle on affine is trivial}, $L$ is trivial.  By Proposition \ref{Correspondence for equivariant structures on O_X}, it suffices to prove this in the $\mathbb{F}_1$-case (i.e., monoid schemes). Let $G$ be the torus over $\mathbb{F}_1$ and $X$ be the affine $\mathbb{F}_1$-scheme associated to $X_R$, i.e., $X_R = X\otimes_{\mathbb{F}_1}R$.  Since $G$ and $X$ are affine, we may apply Proposition \ref{proposition: equivariant line bundles and characters} to see that $G$-actions which make $L$ into an equivariant line bundle are in bijective correspondence with homomorphisms $\chi: G \rightarrow GL_1$ and this correspondence is such that
\[
g1_x = \chi_A(g)1_{gx} \quad \textrm{for } (g, x)\in G(A) \times X(A).
\]
Applying these propositions again, such homomorphisms are in bijective correspondence with $G$-actions on $L_{U_{\{0\}}}$ which make it into a line bundle, and this correspondence is such that 
\[
g1_x = \chi_A(g)1_{gx}, \quad \textrm{for } (g, x)\in G(A) \times U_{\{0\}}(A).
\]
Composing the above bijections gives the desired bijective correspondence.  It remains to show that the bijection we obtained is given by restricting the equivariant line bundle structure on $L$ to one on $L_{U_{\{0\}}}$.  Let $(g, x)\in G(A) \times U_{\{0\}}(A) \subseteq G(A) \times X(A)$.  Then in $L$, $g$ acts on $1_x$ via $g1_x = \chi(g)1_{gx}$, where $\chi: G\rightarrow GL_1$ is the homomorphism corresponding to the action on $L$.  Thus the same is true for the restriction of this action to $L_{U_{\{0\}}}$.  But the bijection above is given by the property that the action on $L_{U_{\{0\}}}$ satisfies $g1_x = \chi(g)1_{gx}$, so it agrees with the restriction map.
\end{proof}

\begin{pro}\label{proposition: construction of equivariant structure on line bundle}
Let $X_R$ be a toric scheme over an idempotent semifield $R$ and let $G_R$ be the torus.  Let $L$ be a line bundle over $X_R$.  Then there is some action which makes $L$ into an equivariant line bundle.
\end{pro}
\begin{proof}
This is true if $L$ is trivial by Example \ref{example: trivial eq} and in particular we may make $L_{U_{\{0\}}}$ into an equivariant line bundle.  For each cone $\sigma$, we obtain an equivariant vector bundle structure on $L_{U_\sigma}$ via Lemma \ref{lemma: extension of equivariant line bundle structure in the affine case}.  We must check agreement on overlaps.  By \cite[Section 1.4]{Ful93}, the intersection of two such open sets is given by $U_\sigma \cap U_\tau = U_{\sigma \cap \tau}$.    

We obtain two equivariant line bundle structures on $L_{U_{\sigma \cap \tau}}$ by restricting the actions on $L_{U_\sigma}$ and $L_{U_\tau}$ respectively.  By construction, both agree when restricted further to the torus $U_{\{0\}}$.  Thus we may apply Lemma \ref{lemma: extension of equivariant line bundle structure in the affine case} to $U_{\sigma \cap \tau}$ to see that both actions on $L_{U_{\sigma \cap \tau}}$ are equal.  This implies agreement on overlaps, so the result follows from Proposition \ref{proposition: gluing equivariant vector bundles}
\end{proof}

We next turn to the problem of classifying equivariant line bundles which are trivial as line bundles.  By Proposition \ref{proposition: equivariant line bundles and characters}, all such equivariant line bundles may be constructed from homomorphisms $G \rightarrow GL_1$ where $G$ is the torus over $\mathbb{F}_1$.  The following lemma classifies such homomorphisms.

\begin{lem}\label{lemma: characters as elements of the dual lattice}
Let $\Lambda \subseteq \mathbb{R}^n$ be a lattice (viewed as a group and written multiplicatively).  Define the $\mathbb{F}_1$-scheme $G$ as $G = \Spec (\Lambda \cup \{ 0 \})$.  The functor of points of any morphism $\chi: G \rightarrow GL_1$ satisfies $\chi_A(xy) = \chi_A(x) \chi_A(y)$.  Moreover, there is an isomorphism between the group of such morphisms and $\Lambda$.
\end{lem}
\begin{proof}The fact there is an isomorphism between $\Lambda$ and the group of morphisms $G\rightarrow GL_1$ follows from Lemma \ref{units as morphisms}. 

Let $\chi:G \rightarrow GL_1$ be a morphism, and let $u \in \Lambda$ be the corresponding unit.  On the level of $\mathbb{F}_1$-algebras $\chi$ corresponds to the map $\mathbb{Z} \cup \{0\} \rightarrow \Lambda \cup \{ 0 \}$ sending the generator of $\mathbb{Z}$ to $u$.  Let $\mu, \pi_1, \pi_2: G \times G \rightarrow G$ be the multiplication and the two projections respectively.  The corresponding maps of $\mathbb{F}_1$-algebras $\Lambda \cup \{0\} \rightarrow (\Lambda \cup \{0\}) \otimes_{\mathbb{F}_1} (\Lambda \cup \{0\})$ are given respectively by $\mu^*(x) = x \otimes x$, $\pi_1^*(x) = x \otimes 1$, and $\pi_2^*(x) = 1 \otimes x$.  Hence the same formulas are true for the induced maps on unit groups of $\mathbb{F}_1$-algebras.  Thus we have
\[ \mu^*(u) = u \otimes u = (u \otimes 1)(1 \otimes u) = \pi_1^*(u) \pi_2^*(u) \]
Reinterpreting this in terms of maps $G \times G \rightarrow GL_1$ via Lemma \ref{units as morphisms}, we see that $\chi \circ \mu$ is given as the product of $\chi \circ \pi_1$ and $\chi \circ \pi_2$.  Rewriting this in terms of functors of points yields $\chi_A(xy) = \chi_A(x) \chi_A(y)$.
\end{proof}

While each element of $\Lambda$ yields an action on the trivial line bundle on a toric scheme over $\mathbb{F}_1$, two such actions might result in isomorphic equivariant line bundles.

\begin{lem}\label{lemma: dual lattice points corresponding to the trivial equivariant line bundle in the affine case}
Let $X = U_\sigma$ be a toric $\mathbb{F}_1$-scheme corresponding to a cone $\sigma$.  Let $\Lambda$ be the dual lattice and $G$ be the corresponding torus.  Let $x \in \Lambda$ and let $\chi: G \rightarrow GL_1$ be the corresponding map under Lemma \ref{lemma: characters as elements of the dual lattice}.  Then there is some morphism $f: X \rightarrow GL_1$ whose functor of points satisfies
\begin{equation}\label{eq: character of trivial equivariant line bundle}
\chi_A(g) = f_A(gx) / f_A(x)
\end{equation}
if and only if $x \in (\Lambda \cap \sigma^\vee)^\times$.
\end{lem}
\begin{proof}
Let $\alpha:G\times X \to X$ be the action and let $\pi_i: G \times X \rightarrow X$ be the projections. The condition
\[
f_A(x) \chi_A(g) = f_A(gx)
\]
can be rewritten as follows: 
\begin{equation}\label{eq: lemma}
(f \circ \pi_2) (\chi \circ \pi_1) = f \circ \alpha.
\end{equation}
Letting $x \in \Lambda$ and $u \in \Gamma(U_\sigma, \mathcal{O}_{U_\sigma}^\times) = (\Lambda \cap \sigma^\vee)^\times$ correspond to $\chi$ and $f$ respectively under Lemmas \ref{units as morphisms}, the condition \eqref{eq: lemma} may be rewritten as 
\begin{equation}
\pi_2^* u \pi_1^* x = \alpha^* u
\end{equation}
The map of $\mathbb{F}_1$-algebras
\[
(\Lambda \cup \{ 0 \}) \otimes_{\mathbb{F}_1} (\Lambda \cap \sigma^\vee \cup \{ 0 \}) \rightarrow (\Lambda \cap \sigma^\vee \cup \{ 0 \})
\]
corresponding to the action $G \times X \rightarrow X$ is given by $u \mapsto u \otimes u$, while the projections correspond to algebra homomorphisms given by tensoring with $1$.  Thus the same is true for the induced maps on unit groups (with tensor replaced by Cartesian product), so the above equation may be rewritten as 
\[ (1, u) (x, 1) = (u, u)\]
or simply $x = u$.
So the existence of $f: X \rightarrow GL_1$ satisfying \eqref{eq: character of trivial equivariant line bundle} is equivalent to the existence of $u \in (\Lambda \cap \sigma^\vee)^\times$ satisfying $x = u$.
\end{proof}

Let $X$ be a toric $\mathbb{F}_1$-scheme with fan $\Delta$.  Let $\Lambda$ be the dual lattice and $G$ be the corresponding torus.  Let $x \in \Lambda$ and let $\chi: G \rightarrow GL_1$ be the corresponding map under Lemma \ref{lemma: characters as elements of the dual lattice}. Let $\Pic_G(X)$ be the set of isomorphism classes of equivariant line bundles on $X$. From Proposition \ref{proposition: equivariant Picard group}, $\Pic_G(X)$ is a group. Moreover, Lemma \ref{lemma: characters as elements of the dual lattice} shows that there is a group homomorphism $\psi:\Lambda \to \Pic_G(X)$. The following proposition characterizes the kernel of $\psi$.

\begin{pro}\label{proposition: dual lattice points corresponding to the trivial equivariant line bundle}
With the same notation as above, the equivariant line bundle corresponding to $\chi$ under Proposition \ref{proposition: equivariant line bundles and characters} is isomorphic to the one corresponding to the trivial homomorphism if and only if $x \in \bigcap\limits_{\sigma \in \Delta} (\Lambda \cap \sigma^\vee)^\times$.
\end{pro}
\begin{proof}
By Lemma \ref{lemma: isomorphic equivariant structures on O_X}, the equivariant line bundle corresponding to $\chi$ is isomorphic to the one corresponding to the trivial character if and only if there is some $f: X \rightarrow GL_1$ whose functor of points satisfies \eqref{eq: character of trivial equivariant line bundle} for all $(g, x)\in  G(A) \times X(A)$.  By gluing morphisms of schemes, this is equivalent to the existence of a collection of morphisms $f_\sigma: U_\sigma \rightarrow GL_1$ which satisfy \eqref{eq: character of trivial equivariant line bundle} and agree on overlaps.  

First we claim that the condition of agreement on overlaps in the above correspondence is true.  In fact, since $U_\sigma \cap U_\tau = U_{\sigma \cap \tau}$ is itself toric, in order to check agreement on overlaps, it will suffice by Lemma \ref{lemma: regular function is determined on torus} and Lemma \ref{units as morphisms} to check that all the maps $f_\sigma$ have the same restriction to the torus.  Letting $1$ be the identity element of the torus, \eqref{eq: character of trivial equivariant line bundle} gives
\begin{equation}
\chi_A(g) = (f_\sigma)_A(g) / (f_\sigma)_A(1).
\end{equation}
The only natural transformation (which is an inclusion) $i:\Spec \mathbb{F}_1 \rightarrow GL_1$ is the constant map $1$ (this follows from Lemma \ref{units as morphisms}). For any $\mathbb{F}_1$-algebra $A$, $i_A$ is a map from the one point set $\Spec (\mathbb{F}_1)(A)=\{p\}$ to the group $\text{GL}_1(A)$. One can easily observe that the image of the only point $p$ is the identity element $1$ of $\text{GL}_1(A)$. We can write
\[
(f_\sigma)_A(1) = (f_\sigma \circ i)_A(p), 
\]
and $f_\sigma \circ i$ is a morphism from $\Spec \mathbb{F}_1$ to $\text{GL}_1$. Hence $(f_\sigma)_A(1) = 1$ and $\chi_A(g) = (f_\sigma)_A(g)$ for all $g \in G(A) \subseteq U_\sigma(A)$, which establishes the claim.

It remains to classify $\chi$ such that there exists a collection of morphisms $f_\sigma: U_\sigma \rightarrow GL_1$ satisfying \eqref{eq: character of trivial equivariant line bundle}.  Since we may construct each $f_\sigma$ independently, it follows from Lemma \ref{lemma: dual lattice points corresponding to the trivial equivariant line bundle in the affine case} that such $\chi$ correspond to $x \in \bigcap\limits_{\sigma \in \Delta} (\Lambda \cap \sigma^\vee)^\times$.
\end{proof}

\begin{mythm}\label{theorem: exact sequence}
Let $R$ be an idempotent semifield. Let $X$ be a toric scheme over $R$.  Let $G$ be the corresponding torus, $\Lambda$ be the dual lattice and $\Delta$ be the fan corresponding to $X$.  Then there is an exact sequence of abelian groups
\[ 
0 \rightarrow \bigcap_{\sigma \in \Delta} (\Lambda \cap \sigma^\vee)^\times \rightarrow \Lambda \rightarrow \mathrm{Pic}_G(X) \rightarrow \mathrm{Pic}(X) \rightarrow 0 .
\]
\end{mythm}
\begin{proof}
Proposition \ref{proposition: construction of equivariant structure on line bundle} tells us that the forgetful morphism $\phi: \mathrm{Pic}_G(X) \rightarrow \mathrm{Pic}(X)$ is surjective.  The combination of Proposition \ref{proposition: equivariant line bundles and characters} and Lemma \ref{lemma: characters as elements of the dual lattice} gives us a group homomorphism $\psi: \Lambda \rightarrow \mathrm{Pic}_G(X)$ whose image is $\mathrm{ker}(\phi)$.  The kernel of $\psi$ is characterized by Proposition \ref{proposition: dual lattice points corresponding to the trivial equivariant line bundle}.
\end{proof}

\section{Klyachko theorem for toric schemes over an idempotent semifield}\label{section: Tropical Klyachko theorem}

In this section, we prove a version of Klyachko theorem classifying toric vector bundles on a toric scheme over an idempotent semifield $K$. By a toric line bundle, we mean a torus-equivariant line bundle. 

We begin with a gluing construction, which constructs a toric line bundle by gluing together toric line bundles on open subsets corresponding to cones. 


Let $X$ be a toric scheme over $\mathbb{F}_1$ corresponding to a fan $\Delta$ with dual lattice $\Lambda$. Let $S$ be a set of families $u_\sigma \in \Lambda$ indexed by cones $\sigma\in \Delta$ such that for each such inclusion $\tau\subseteq\sigma$ of cones, the following holds
\[
u_\sigma^{-1} u_\tau\in (\Lambda \cap \tau^\vee)^\times.
\]

For cones $\tau\subseteq\sigma$, we will let $\phi_{\sigma\tau} = u_\sigma^{-1} u_\tau\in (\Lambda \cap \tau^\vee)^\times$.

Observe that $(\Lambda \cap \tau^\vee)^\times = \Gamma(U_{\tau}, \mathcal{O}_{U_\tau}^\times)$, so $\phi_{\sigma\tau}$ can be thought of as a morphism $\phi_{\sigma\tau}: U_\tau \rightarrow \text{GL}_1$ by Lemma \ref{units as morphisms}.  Similarly, we will identify $u_\sigma \in \Lambda = \Gamma(G, \mathcal{O}_G^\times)$ with a morphism $u_\sigma: G\rightarrow \text{GL}_1$, where $G = \Spec (\Lambda \cup \{0\})$ denotes the torus.  Lemma \ref{lemma: characters as elements of the dual lattice} ensures this morphism is a homomorphism. 

\begin{lem}\label{lem: multiplicative lemma}
With the same notation as above, for any $\mathbb{F}_1$-algebra $A$, $g \in G(A)$, $x \in U_\tau(A)$, one has the following
\begin{equation}\label{eq: multiplicative}
(\phi_{\sigma\tau})_A(gx)=(\phi_{\sigma\tau})_A(g)(\phi_{\sigma\tau})_A(x),
\end{equation}
where we view $\phi_{\sigma\tau}:U_\tau \to \text{GL}_1$.
\end{lem}
\begin{proof}
For the notational convenience, we let $\gamma:=\phi_{\sigma\tau}$. Let $T=\angles{t}$ so that $\text{GL}_1=\Spec T$, and $M=\Gamma(U_\tau,\mathcal{O}_{U_\tau}^\times)$.

In the affine case, Lemma \ref{units as morphisms} states that any $\gamma \in M^\times$ corresponds the map $\gamma:T \to M$ sending $t$ to $\gamma$. With $g:\Lambda \to A$ and $x:M \to A$, we have
\[
\gamma_A(g)=g\circ \gamma: T \to A, \quad \gamma_A(x)=x \circ \gamma: T \to A,
\]
where for $\gamma_A(g)$ we use the fact that $M=(\Lambda \cap \tau^\vee)^\times \subseteq \Lambda$. Since the group structure on $\text{GL}_1$ is given by $\Delta:T \to T \otimes T$, which sends $t$ to $t\otimes t$, we have
\[
(\gamma_A(g)\gamma_A(x))(t) = \mu\circ ((g\circ \gamma)\otimes (x\circ \gamma))\circ \Delta(t)=\mu\circ ((g\circ \gamma)(t)\otimes (x\circ \gamma)(t))
\]
\[
=g(\gamma)x(\gamma), 
\]
where $\mu:A \otimes A \to A$ is multiplication of $A$ and we identified $\gamma(t)$ with the element $\gamma \in M$. 

Next, let $\alpha:G \otimes U_\tau \to U_\tau$ be the action of torus. We let $\alpha^*:M \to \Lambda\otimes M$ be the corresponding map sending $m$ to $m\otimes m$. Then, we have
\[
(\gamma_A(gx))(t)=\mu\circ (g\otimes x)\circ \alpha^*\circ \gamma(t) = \mu\circ (g\otimes x)\circ \alpha^*\circ \gamma(t)
\]
\[
=\mu (g\otimes x)(\gamma \otimes \gamma) = g(\gamma)x(\gamma).
\]
This proves the desired equality \eqref{eq: multiplicative}.
\end{proof}

\begin{lem}\label{lemma: Gluing equivariant line bundles}
With the same notation as above, there is a surjection from $S$ to the set of isomorphism classes of toric line bundles on $X$.  

The image of $\{u_\sigma\}$ under this map is the unique toric line bundle $L$ which has a family of sections $s_\sigma \in \Gamma(U_\sigma, L)$ such that
\[
s_\sigma = u_\sigma u_\tau^{-1} s_\tau \quad \forall \sigma, \tau
\]
and 
\[
g(s_\sigma)_x = u_\sigma(g) (s_\sigma)_{gx} \quad \forall g\in G(A), x\in X(A),
\]
and such that $s_\sigma$ is a basis over $U_\sigma$. 
\end{lem}
\begin{proof}
For each cone $\sigma$ with our fixed $u_\sigma \in S$, applying Proposition \ref{proposition: equivariant line bundles and characters}  gives rise to an equivariant line bundle $L_{\sigma}$ (whose underlying line bundle is trivial) on $U_{\sigma}$ together with a basis $s_\sigma$ of sections for the equivariant line bundle.

We construct a line bundle $L$ on $X$ by gluing the $L_{\sigma}$ together by identifying $s_\sigma\mid_{U_\tau}$ with $\phi_{\sigma\tau}s_\tau$ for each inclusion of cones $\tau\subseteq\sigma$.  $L_\sigma$ is then the restriction of $L$ to $U_\sigma$. To be precise, for each $U_{\sigma_i}$ and $U_{\sigma_j}$ and $\delta=\sigma_i \cap \sigma_j$, we consider the isomorphism 
\[
\psi_{ij}:L_{\sigma_i}\mid_{\delta} \to L_{\sigma_j}\mid_{\delta}
\]
given by 
\[
(s_{\sigma_i})\mid_{U_\delta} \mapsto u_{\sigma_i}^{-1}u_{\sigma_j}(s_{\sigma_j})\mid_{U_\delta}.
\]
We have to check the cocycle condition $\psi_{ik}=\psi_{jk}\circ \psi_{ij}$. But, on $U_\tau=U_{\sigma_i} \cap U_{\sigma_j} \cap U_{\sigma_k}$, $\psi_{ik}$ sends $(s_{\sigma_i})\mid_{U_\tau}$ to $u_{\sigma_i}^{-1}u_{\sigma_k}(s_{\sigma_k})\mid_{U_\tau}$, whereas we have
\[
\psi_{jk}(\psi_{ij}((s_{\sigma_i})\mid_{U_\tau}) =\psi_{jk}(u_{\sigma_i}^{-1}u_{\sigma_j}(s_{\sigma_j})\mid_{U_\tau}) =u_{\sigma_i}^{-1}u_{\sigma_j} u_{\sigma_j}^{-1}u_{\sigma_k}(s_{\sigma_k})\mid_{U_\tau}, 
\]
showing the cocycle condition. 

Now, we show that $L$ is a $G$-equivariant line bundle. Let $A$ be an $\mathbb{F}_1$-algebra. We view $u_\sigma$ as a morphism $G \to \text{GL}_1$. For the notational convenience, we simply write $u_\sigma$ instead of the following: $u_\sigma(A):G(A) \to \text{GL}_1(A)$. Let $g\in G(A)$ and $x\in U_\tau(A)$.  By Proposition \ref{proposition: equivariant line bundles and characters}, we know that the $G$-action on $L_\sigma$ satisfies the following
\begin{equation}
g (s_\sigma)_x = u_\sigma(g) (s_\sigma)_{gx}.
\end{equation}
On the other hand, the $G$-action on $L_\tau$ satisfies
\begin{equation}
g (\phi_{\sigma\tau}(x)(s_\tau)_x) = \phi_{\sigma\tau}(x) u_\tau(g) (s_{\tau})_{gx} = (\phi_{\sigma\tau}(x) u_\tau(g) \phi_{\sigma\tau}(gx)^{-1})(\phi_{\sigma\tau}(gx)(s_\tau)_{gx}).
\end{equation}
We claim 
\begin{equation}\label{eq: homo}
u_\sigma(g) = \phi_{\sigma\tau}(x) u_\tau(g) \phi_{\sigma\tau}(gx)^{-1}. 
\end{equation}
In fact, from Lemma \ref{lem: multiplicative lemma}, we have $\phi_{\sigma\tau}(gx)=\phi_{\sigma\tau}(g)\phi_{\sigma\tau}(x)$. Hence, \eqref{eq: homo} becomes the following
\[
u_\sigma(g) = \phi_{\sigma\tau}(x) u_\tau(g) \phi_{\sigma\tau}(g)^{-1}\phi_{\sigma\tau}(x)^{-1} \iff \phi_{\sigma\tau}(g)u_\sigma(g)=u_\tau(g),
\]
which is clear from the definition of $\phi_{\sigma\tau}$. 

Thus the $G$-action on $L_\tau$ satisfies 
\begin{equation}g (s_\sigma\mid_{U_\tau})_x = g (\phi_{\sigma\tau}(x)(s_\tau)_x) = (\phi_{\sigma\tau}(x) u_\tau(g) \phi_{\sigma\tau}(gx)^{-1}) (s_\sigma\mid_{U_\tau})_{gx} = u_\sigma(g) (s_\sigma\mid_{U_\tau})_{gx}
\end{equation}
and clearly the action on $(L_\sigma)_{U_\tau}$ satisfies $g (s_\sigma\mid_{U_\tau})_x = u_\sigma(g) (s_\sigma\mid_{U_\tau})_{gx}$ as well.  Thus the actions agree.

By Proposition \ref{proposition: gluing equivariant vector bundles} the actions on $L_\sigma$ glue together to make $L$ into an equivariant vector bundle.  For surjectivity, we observe that any equivariant vector bundle $L$ can be obtained by gluing together the vector bundles $L_{U_\sigma}$ and the actions on them.  The remaining claims are clear from examining the above construction.
\end{proof}

We are now in a position to give a classification of toric line bundles, which does not require knowledge of $\mathrm{Pic}(X)$.

\begin{rmk}
One can easily observe that $(\Lambda \cap \sigma^\vee)^\times = \Lambda \cap \sigma^\perp$. In particular, Proposition \ref{proposition: classification toric line bundles} and Corollary \ref{Corollary: Klyachko-like result for line bundles} are precisely the same as toric line bundles on affine toric varieties over a field. 
\end{rmk}

\begin{pro}\label{proposition: classification toric line bundles}
Let $X$ be a toric scheme over $\mathbb{F}_1$ with fan $\Delta$ and dual lattice $\Lambda$.  Then isomorphism classes of toric line bundles on $X$ are in one-to-one correspondence with families of elements $[u_\sigma]\in \Lambda / (\Lambda \cap \sigma^\vee)^\times$ indexed by cones, which satisfy the compatibility condition that for $\tau\subseteq\sigma$, $[u_\tau]$ is the image of $[u_\sigma]$ under the quotient map $\pi_{\sigma\tau}:\Lambda / (\Lambda \cap \sigma^\vee)^\times \rightarrow \Lambda / (\Lambda \cap \tau^\vee)^\times$.
\end{pro}
\begin{proof}
Given a family of $[u_\sigma]$ as above, we pick representatives, and observe they satisfy the compatibility condition of Lemma \ref{lemma: Gluing equivariant line bundles}.  So there is a unique toric line bundle $L$ which locally possesses a basis consisting of a section $s_\sigma \in \Gamma(U_\sigma, L)$ with the action given by 
\[
g(s_\sigma)_x = u_\sigma(g) (s_\sigma)_{gx}
\]
under the functor of points and with $s_\sigma = u_\sigma^{-1} u_\tau s_\tau$.  The desired bijective correspondence $f$ will send $\{ [u_\sigma] \}_{\sigma\in\Delta}\}$ to $L$.

First we check this is well-defined.  Choose representatives $u_\sigma$ of $[u_\sigma]$ and let $L$ and $s_\sigma$ be as above.  Let $v_{\sigma} \in [u_{\sigma_0}]$ be another representative.  Then, we have
\[
u_\sigma v_{\sigma}^{-1}\in (\Lambda \cap \sigma^\vee)^\times = \Gamma(U_\sigma, \mathcal{O}_X^\times).
\]
Thus it makes sense to define sections $t_\sigma = (u_\sigma v_{\sigma}^{-1}) s_\sigma$. We claim that for each $\sigma$, one has the following: 
\begin{equation}
g(t_\sigma)_x = v_{\sigma}(g) (t_\sigma)_{gx}.    
\end{equation}
In fact, with $\varphi_\sigma=u_\sigma v_\sigma^{-1} $, we have
\[
(t_\sigma)_x = \varphi_\sigma(x)(s_\sigma)_x.
\]
Hence, we have
\[
g(t_\sigma)_x = g (\varphi_\sigma(x)(s_\sigma)_x) = \varphi_\sigma(x) u_\sigma(g)(s_\sigma)_{gx} = \varphi_\sigma(g^{-1})\varphi_\sigma(gx)u_\sigma(g)(s_\sigma)_{gx}
\]
\[
=u_\sigma(g^{-1})v_\sigma^{-1}(g^{-1})u_\sigma(g)\varphi_\sigma(gx)(s_\sigma)_{gx}=v_\sigma(g)\varphi_\sigma(gx)(s_\sigma)_{gx}=v_\sigma(g)(t_\sigma)_{gx}.
\]
In addition, 
\begin{equation}
t_\sigma = (u_\sigma v_\sigma^{-1}) s_\sigma = (u_\sigma v_\sigma^{-1} )u_\sigma^{-1} u_\tau s_\tau = v_\sigma^{-1} u_\tau s_\tau = v_\sigma^{-1} v_\tau t_\tau
\end{equation}

Thus $L$ with the sections $t_\sigma$ satisfies the defining property of the line bundle $L$ that is the image of the family of $[v_\sigma]$ under $f$.  Since the families of $[u_\sigma]$ and $[v_\sigma]$ map the same place under $f$, $f$ is well-defined.

Surjectivity follows from the surjectivity in Lemma \ref{lemma: Gluing equivariant line bundles}.

Let $L, L'$ correspond to $\{ [u_\sigma] \}$ and $\{ [v_\sigma ] \}$ respectively.  To prove injectivity of $f$, we suppose $L\cong L'$.  Then $L_{U_\sigma} \cong L'_{U_\sigma}$ for all cones $\sigma$.  Any line bundle on $U_\sigma$ is trivial, and $L_{U_\sigma}$ has a basis $s_\sigma$ satisfying $g(s_\sigma)_x = u_\sigma(g) (s_\sigma)_{gx}$.  Thus it corresponds under Proposition~\ref{proposition: equivariant line bundles and characters} and Lemma~\ref{lemma: characters as elements of the dual lattice} to $u_\sigma \in \Lambda$.  The same holds for $L'_{U_\sigma}$ and $v_\sigma$.  By Proposition \ref{proposition: dual lattice points corresponding to the trivial equivariant line bundle}, $u_\sigma \cong v_\sigma^{-1} \;\mathrm{mod}\, (\Lambda \cap \sigma^\vee)^\times$.  So $[u_\sigma] = [v_\sigma]$, establishing injectivity.
\end{proof}

\begin{lem}\label{lemma: equivalent congruent}
Let $\Delta$ be a fan and $\Lambda$ be the dual lattice.  For each cone $\sigma\in\Delta$, let $u_\sigma\in\Lambda$.  The following are equivalent:
\begin{enumerate}
\item 
For each inclusion of cones $\tau\subseteq\sigma$, $u_\sigma$ is congruent to $u_\tau$ modulo $(\Lambda \cap \tau^\vee)^\times$.
\item 
For each inclusion of cones $\tau\subseteq\sigma$ and each $x\in \tau$, $\langle u_\sigma, x\rangle = \langle u_\tau, x\rangle$.
\item 
For each cone $\sigma\in\Delta$ and each ray $\rho\in\Delta$ with $\rho\subseteq\sigma$, $\langle u_\sigma, x\rangle = \langle u_\rho, x\rangle$ for all $x \in \rho$.
\end{enumerate}
\end{lem}
\begin{proof}
For the proof, we will use the additive notation for $\Lambda$.

Suppose $u_\sigma$ is congruent to $u_\tau$ modulo $(\Lambda \cap \tau^\vee)^\times$.  Then $u_\sigma - u_\tau\in \tau^\vee$ as is $u_\tau - u_\sigma$.  So for any $x\in \tau$, we have both $\langle u_\sigma - u_\tau, x\rangle \geq 0$ and the reverse inequality.  Thus $\langle u_\sigma, x\rangle = \langle u_\tau, x\rangle$.

Conversely suppose $\langle u_\sigma, x\rangle = \langle u_\tau, x\rangle$ for all $x\in\tau$.  Then $\langle u_\sigma - u_\tau, x\rangle = 0$, so $(u_\sigma - u_\tau), (u_\tau - u_\sigma) \in \tau^\vee$.  Since both also belong to $\Lambda$, $u_\sigma - u_\tau \in (\Lambda \cap \tau^\vee)^\times$.

The second condition trivially implies the third, by taking $\tau$ to be a ray and $x$ to be the spanning vector of the ray.  Conversely suppose that for each cone $\sigma$ and each ray $\rho\subseteq\sigma$, we have $\langle u_\sigma, x\rangle = \langle u_\rho, x\rangle$ for all $x \in \rho$.  Let $\tau\in\Delta$ be such that $\tau\subseteq\sigma$.  For any ray $\rho$ contained in $\tau$ and $x \in \rho$, we have $\langle u_\sigma, x\rangle = \langle u_\rho, x\rangle = \langle u_\tau, x\rangle$.  Thus $u_\sigma$ and $u_\tau$ have the same inner product with any vector in the span of such rays.  But each cone in $\Delta$ is the span of the rays contained within it, so $\langle u_\sigma, x\rangle = \langle u_\tau, x\rangle$ for all $x\in\tau$.
\end{proof}

We can recast the classification of line bundles in a form that looks more like Klyachko's classification.

\begin{cor}\label{corollary: equivariant picard}
Let $X$ be a toric scheme over $\mathbb{F}_1$ with fan $\Delta$ and dual lattice $\Lambda$.  Then isomorphism classes of toric line bundles on $X$ are in one-to-one correspondence with families of elements $[u_\sigma]\in \Lambda / (\Lambda \cap \sigma^\vee)^\times$ indexed by cones, which satisfy the compatibility condition that for every ray $\rho\in\Delta$ that is a face of $\sigma$, $\langle u_\sigma, x\rangle$ depends only on $\rho$ and not on $\sigma$ for all $x \in \rho$.
\end{cor}

\begin{cor}\label{Corollary: Klyachko-like result for line bundles}
Let $X_R$ be a toric scheme over an idempotent semifield $R$ with fan $\Delta$ and dual lattice $\Lambda$.  Then isomorphism classes of toric line bundles on $X_R$ are in one-to-one correspondence with families of elements $[u_\sigma]\in \Lambda / (\Lambda \cap \sigma^\vee)^\times$ indexed by cones, which satisfy the compatibility condition that for every ray $\rho\in\Delta$ that is a face of $\sigma$, $\langle u_\sigma, x \rangle$ depends only on $\rho$ and not on $\sigma$ for all $x \in \rho$.
\end{cor}
\begin{proof}
One may check that Lemma \ref{lemma: Gluing equivariant line bundles} and Proposition \ref{proposition: classification toric line bundles} are still valid with $X_R=X\otimes R$, where $X$ is the toric scheme over $\mathbb{F}_1$ associated to $\Delta$. 
\end{proof}

\begin{mydef}\label{definition: Klyachko family}
A \emph{Klyachko family} is a family of elements $[u_\sigma]\in \Lambda / (\Lambda \cap \sigma^\vee)^\times$ indexed by cones, which satisfy the compatibility condition that for every ray $\rho\in\Delta$ that is a face of $\sigma$, $\langle u_\sigma, x \rangle$ for $x \in \rho$ depends only on $\rho$ and not on $\sigma$. 
\end{mydef}

Since equivariant vector bundles decompose uniquely as direct sums of equivariant line bundles, we obtain a classification of equivariant vector bundles.

\begin{cor}
There is a one-to-one correspondence between isomorphism classes of equivariant vector bundles and $S_n$-orbits of n-tuples of Klyachko families. 
\end{cor}

It remains to recast this in a way that looks more like Klyachko's theorem.

\begin{mydef}A filtration $E_i\subseteq M$ on a module $M$ over a semiring $R$ indexed by $\mathbb{Z}$ is called \emph{exhaustive} if 
\[
\bigcap_{i\in\mathbb{Z}} E_i = 0 \textrm{ and } \bigcup_{i\in\mathbb{Z}} E_i = M.
\]
\end{mydef}

\begin{mydef}\label{definition: Delta space}
Fix fan $\Delta$ and dual lattice $\Lambda$. An $n$-dimensional \emph{$\Delta$-Klyachko space} over an idempotent semifield $K$ is a free $K$-module $E$ of rank $n$ with collections of decreasing exhaustive filtrations $\{E^\rho(n)\}$ indexed by the rays of $\Delta$, satisfying the following compatibility condition: for each cone $\sigma \in \Delta$, there exist a decomposition $E=\bigoplus \limits_{[u] \in \Lambda / (\Lambda \cap \sigma^\vee)^\times} E_{[u]}$ such that
\[
E^\rho(i) = \sum_{\left<[u],v_\rho\right> \geq i}E_{[u]},
\]
for every $\rho \preceq \sigma$ and $i \in \mathbb{Z}$, where $v_\rho$ is the primitive generator of $\rho$.
\end{mydef}

Let $E$ and $F$ be $\Delta$-Klyachko spaces over an idempotent semifield $K$. By a morphism from $E$ to $F$, we mean a $K$-linear map $\phi:E \to F$ such that $\phi(E^\rho(i)) \subseteq F^\rho(i)$ for every ray $\rho \in \Delta$ and $i \in \mathbb{Z}$.

\begin{lem}
Let $E$ and $F$ be $\Delta$-Klyachko spaces over an idempotent semifield $K$.  Let $\phi: E \rightarrow F$ be an isomorphism of $\Delta$-Klyachko spaces.  Let $[u]\in \Lambda / (\Lambda \cap \sigma^\vee)^\times$ for some cone $\sigma$.  Then $\phi(E_{[u]}) = F_{[u]}$.
\end{lem}
\begin{proof}
Since $F = \bigoplus\limits_{[v]\in \Lambda / (\Lambda \cap \sigma^\vee)^\times} F_{[v]}$, and since direct summands of free modules over $K$ are free, there is a basis of $F$ such that every element lies in $F_{[v]}$ for some $v$ (the basis is a disjoint union of the bases of the summands).  Since the basis is unique up to rescaling and permutation, this is in fact true for any basis.

Fix $[u]\in \Lambda / (\Lambda \cap \sigma^\vee)^\times$.  Similar to the above observation $E_{[u]}$ is free with a basis that is a subset of some basis of $E$.  Let $e\in E_{[u]}$ be an element of such a basis.  Then $\phi(e)$ belongs to some basis of $F$ so $\phi(e)\in F_{[v]}$ for some $v$.  It remains to show $[v] = [u]$.

For any ray $\rho$ that is a face of $\sigma$, the minimum $i$ such that $e \in E^\rho(i)$ is $\langle u, v_\rho \rangle$, where $v_\rho$ is the primitive generator of $\rho$. Because $\phi$ is an isomorphism, this must be identical to the minimum $i$ such that $\phi(e) \in F^{\rho}(i)$, which is $\langle v, v_\rho \rangle$.  Thus we have $\langle u, v_\rho \rangle = \langle v, v_\rho \rangle$, i.e $\langle u - v, v_\rho \rangle = 0$.  For any $w\in \sigma$, $w$ lies in the subspace spanned by such rays, so $\langle u - v, w\rangle = 0$.  In particular because $u - v$ and $v - u$ have nonnegative inner product with any $w\in \sigma$, and since they are elements of $\Lambda$, $u - v \in (\Lambda \cap \sigma^\vee)^\times$.  So $[u] = [v]$ and hence $\phi(E_{[u]})\subseteq F_{[u]}$.  The reverse inclusion follows by applying this result to $\phi^{-1}$.
\end{proof}

\begin{lem}\label{lemma: correspondence family space}
There is a one-to-one correspondence between isomorphism classes of $n$-dimensional $\Delta$-Klyachko spaces over an idempotent semifield $K$ and $S_n$-orbits of $n$-tuples of Klyachko families.
\end{lem}
\begin{proof}
Suppose we are given an $n$-tuple of Klyachko families, the $k$-th of which is denoted $[u_{\sigma, k}] \in \Lambda / (\Lambda \cap \sigma^\vee)^\times$.  For each ray $\rho$, we let
\[
i_{\rho, k} = \langle[u_{\sigma,k}],v_\rho\rangle,
\]
where $\sigma$ is some cone containing $\rho$ as a face and $v_\rho$ is the primitive generator of $\rho$.  By the compatibility condition in the definition of Klyachko families, the choice of $\sigma$ is irrelevant.  
Now let $E = K^n$. Let $E^\rho(i)$ be spanned by the standard basis vectors $e_k$ such that $i \geq i_{\rho, k}$.  One can easily see that this is an exhaustive filtration.  For a cone $\sigma$, and an element $[u]\in \Lambda / (\Lambda \cap \sigma^\vee)^\times$, let
\[
E_{[u]}= \text{span}\{e_k \mid [u] = [u_{\sigma, k}]\}.
\]
Clearly for each $k$, this latter condition is true for exactly one $[u]$, so the $E_{[u]}$ yield a partition of the standard basis and hence
\begin{equation}\label{eq: decomp}
E=\bigoplus \limits_{[u] \in \Lambda / (\Lambda \cap \sigma^\vee)^\times} E_{[u]}.
\end{equation}
For the compatibility condition, recall that
\begin{equation}\label{eq: rho}
E^\rho(i) = \text{span}\{e_k \mid i \geq \langle[u_{\sigma,k}],v_\rho\rangle \}. 
\end{equation}
This basis is the disjoint union over $[u] \in \Lambda / (\Lambda \cap \sigma^\vee)^\times$ satisfying $i \geq \langle[u],v_\rho\rangle$ of the set of $e_k$ such that $[u_{\sigma, k}] = [u]$ or equivalently such that $e_k$ is an element of the basis of $E_{[u]}$.  Since the basis of $E^\rho(i)$ is the disjoint union of the bases of such $E_{[u]}$, $E^\rho(i)$ is the direct sum of such modules, i.e. 
\begin{equation} \label{eq: isotypical}
E^\rho(i) = \sum_{\left<[u],v_\rho\right> \geq i}E_{[u]}.
\end{equation}
If we started with a different $n$-tuple of Klyachko families that belongs to this $S_n$-orbit, they are related by some permutation $p \in S_n$, and it is clear that $e_k \mapsto e_{p(k)}$ defines an isomorphism between the resulting $\Delta$-Klyachko spaces.  So we have a well-defined map from $S_n$-orbits of $n$-tuples of Klyachko families to isomorphism classes of $\Delta$-Klyachko spaces.

For surjectivity, let $E$ be an $n$-dimensional $\Delta$-Klyachko space.  Write $E$ as a direct sum of $1$-dimensional free modules $E = L_1 \oplus\ldots\oplus L_n$.  Each $L_i$ is spanned by a single basis vector of $E$, and by Lemma~\ref{lemma: direct summand of free module} any direct summand of $E$ is spanned by a subset of the basis.  So in the direct sum decomposition $E=\bigoplus \limits_{[u] \in \Lambda / (\Lambda \cap \sigma^\vee)^\times} E_{[u]}$ corresponding to a cone $\sigma$, exactly one $E_{[u]}$ contains $L_k$.  Call this $[u_{\sigma, k}]$. 

We claim that $\{[u_{\sigma,k}]\}_{\sigma \in \Delta,~k=1,\dots,n}$ is an $n$-tuple of Klyachko family which maps to $E$ under the above construction. In fact, one can easily see that $\{[u_{\sigma,k}]\}_{\sigma \in \Delta,~k=1,\dots,n}$ corresponds to $E$ from the above construction, and hence we only have to check that $\langle [u_{\sigma,k}],v_\rho\rangle$ for a ray $\rho \preceq \sigma$ only depends on $\rho$ (not $\sigma$). To see this, observe that since $E^\rho(i)$ is a direct sum of isotypical components as in \eqref{eq: isotypical}, we have the following:
\begin{equation}\label{eq: iff}
L_k \subseteq E^\rho(i) \iff E_{[u_{\sigma,k}]} \subseteq E^\rho(i) \iff \langle[u_{\sigma,k}],v_\rho\rangle \geq i. 
\end{equation}
Since the condition $L_k \subseteq E^\rho(i)$ does not depend on a choice of $\sigma$, the condition $\langle [u_{\sigma,k}],v_\rho\rangle \geq i$ does not depend on a choice of $\sigma$ either. It follows that the minimum $i$ for which \eqref{eq: iff} holds, which is $\langle [u_{\sigma,k}], v_\rho \rangle$, does not depend on $\sigma$. In particular, we have
\[
\langle [u_{\sigma,k}], v_\rho \rangle = \langle [u_{\rho,k}], v_\rho \rangle.
\]
This shows the compatibility condition, and hence $\{[u_{\sigma,k}]\}_{\sigma \in \Delta,~k=1,\dots,n}$ is an $n$-tuple of Klyachko family which maps to $E$ under the above construction.

Now, we prove injectivity. Suppose that we have two $n$-tuples of Klyachko families $\{[u_\sigma]\}$ and $\{[v_\sigma]\}$ which give rise to isomorphic $\Delta$-Klyachko spaces $E$ and $F$. In other words, we have an isomorphism $\varphi:E=K^n \to F=K^n$ of free $K$-modules such that 
\begin{equation}\label{eq: filtration}
\varphi(E^\rho(i)) = F^\rho(i)
\end{equation}
for every ray $\rho \in \Delta$ and $i \in \mathbb{Z}$. From Proposition~\ref{proposition: exact sequence R, GL, S_n}, $\varphi$ is obtained by permutation of standard basis vectors and their rescaling. Since permuting the standard basis vectors results in another representative in an $S_n$-orbit of  $n$-tuple of Klyachko families, we may assume that $\varphi$ is obtained by rescaling the standard basis vectors. 

Fix a cone $\sigma \in \Delta$. Let $\{[u_{\sigma,k}]\}_{k=1,\dots,n}$ and $\{[v_{\sigma,k}]\}_{k=1,\dots,n}$ be the elements in the Klyachko families indexed by the cone $\sigma$. Now, observe that $[u_{\sigma,k}]$ is $[u]$ in \eqref{eq: decomp} if $E_{[u]}$ contains $e_k$, which is true if and only if $F_{[u]}$ contains $a_ke_k$ for some $a_k \in K^*$. This in turn holds if and only if $e_k \in F[u]$, which is equivalent to $[u]=[v_{\sigma,k}]$. This proves injectivity. 
\end{proof}

We thus have proved the following version of tropical Klyachko theorem from Corollary \ref{corollary: equivariant picard} and Lemma \ref{lemma: correspondence family space}.

\begin{mytheorem}\label{theorem: tropical Klyachko}
Let $K$ be an idempotent semifield. Let $X_K$ be the toric scheme over $K$ associated to a fain $\Delta$. The set of isomorphism classes of toric vector bundles on $X_K$ is in one-to-one correspondence with the set of isomorphism classes of $\Delta$-Klyachko spaces over $K$.
\end{mytheorem}

\begin{rmk}\label{remark: valuated matroids}
A $\Delta$-Klyachko space over $\TT$ naturally determines a tropical toric reflexive sheaf as in \cite[Definition 1.1 (Definition 5.1)]{khan2024tropical} as follows; Let $E=\mathbb{T}^n$ be a $\Delta$-Klyachko space of rank $n$, for some fan $\Delta$. $E$ gives rise to a simple valuated matroid $\mathcal{M}$ with ground set $\{1,2,\dots,n\}$. For each ray $\rho \in \Delta$, we have the collection $\{E^\rho(i)\}_{i \in \mathbb{Z}}$. Now, from \eqref{eq: rho}, we can confirm that each $E^\rho(i)$ is a flat of the underlying matroid $\underline{\mathcal{M}}$; note that $\underline{\mathcal{M}}$ is the free matroid of rank $n$ on $\{1,\dots,n\}$. The conditions that 
\begin{enumerate}
    \item 
$E^\rho(j)\leq E^\rho(i)$ if $j >i$;
\item 
$E^\rho(i) = \emptyset$ for $i \gg 0$;
\item 
$E^\rho(i) = \{1,\dots,n\}$ for $i \ll 0$;
\end{enumerate}
are clear from the definition (with the order reversed). Now the definition of tropical toric vector bundle \cite[Definition 5.4]{khan2024tropical} is a tropical toric reflexive sheaf together with a compatibility condition which is precisely \eqref{eq: rho}. 
\end{rmk}

\appendix

\section{Toric varieties, toric monoid schemes, and toric schemes over $\mathbb{T}$}\label{section: Toric varieties, toric monoid schemes, and tropical toric schemes}

We recall how toric varieties, toric monoid schemes and toric schemes over $\mathbb{T}$ spaces are related and what the ``torus action'' is in the latter cases. We proceed to define toric vector bundles in each case. This section is largely expository. 

We start with the connection between toric monoid schemes and toric varieties. This correspondence has been observed in \cite{Kato94}, \cite{deitmar2008f1} and \cite{cortinas2015toric}. The following theorem is well known. \footnote{For instance, see \cite[Example 5.8]{cortinas2015toric}.} 

We note that each toric monoid scheme has a unique generic point (\cite[Lemma 2.3 and Theorem 4.4]{cortinas2015toric}). 

\begin{mythm}\label{thm:F1-to-C-toric}
Let $X$ be a toric monoid scheme and $k$ a field. Then $X_k$ is a toric variety.
\end{mythm}

\begin{rmk} 
Theorem~\ref{thm:F1-to-C-toric} also holds when $X_k$ is reducible. Then the group $G \cong \mathbb{Z}^r \times F$ for a finite Abelian group $F$. 
We denote by $\mathcal{G} = \Spec G$ the torus. It is easy to see that $\mathcal{G} = \Spec G$ is dense in $X$ and note that $\mathcal{G}(k) \cong (k^{\times})^{r} \times F$. Since the torus $\mathcal{G}_k$ is dense in $X_k$, it meets every irreducible component of $X_k$, and $F$ permutes the irreducible components. 
\end{rmk}

In order to understand the torus action $\mathcal G$ on $X$ it is enough to do so in the case when $X$ is affine, i.e. $X=\Spec A$ for a toric monoid $A$. Let $k$ be a field and let $X_{k} = \Spec k[A]$ and $\mathcal{G}_k = \Spec k[G]$. Here $X_k$ is a toric variety with torus $\mathcal{G}_k$. By definition, the action of the (dense) torus on itself extends to an action on $X_k$. The following proposition explains the relation between the points of the toric variety $X_k$ and the monoid scheme $X$ and the actions of the corresponding tori on them.

\begin{pro} \label{proposition: constant torus action}
There is an injection from the points of the toric monoid scheme $X = \Spec A$ to the points of the affine toric variety $X_k = \Spec k[A]$. The coaction of the torus on $k[A]$ gives rise to a coaction on $A$, which is constant on the points of $X$. In particular, there is a one to one correspondence between the points of $X$ and the orbits of the torus action.
\end{pro}

\begin{proof}
Consider the map $\phi: A \rightarrow k[A]$ sending a monomial $g \in A$ to $g \in k[A]$. Let $\mathfrak{p} \subset A$ be a prime ideal, then let $\mathfrak{q}$ be the monomial ideal generated by $\phi(\mathfrak{p})$. The ideal $\mathfrak{q}$ is prime. To see this argue by contradiction, assume $g,h \not\in \mathfrak{q}$ but $gh \in \mathfrak{q}$. Write $g$ and $h$ as sums of monomials $g = \sum_{i=1}^{k} m_i$ and $h = \sum_{i=1}^{s} n_i$, then $gh = \sum_{i,j} m_i n_j$. Since $\mathfrak{q}$ is a monomial ideal, if $f \in \mathfrak{q}$ is a polynomial, then each monomial of $f$ are in $\mathfrak{q}$. Thus every monomial $m_i n_j$ of $gh$ is in $\mathfrak{q}$. Since $\mathfrak{p}$ is prime it follows that either $m_i$ or $n_j$ are in $\mathfrak{p}$ and $\mathfrak{q}$ for every pair $i,j$. Thus we conclude that $\mathfrak{q}$ is prime. \par\smallskip
Thus the map $\phi$ is injective and sends prime ideals in $A$ to monomial prime ideals in $k[A]$ which are in turn in one to one correspondence with orbits of the torus action on the toric variety. The second part follows for the well-known lemma below.
\end{proof}

This lemma is a summary of the construction from \cite[Section 5]{CLS}.
\begin{lem}
Let $k$ be a field, then the ideal $\mathfrak{a} \subseteq k[A]$ corresponds to a torus invariant irreducible subvariety if and only if $\mathfrak{a}$ is a monomial ideal.
\end{lem}

Now we look closer at toric schemes as defined in Definition \ref{definition: tropical toric schemes}. Toric schemes arise from toric monoid schemes which in turn correspond to fans as shown in \cite{cortinas2015toric}. In particular, if $\Delta$ is a fan, then a toric scheme over a semiring $R$ is covered by open sets isomorphic to $\Spec R[S_\sigma]$, where  $\sigma$ are the cones of $\Delta$ and $S_\sigma:=\sigma^\vee \cap \Lambda$.

\begin{mytheorem}\label{thm:F1-to-T-toric}
Let $X$ be a toric monoid scheme. Via base change to an idempotent semifield $K$ we obtain a toric scheme $X_K$ over $K$ with dense torus $\mathcal{G}_K$ whose action on itself extends to an action on $X_K$. 
\end{mytheorem}

\begin{proof}
Let $X$ be toric monoid scheme, and let $U = \Spec M$ be an affine open subset of $X$, where $M$ is a cancellative monoid. Each affine open set of $X$ has a unique generic point $\eta$. Let $G$ be $\Frac(M)=\mathcal{O}_{\Spec M, \eta}$. Since $G$ is also the stalk $\mathcal{O}_{X,\eta}$, $G$ does not depend on the choice of $U$. Since $M$ is cancellative, the quotient map $\varphi : M \rightarrow G$ is injective and induces the maps $\varphi_K : K[M] \rightarrow K[G]$ and 
\begin{equation}\label{eq: torus action for fields}
K[M] \rightarrow K[G]\otimes K[M], \quad a\mapsto \varphi_K(a)\otimes a.
\end{equation}
The last map induces an action of the algebraic group $\mathcal{G}_K = \Spec\ K[G]$ on $\Spec\  K[M]$ which extends to an action of $\mathcal{G}_K$ on $X_{K}$ (as this is compatible with the restriction maps of the structure sheaf). 
To see that that $\mathcal{G}_K$ is dense in $U_K$, and hence it is dense in $X_K$ we refer to the next lemma.
\end{proof}

\begin{lem}
Let $K$ be an idempotent semifield.  Let $X$ be a toric monoid scheme with a dense torus $\mathcal{G}$. Then $\mathcal{G}_K$ is dense in $X_K$.
\end{lem}
\begin{proof}
	We may assume that $X$ is affine. Since $X$ is a toric monoid scheme, $X$ is irreducible. In particular, $X_K$ is irreducible.\footnote{See the proof of Proposition \ref{theorem: bundles stable under scalar extension} in \cite{JMT20}.} On the other hand, since $\mathcal{G}$ is dense in $X$, $\mathcal{G}$ contains a generic point of $X$. It follows that $\mathcal{G}_K$ contains a generic point of $X_K$, and hence is dense in $X_K$. 
\end{proof}

When toric schemes are defined over the tropical semifield $\mathbb{T}$, the relationship between toric monoid schemes and toric schemes over $\mathbb{T}$ in the affine case is depicted in the following diagrams in a view of tropicalization. 

\begin{center}
    \begin{tikzcd} 
    S_\sigma = M \arrow[r,"\otimes k"]\arrow[dr,"\otimes \T"]&k[M]\arrow[d, "\text{trop}"]&&
    X\arrow[r,""]\arrow[dr," "]&X_k(\sigma)\arrow[d, "\text{scheme trop}"]\\& \T[M] && &X_\T(\sigma)
    \end{tikzcd}
\end{center}

The vertical map on the right is the scheme-theoretic tropicalization of \cite{giansiracusa2016equations}. 
Whether we are working over a field $k$ or the tropical semifield $\T$, the torus action of $\mathcal{G}_k$ or $\mathcal{G}_\mathbb{T}$ arises from the monoid scheme as explained in \cite[Construction 4.2 and Example 5.8]{cortinas2015toric} and the proof of 
Theorem~\ref{thm:F1-to-T-toric}. Therefore, the action of the torus commutes with tropicalization on affine open sets. But since localization also commutes with tropicalization (we are localizing at a monomial) the affine pictures glue over the fan. In particular, we have the following commutative diagram.

\begin{center}
    \begin{tikzcd} 
    \mathcal{G}_k \times_k X_k(\Delta) \arrow[r, "\text{trop}"]\arrow[d, " "]& \mathcal{G}_\mathbb{T} \times_{\T} X_\T(\Delta)\arrow[d, " "] \\
    X_k(\Delta)\arrow[r, "\text{trop}" ]& X_\T(\Delta)
    \end{tikzcd}
\end{center}

\section{$S_n$ as a functor} \label{sn as a functor}

In this section, we prove that $S_n$ can be seen as a group scheme defined over $\mathbb{N}$.

Let $n$ be a positive integer and $S_n$ be the symmetric group on $n$ letters. We define the following semiring $R_n$: 
\begin{equation}\label{eq: R_n}
R_n=\frac{\mathbb{N}[e_\sigma]_{\sigma \in S_n}}{\angles{\{e_\sigma \cdot e_\tau = 0\}_{\sigma \neq \tau},\sum_{\sigma \in S_n} e_\sigma =1}}.
\end{equation}
In other words, $R_n$ is generated by elements $\{e_{\sigma}\}_{\sigma \in S_n}$ with relations $e_\sigma \cdot e_\tau =0$ for $\sigma \neq \tau$ and $\sum\limits_{\sigma \in S_n} e_\sigma =1$. Note that with the relations given, one has that
\[
e_{\sigma}^2 = e_\sigma, \quad \forall \sigma \in S_n.
\]

We define a comultiplication on $R_n$ as follows:
\begin{equation}\label{eq: comultiplication}
\Delta(e_\sigma) = \sum_\tau e_\tau \otimes e_{\tau^{-1}\sigma}. 
\end{equation}
The counit is defined as follows:
\begin{equation}\label{eq: counit}
\varepsilon: R_n \to \mathbb{N}, \quad e_\sigma \mapsto \begin{cases}
    1 & \text{if } \sigma= \text{id}, \\
    0 & \text{otherwise.}
\end{cases}
\end{equation}
The antipode is defined as follows:
\begin{equation}
S: R_n \to R_n, \quad e_\sigma \mapsto e_{\sigma^{-1}}. 
\end{equation}

\begin{lem}
With the same notation as above, for any semiring $A$, $R_n(A):=\Hom(R_n,A)$ is equipped with a natural group structure. 
\end{lem}
\begin{proof}
For $f, g \in R_n(A)$, the group multiplication is defined as follows:
\begin{equation}\label{eq: group mul}
f*g :=\mu \circ (f\otimes g )\otimes \Delta, 
\end{equation}
where $\mu:R_n \otimes_\mathbb{N} R_n \to R_n$ is the multiplication. 

The inverse is given as follows:
\begin{equation}\label{eq: group inv}
f^{-1}=f\circ S. 
\end{equation}

The unit is given as follows:
\begin{equation}\label{eq: group id}
e=i\circ \varepsilon, \quad i:\mathbb{N} \to A. 
\end{equation}

Now, we check the group axioms. To show that $e$ is the identity, we only have to check that $e*f(e_\sigma)=f*e(e_\sigma)=f(e_\sigma)$ as follows:
\[
f*e(e_\sigma) = \mu\circ (f\otimes e)\circ\Delta(e_\sigma)=\mu\circ (f\otimes e)(\sum_\tau e_\tau \otimes e_{\tau^{-1}\sigma}) 
\]
\[= \sum_\tau f(e_\tau)e(e_{\tau^{-1}\sigma})=f(e_\sigma)e(1)=f(e_\sigma). 
\]
Likewise, 
\[
e*f(e_\sigma) = \mu\circ (e\otimes f)\circ \Delta(e_\sigma)=\sum_\tau e(e_\tau)f(e_{\tau^{-1}\sigma}) = e(1)f(e_\sigma)=f(e_\sigma). 
\]

For the existence of inverses, we check $f*f^{-1}=e$ as follows:
\[
f*f^{-1}(e_\sigma)=\sum_\tau f(e_\tau)f(S(e_{\tau^{-1}\sigma}))=\sum_\tau f(e_\tau)f(e_{\sigma^{-1}\tau}) 
\]
\[
=\sum_\tau f(e_\tau e_{\sigma^{-1}\tau}) = \begin{cases}
 \sum_\tau f(e_\tau)=1 & \text{if } \sigma=\text{id}, \\
 0 & \text{otherwise.}
\end{cases}
\]
In particular, $f*f^{-1}(e_\sigma)=e(e_\sigma)$. Similarly, $f^{-1}*f=e$. 

Finally, we check the associativity: let $f,g,h \in R_n(A)$:, 
\[
((f*g)*h)(e_\sigma)=\sum_\tau (f*g)(e_\tau)h(e_{\tau^{-1}\sigma}) = \sum_\tau \sum_{\delta} f(e_\delta)g(e_{\delta^{-1}\tau}) h(e_{\tau^{-1}\sigma})
\]
On the other hand:
\[
(f*(g*h))(e_\sigma)=\sum_{\alpha} f(e_\alpha) (g*h)(e_{\alpha^{-1}\sigma}) = \sum_\alpha \sum_\beta f(e_\alpha)g(e_\beta)h(e_{\beta^{-1}\alpha^{-1}\sigma} ), 
\]
and one can easily see that the two expressions are identical with $\delta = \alpha$ and $\tau = \delta \beta$. 
\end{proof}

Define an element $f_\sigma$ in $R_n$ by 
\[
f_\sigma := \sum_{\tau \neq \sigma} e_\tau.
\] 
Observe that for each $\sigma$, $(e_\sigma, f_\sigma)$ is an idempotent pair. For instance, if $n=2$, this is equivalent to the notion of idempotent pairs as in Definition \ref{definition: idempotent pair}, i.e., 
\[
e_\sigma+f_\sigma=1, \quad e_\sigma f_\sigma =0. 
\]
In fact, one can observe that $R_n$ is just a direct sum of copies of $\mathbb{N}$, labelled by elements of $S_n$.  

Let $S_n(A)$ be the group of functions from the set of connected components of $\Spec A$ to $S_n$. Then, we have the following.

\begin{lem}
With the same notation as above, $\Hom(\Spec A, \Spec R_n)=R_n(A)$ is isomorphic to $S_n$ as a group if $\Spec A$ is connected. 
\end{lem}
\begin{proof}
Let $A$ be a semiring. A map from $\Spec A$ to $\Spec R_n$ is equivalent to a homomorphism $\varphi:R_n \to A$, which is determined by the elements $\varphi(e_\sigma)$, for $e_\sigma \in R_n$. Since $\Spec A$ is connected, it follows from Lemma~\ref{proposition: connected implies no idempotet paris} that $A$ has only trivial idempotent pairs. In other words, if $(e,f)$ is an idempotent pair of $A$, then $\{e,f\} = \{0,1\}$. 

We claim that each $\varphi \in R_n(A)$ determines a unique element $\sigma \in S_n$, giving us a bijection between $R_n(A)$ and $S_n$. In fact, for each $\sigma \in S_n$, the pair $(\varphi(e_\sigma),\varphi(f_\sigma))$ should be a trivial idempotent pair, i.e., $\varphi(e_\sigma)$ is either $0$ or $1$ for each $\sigma \in S_n$. It is easy to see from the identity $\sum_{\sigma \in S_n} e_\sigma =1$ that $\varphi(e_\sigma)=1$ for at least one choice of $\sigma$, and from the identities $e_\sigma\cdot e_\tau=0$ that $\varphi(e_\sigma)$ cannot be equal to $1$ for two choices of $\sigma$. So we have obtained an element of $S_n$, specifically the unique element such that $\varphi(e_\sigma)=1$. 

On the other hand, for a given element $\sigma$ in $S_n$, note that there is a unique homomorphism $\varphi:R_n \to A$  with $\varphi(e_\sigma)=1$ and $\varphi(e_\tau)=0$ for all other $\tau$, since this choice of $\varphi(e_\tau)$ satisfies the relation defining $R_n$. This gives us a desired bijection between $R_n(A)$ and $S_n$. 

Finally, we need to check the comultiplication in \eqref{eq: comultiplication} is right one to give a canonical isomorphism $R_n(A) \simeq S_n$ as groups with the above bijection. But, it is a straightforward calculation that
\[
\mu\circ (\sigma\otimes \tau) \circ \Delta(x) = (\sigma\tau)(x),
\]
where we abuse notation by identifying elements of $S_n$ with elements of $R_n(A)$ via the bijection. Also, the unit $e \in R_n(A)$, as in \eqref{eq: group id}, maps to the unit in $S_n$ via the bijection. 
\end{proof}

Showing that this representing object has the correct $A$-valued points for disconnected $A$ should require an extra condition.  Even in the ring-theoretic case, one may need some finiteness condition in order to describe a scheme as the disjoint union of connected components. For example, this is not necessarily true for the spectrum of an infinite Boolean algebra, which is totally disconnected, but in general is not discrete.  

Suppose $\Spec A$ has finitely many connected components, but is not connected.  Then we can find a non-trivial idempotent pair $(e, f)$.  Using the idempotence of $e$, it is easy to see that $eA$ is a semiring (with identity element $e$), and similarly for $fA$.  Note that any $a$ in $A$ can be written as $ea + fa$. Furthermore, if $a = eb + fc$ then multiplying by $f$ gives $fa = fc$ and similarly $ea = eb$, so the decomposition as a sum of an element of $eA$ and an element of $fA$ is unique, so we get a bijection between $eA \times fA$ and $A$. One can check that this is a homomorphism. So $A$ is a product of nontrivial semirings, and $\Spec A$ is the coproduct $\Spec eA \bigsqcup \Spec fA$.  In particular if $S_n$ denotes the functor represented by $R_n$, then we have
\[
S_n(A) = S_n(eA) \times  S_n(fA). 
\]
As in the case for rings, the following holds. 
\begin{lem}\label{lemma: decomposition lemma}
With the same notation as above, $\Spec A$ is homeomorphic to $\Spec eA \bigsqcup \Spec fA$. \end{lem}

In particular, $\Spec eA$ and $\Spec fA$ have fewer connected components. Hence, by inductive hypothesis, elements of $S_n(eA)$ correspond to maps from the connected components of $e_A$ to the discrete set $S_n$, and similarly for $fA$. By the disjoint union property, a map from the connected components of $\Spec A$ to $S_n$ is the same as a pair of maps from the connected components of $\Spec eA$ and $\Spec fA$ to $S_n$, and such pairs correspond to elements of $S_n(eA) \times S_n(fA) = S_n(A)$. Hence we have the following. 

\begin{pro}\label{proposition: finitely many components}
 With the same notation as above, $R_n(A)$ is isomorphic to $S_n$ if $\Spec A$ has finitely many connected components. 
\end{pro}

\begin{rmk}\label{remark: exact sequence remark}
In Proposition~\ref{proposition: exact sequence R, GL, S_n}, we prove that there exists an exact sequence as follows, when $A$ is zero-sum free and has only trivial idempotent pairs:
\begin{equation}\label{eq: exact seq}
0 \to (A^\times)^n \to \text{GL}_n(A) \to S_n \to 0. 
\end{equation}
From Proposition \ref{proposition: finitely many components}, if $\Spec A$ is the coproduct of connected components, one can obtain an exact sequence \eqref{eq: exact seq} as a product of copies of the exact sequences in the connected case. A scheme-theoretic version of the above exact sequence is in \cite[Theorem 15.4]{BJ24}.
\end{rmk}

\bibliography{Vectorbundle}\bibliographystyle{alpha}

\end{document}